\documentclass[11pt, a4paper, oneside]{amsart}

\usepackage[english]{babel}
\usepackage{amsmath, amsthm, amsfonts, mathrsfs, amssymb}
\usepackage{mathtools}
\mathtoolsset{centercolon}
\usepackage{booktabs}
\usepackage[shortlabels]{enumitem}
\setlist[itemize]{leftmargin=20pt}
\usepackage[colorlinks, citecolor = blue]{hyperref}
\usepackage[hmargin=3cm,vmargin=3cm]{geometry}
\usepackage[color=green!40]{todonotes}
\usepackage{comment}

\usepackage{color}
\usepackage{graphicx}
\usepackage{tikz-cd}
\usetikzlibrary{arrows}

\makeatletter
\DeclareFontFamily{OMX}{MnSymbolE}{}
\DeclareSymbolFont{MnLargeSymbols}{OMX}{MnSymbolE}{m}{n}
\SetSymbolFont{MnLargeSymbols}{bold}{OMX}{MnSymbolE}{b}{n}
\DeclareFontShape{OMX}{MnSymbolE}{m}{n}{
    <-6>  MnSymbolE5
   <6-7>  MnSymbolE6
   <7-8>  MnSymbolE7
   <8-9>  MnSymbolE8
   <9-10> MnSymbolE9
  <10-12> MnSymbolE10
  <12->   MnSymbolE12
}{}
\DeclareFontShape{OMX}{MnSymbolE}{b}{n}{
    <-6>  MnSymbolE-Bold5
   <6-7>  MnSymbolE-Bold6
   <7-8>  MnSymbolE-Bold7
   <8-9>  MnSymbolE-Bold8
   <9-10> MnSymbolE-Bold9
  <10-12> MnSymbolE-Bold10
  <12->   MnSymbolE-Bold12
}{}

\let\llangle\@undefined
\let\rrangle\@undefined
\DeclareMathDelimiter{\llangle}{\mathopen}%
                     {MnLargeSymbols}{'164}{MnLargeSymbols}{'164}
\DeclareMathDelimiter{\rrangle}{\mathclose}%
                     {MnLargeSymbols}{'171}{MnLargeSymbols}{'171}
\makeatother

\newcommand{\F}{\ensuremath{\mathbf{F}}}

\newcommand{\R}{\ensuremath{\mathbf{R}}}
\newcommand{\C}{\ensuremath{\mathbf{C}}}

\newcommand{\mb}{\mathbf}
\newcommand{\mc}{\mathcal}

\DeclareMathOperator{\supp}{supp}
\DeclareMathOperator{\ind}{\mathbf{1}}

\renewcommand{\emptyset}{\varnothing}
\def\avint_#1{\mathchoice{\mathop{\kern 0.2em\vrule width 0.6em height 0.69678ex depth -0.58065ex \kern -0.8em \intop}\nolimits_{\kern -0.4em#1}}{\mathop{\kern 0.1em\vrule width 0.5em height 0.69678ex depth -0.60387ex \kern -0.6em \intop}\nolimits_{#1}} {\mathop{\kern 0.1em\vrule width 0.5em height 0.69678ex depth -0.60387ex \kern -0.6em \intop}\nolimits_{#1}} {\mathop{\kern 0.1em\vrule width 0.5em height 0.69678ex depth -0.60387ex \kern -0.6em \intop}\nolimits_{#1}}}

\newtheorem{TheoremLetter}{Theorem}
{}
\newtheorem{theorem}{Theorem}
\newtheorem{corollary}[theorem]{Corollary}
\newtheorem{lemma}[theorem]{Lemma}
\newtheorem{proposition}[theorem]{Proposition}
\newtheorem*{proposition*}{Proposition}

{}

{}
\newtheorem{CorollaryLetter}[TheoremLetter]{Corollary}
{}

\theoremstyle{remark}
\newtheorem{remark}[theorem]{Remark}

\theoremstyle{definition}
\newtheorem{definition}[theorem]{Definition}
\newtheorem*{definition*}{Definition}

\numberwithin{theorem}{section}
\numberwithin{equation}{section}
\title{A lattice approach to matrix weights}
\author{Zoe Nieraeth}
\thanks{Z. N. is supported by the grant Juan de la Cierva formación 2021 FJC2021-046837-I, the Basque Government through the BERC 2022-2025 program, by the Spanish State Research Agency project PID2020-113156GB-I00/AEI/10.13039/501100011033 and through BCAM Severo Ochoa excellence accreditation SEV-2023-2026.}
\address{Zoe Nieraeth (she/her)\hfill\break\indent BCAM\textendash  Basque Center for Applied Mathematics, Bilbao, Spain}
\email{zoe.nieraeth@gmail.com}
\allowdisplaybreaks

\begin{document}
\begin{abstract}
In this paper we recontextualize the theory of matrix weights within the setting of Banach lattices. We define an intrinsic notion of directional Banach function spaces, generalizing matrix weighted Lebesgue spaces. Moreover, we prove an extrapolation theorem for these spaces based on the boundedness of the convex-set valued maximal operator. We also provide bounds and equivalences related to the convex body sparse operator. Furthermore, we introduce a weak-type analogue of directional Banach function spaces. In particular, we show that the weak-type boundedness of the set valued maximal operator on matrix weighted Lebesgue spaces is equivalent to the matrix Muckenhoupt condition, with equivalent constants. 

Finally, we apply our main results to matrix-weighted variable Lebesgue and Morrey spaces, obtaining new extrapolation results and characterizations extending the known ones of the scalar setting.
\end{abstract}




\maketitle

\subsection*{Notation}
Throughout this work, $d,n\geq 1$ are fixed integers denoting dimension. We let $\F$ denote either the field $\R$ or the field $\C$. For $u,v\in\F^n$ we write $u=(u_1,\ldots,u_n)$, $v=(v_1,\ldots,v_n)$, and we define the standard scalar product and norm
\[
u\cdot v:=\sum_{k=1}^n u_k\overline{v_k},\quad |u|:=|u\cdot u|^{\frac{1}{2}}.
\]
We write $A\lesssim B$ to mean that there is some constant $C$ for which $A\leq CB$. If the constant $C$ depends on parameters $\alpha_1,\alpha_2,\ldots$, then we write $A\lesssim_{\alpha_1,\alpha_2,\ldots} B$. Moreover, $A\gtrsim B$ and $A\gtrsim_{\alpha_1,\alpha_2,\ldots}B$ are defined similarly. We write $A\eqsim B$ when both $A\lesssim B$ and $A\gtrsim B$, and similarly with parameter subscripts.

\section{Introduction}

\subsection{Matrix weighted Lebesgue spaces} Given a Calder\'on-Zygmund operator $T$ and a space of functions $X$ modeling the initial data of a given partial differential equation, the question whether $T$ maps $X$ to itself is fundamental to the theory of singular integrals. This question has been fully answered when $X$ is a weighted Lebesgue space $L^p_w(\R^d)$. Here, $w$ is an a.e. positive measurable function on $\R^d$, and we say that $f\in L^p_w(\R^d)$ if $f$ is a measurable function on $\R^d$ satisfying
\[
\int_{\R^d}\!|wf|^p\,\mathrm{d}x<\infty.
\]
Then we have $T:L^p_w(\R^d)\to L^p_w(\R^d)$ precisely when $1<p<
\infty$, and $w$ satisfies the Muckenhoupt $A_p$ condition, i.e.,
\[
[w]_p:=\sup_Q\Big(\frac{1}{|Q|}\int_Q\!w^p\,\mathrm{d}x\Big)^{\frac{1}{p}}\Big(\frac{1}{|Q|}\int_Q\!w^{-p'}\,\mathrm{d}x\Big)^{\frac{1}{p'}}<\infty,
\]
see, e.g., \cite{Gr14a}. Here the supremum is taken over all cubes $Q$ in $\R^d$, and $p'$ is the H\"older conjugate of $p$.  Note that the way we have defined $[w]_p$ differs from the more common definition of the Muckenhoupt characteristic given by
\[
[w]_{A_p}:=\sup_Q\Big(\frac{1}{|Q|}\int_Q\!w\,\mathrm{d}x\Big)\Big(\frac{1}{|Q|}\int_Q\!w^{1-p'}\,\mathrm{d}x\Big)^{p-1}=[w^{\frac{1}{p}}]_p^p.
\]
This is due to the fact that we have introduced our weights using the multiplier approach instead of the change of measure approach. For a further elaboration on this choice, we refer the reader to \cite{LN23b}.

In the `90s, this theory has been extended to the setting of so-called matrix weights. We let $\F$ denote either the field $\R$ or $\C$. A matrix valued mapping $W:\R^d\to\F^{n\times n}$ is called a matrix weight if it is Hermitian and positive definite a.e. For an exponent $0<p\leq\infty$, we define $L^p_W(\R^d;\F^n)$ as the space of measurable functions $f:\R^d\to\F^n$ for which $|Wf|\in L^p(\R^d)$, i.e., for which
\[
\|f\|_{L^p_W(\R^d;\F^n)}:=\Big(\int_{\R^d}\!|W(x)f(x)|^p\,\mathrm{d}x\Big)^{\frac{1}{p}}<\infty,
\]
where the integral is replaced by an essential supremum when $p=\infty$. Rather than using the more conventional definition in terms of the matrix $W(x)^{\frac{1}{p}}$, we have opted to normalize it this way as a natural extension of the multiplier notation in the scalar-valued case; see also \cite{BC23} for an elaboration on this choice. For a linear operator $T$ acting on functions taking values in $\F$ and a measurable function $f=(f_1,\ldots,f_n):\R^d\to\F^n$, we define 
\[
\widetilde{T}f(x):=(Tf_1(x),\ldots,Tf_n(x)).
\]
When $T$ is a Calder\'on-Zygmund operator, finding an alternative to the Muckenhoupt condition for matrix weights was initiated by Nazarov, Treil, and Volberg in the `90s, see \cite{NT96, Tr89, TV97a, TV97b, Vo97}. In these works, they studied the boundedness of $\widetilde{H}$ on $L^p_W(\R;\C^n)$, where $H$ is the Hilbert transform. The theory was extended to general Calder\'on-Zygmund operators by Christ and Goldberg \cite{CG01}. Though usually written in a more involved way using function norms, the matrix $A_p$ condition can be formulated as $W^{-1}\in L^{p'}_{\text{loc}}(\R^d;\F^{n\times n})$, and there is a constant $C>0$ such that for every cube $Q$ and for every $u\in\F^n$ there is a non-zero $v\in\F^n$ such that
\begin{equation}\label{eq:matrixapintro1}
\Big(\frac{1}{|Q|}\int_Q\!|W(x)u|^p\,\mathrm{d}x\Big)^{\frac{1}{p}}\Big(\frac{1}{|Q|}\int_Q\!|W(x)^{-1}v|^{p'}\,\mathrm{d}x\Big)^{\frac{1}{p'}}\leq C|u\cdot v|,
\end{equation}
where $u\cdot v=\sum_{k=1}^nu_k\overline{v_k}$ is the standard scalar product in $\F^n$. We denote the optimal constant $C$ by $[W]_p$. 

In \cite{NPTV17} it was shown that Calder\'on-Zygmund operators satisfy a so-called convex body domination. More precisely, if $T$ is a Calder\'on-Zygmund operator, then there is a constant $C_T>0$ such that for every bounded function $f$ with bounded support there is a \emph{sparse} collection of cubes $\mc{S}$ for which for a.e. $x\in\R^d$ we have
\[
\widetilde{T}f(x)\in C_T\sum_{Q\in\mc{S}}\llangle f\rrangle_Q\ind_Q(x),
\]
where
\begin{equation}\label{eq:convexsetaverageintro}
\llangle f\rrangle_Q:=\Big\{\frac{1}{|Q|}\int_Q\!hf\,\mathrm{d}x: \|h\|_{L^\infty(\R^d)}\leq 1\Big\},
\end{equation}
where scalar multiples of a set are considered pointwise, and where the sum of the sets is interpreted as a Minkowski sum. Using this, they showed the bound
\begin{equation}\label{eq:a2matrixtheoremintro}
\|\widetilde{T}f\|_{L^2_W(\R^d;\F^n)}\lesssim [W]_2^3\|f\|_{L^2_W(\R^d;\F^n)}.
\end{equation}
for the Hilbert transform $T=H$, which was later generalized to all Calder\'on-Zygmnd operators by Culiuc, Di Plinio, and Ou in \cite{CDO18b}. When $n>1$, the third power dependence of $[W]_2$ was very recently shown to be sharp for the Hilbert transform by Domelevo, Petermichl, Treil, and Volberg in \cite{DPTV24}. This is in stark contrast with the $n=1$ case, where sparse domination yields the sharp \emph{square} dependence of $[w]_2$:
\begin{equation}\label{eq:a2theoremintro}
\|Tf\|_{L^p_w(\R^d)}\lesssim[w]_2^2\|f\|_{L^p_w(\R^d)},
\end{equation}
see, e.g., \cite{Le13b}. Furthermore, one can apply the sharp version of the Rubio de Francia extrapolation theorem of \cite{DGPP05} to \eqref{eq:a2theoremintro} to obtain the sharp bound
\[
\|Tf\|_{L^p_w(\R^d)}\lesssim[w]_p^{\max\{p,p'\}}\|f\|_{L^p_w(\R^d)},
\]
valid for every $1<p<\infty$. This extrapolation theorem was very recently extended to the matrix weight setting $n>1$ by Bownik and Cruz-Uribe in \cite{BC23}. Their quantitative bounds match the one of the scalar-valued case $n=1$. Due to the failure of the bound \eqref{eq:a2theoremintro} when $n>2$, applying their result to \eqref{eq:a2matrixtheoremintro} yields
\[
\|\widetilde{T}f\|_{L^p_W(\R^d;\F^n)}\lesssim [W]_p^{\frac{3}{2}\max\{p,p'\}}\|f\|_{L^p_W(\R^d;\F^n)},
\]
which does not recover the known bound
\[
\|\widetilde{T}f\|_{L^p_W(\R^d;\F^n)}\lesssim [W]_p^{p+p'-1}\|f\|_{L^p_W(\R^d;\F^n)}
\]
by Cruz-Uribe, Isralowitz, and Moen in \cite[Corollary~1.16]{CIM18}. Thus, finding the sharp bounds outside of $p=2$ remains open. We refer the reader to \cite{Cr24} for an overview of the history of this problem. Qualitatively, the paper \cite{BC23} marks a significant development in the theory of matrix weights. They fully develop the theory of convex-set valued mappings to define an analogue of the Hardy-Littlewood maximal operator adapted to matrix weights. 

Denote the collection of non-empty, closed, convex, and symmetric subsets of $\F^n$ by $\mc{K}$. Given a mapping $F:\R^d\to\mc{K}$, one can define its set of measurable \emph{selections}
\[
S^0(\R^d;F):=\{f\in L^0(\R^d;\F^n):f(x)\in F(x)\text{ a.e.}\},
\]
where $L^0(\R^d;\F^n)$ denotes the space of measurable functions $f:\R^d\to\F^n$. Moreover, we can define the average $\langle F\rangle_Q$ through the so-called Aumann integral as
\[
\langle F\rangle_Q:=\Big\{\frac{1}{|Q|}\int_Q\!f\,\mathrm{d}x: f\in S^0(\R^d;F)\Big\}.
\]
This generalizes \eqref{eq:convexsetaverageintro} in the sense that if $\mc{K}(f)(x)$ denotes the smallest set in $\mc{K}$ containing the vector $f(x)\in\F^n$, then
\[
\langle \mc{K}(f)\rangle_Q=\llangle f\rrangle_Q.
\]
One can now define $M^{\mc{K}}F:\R^d\to\mc{K}$ by letting $M^{\mc{K}}F(x)$ be the smallest set in $\mc{K}$ containing
\[
\bigcup_{Q\ni x}\langle F\rangle_Q.
\]
Defining $L^p_W(\R^d;\mc{K})$ as the space of those $F:\R^d\to\mc{K}$ for which the function
\[
h(x):=\sup_{u\in F(x)}|W(x)u|
\]
satisfies $h\in L^p(\R^d)$ with norm $\|F\|_{L^p_W(\R^d;\mc{K})}:=\|h\|_{L^p(\R^d)}$, it was shown in \cite[Theorem~6.9]{BC23} that
\[
\|M^{\mc{K}}F\|_{L^p_W(\R^d;\mc{K})}\lesssim [W]_p^{p'}\|F\|_{L^p_W(\R^d;\mc{K})}
\]
for all $1<p\leq\infty$. This completely recovers Buckley's sharp bound \cite{Bu93} when $n=1$. Moreover, defining $T_QF(x):=\langle F\rangle_Q\ind_Q(x)$, they showed in \cite[Proposition~6.6]{BC23} that for all $1\leq p\leq \infty$ we have
\begin{equation}\label{eq:muckenhouptmatrixweightintro}
\sup_Q\|T_Q\|_{L^p_W(\R^d;\mc{K})\to L^p_W(\R^d;\mc{K})}\eqsim[W]_p.
\end{equation}
Even beyond the setting of weighted Lebesgue spaces, this condition of uniformly bounding the averaging operators over all cubes serves as a useful generalization of the Muckenhoupt condition. We refer the reader to \cite{Ni24} and references therein for an overview. 

\subsection{Banach function spaces} The theory of singular integrals in more general spaces of functions has received considerable attention in the past years. A space $X$ is called a Banach function space if it is a complete normed subspace of the space $L^0(\R^d)$ of measurable $\F$-valued functions, and satisfies:
\begin{itemize}
    \item \emph{The ideal property:} if $f\in X$ and $|g|\leq |f|$ a.e., then $g\in X$ with $\|g\|_X\leq\|f\|_X$;
    \item \emph{The saturation property:} if $E\subseteq\R^d$ satisfies $|E|>0$, then there is a subset $F\subseteq E$ with $|F|>0$ and $\ind_F\in X$.
\end{itemize}
The saturation property is equivalent to the existence of a so-called weak order unit, i.e., a function $\rho\in X$ satisfying $\rho>0$ a.e., or to the non-degeneracy property that if
\[
\int_{\R^d}fg\,\mathrm{d}x=0
\]
for all $f\in X$, then $g=0$. A detailed survey of why these are the right assumptions to make in the definition of a Banach function space can be found in the work of Lorist and the author in \cite{LN23b}. As discussed, e.g., in the books \cite{CF13, MMMMM22}, this framework allows one to tackle the theory of singular integrals in spaces such as weighted Morrey spaces and weighted variable Lebesgue spaces, which are spaces that respectively appear as boundary data of elliptic PDEs, or as non-homogeneous data of PDEs within areas such as image processing and quasi-Newtonian fluids.

In order to extend this framework to the setting of matrix weights, one can define a space $X_W$ as the space of those $f\in L^0(\R^d;\F^n)$ for which $|Wf|\in X$ for a given matrix weight $W$. Exactly as was done in \cite{BC23} for $X=L^p(\R^d)$, one can study the boundedness properties of singular integrals in these spaces by extending them to spaces of convex-set valued functions. We can define the space $X_W[\mc{K}]$ as those $F:\R^d\to\mc{K}$ for which
\[
h(x):=\sup_{u\in F(x)}|W(x)u|
\]
belongs to $X$. While this allows one to define spaces such as matrix weight analogues of Morrey spaces and variable Lebesgue spaces, it is unsatisfactory in the sense that the definition of a function space should be intrinsic, i.e., it should not be defined in terms of a given matrix weight. Even in the case $n=1$, studying properties such as the Muckenhoupt condition and boundedness of singular integrals has a very rich theory in intrinsically defined spaces, see \cite{Ni24} and references therein for an overview.

In the scalar-valued case $n=1$, the ideal property states that a Banach function space $X$ is a \emph{Banach lattice} in the sense that the norm preserves the a.e. partial ordering of functions $|g(x)|\leq |f(x)|$. When $n>1$, however, there are many natural partial orderings one can give to $\F^n$. To create a theory compatible with matrix weights, we need to partially order $\F^n$ in such a way that a vector $u$ dominating a vector $v$ should imply that $|Av|\leq|Au|$ for every matrix $A\in\F^{n\times n}$. As it turns out, this has a very elegant characterization in terms of convex bodies. Given $u\in\F^n$, define
\[
\mc{K}(u):=\{\lambda u:\lambda\in\F, |\lambda|\leq 1\},
\]
which is the smallest set in $\mc{K}$ containing $u$. As a (non-strict) partial ordering, we interpret the inclusion $\mc{K}(v)\subseteq\mc{K}(u)$ as $u$ dominating $v$. This can be extended to functions $f,g\in L^0(\R^d;\F^n)$ through writing $\mc{K}(f)(x):=\mc{K}(f(x))$, and asking that
\[
\mc{K}(g)(x)\subseteq\mc{K}(f)(x)
\]
for a.e. $x\in\R^d$. This ordering can be characterized as follows:
\begin{proposition*}
Let $f,g\in L^0(\R^d;\F^n)$. The following are equivalent:
\begin{enumerate}[(i)]
    \item $\mc{K}(g)(x)\subseteq\mc{K}(f)(x)$ for a.e $x\in\R^d$;
    \item $g(x)\in\mc{K}(f)(x)$ for a.e. $x\in\R^d$;
    \item $|g(x)\cdot u|\leq|f(x)\cdot u|$ for all $u\in\F^n$ for a.e. $x\in\R^d$;
    \item $|Ag(x)|\leq|Af(x)|$ for all $A\in\F^{n\times n}$ for a.e. $x\in\R^d$;
    \item $g=hf$ for some $h\in L^\infty(\R^d)$ with $\|h\|_{L^\infty(\R^d)}\leq 1$.
\end{enumerate}
\end{proposition*}
For a proof, see Proposition~\ref{prop:ideal} below. In essence, the ordering $\mc{K}(g)(x)\subseteq\mc{K}(f)(x)$ means that for a.e. $x\in\R^d$, $g(x)$ points in the same direction as $f(x)$, but with a smaller magnitude. We propose the following definition of an \emph{$\F^n$-directional Banach function space}:
\begin{definition*}
We say that $\mb{X}$ is an $\F^n$-directional Banach function space (over $\R^d)$ if it is a complete normed subspace of $L^0(\R^d;\F^n)$ and satisfies the following properties:
\begin{itemize}
    \item \emph{The directional ideal property:} For all $f\in\mb{X}$ and $g\in L^0(\R^d;\F^n)$ satisfying $\mc{K}(g)\subseteq\mc{K}(f)$ a.e., we have $g\in\mb{X}$ with $\|g\|_{\mb{X}}\leq\|f\|_{\mb{X}}$;
    \item\emph{Non-degeneracy:} If $g\in L^0(\R^d;\F^n)$ satisfies $\int_{\R^d}\!f\cdot g\,\mathrm{d}x=0$ for all $f\in\mb{X}$, then $g=0$ a.e.
\end{itemize}    
\end{definition*}
The non-degeneracy property allows us to define the K\"othe dual $\mb{X}'$ of $\mb{X}$ through
\[
\|g\|_{\mb{X}'}:=\sup_{\|f\|_{\mb{X}}=1}\int_{\R^d}\!|f\cdot g|\,\mathrm{d}x,
\]
which again satisfies the directional ideal property. We call $\mb{X}$ \emph{K\"othe reflexive} if $\mb{X}''=\mb{X}$. When $n=1$, K\"othe reflexivity is characterized by the Fatou property through the Lorentz-Luxemburg theorem. A similar assertion is true for $\F^n$-directional quasi-Banach function spaces under a stronger saturation condition, see Theorem~\ref{thm:lorentzluxemburg}. We emphasize here that K\"othe reflexivity is much weaker than reflexivity in the functional analysis sense. For example, $L^1(\R^d;\F^n)$ and $L^\infty(\R^d;\F^n)$ are K\"othe reflexive.

Our main example of an $\F^n$-directional Banach function space is a matrix-weighted space:
\begin{itemize}
    \item Given a matrix weight $W$ and a Banach function space $X$ over $\Omega$, we define $X_W$ as those $f\in L^0(\Omega;\F^n)$ for which $|Wf|\in X$, with
    \[
    \|f\|_{X_W}:=\||Wf|\|_X,
    \]
    see Section~\ref{sec:matrixweights} for an elaboration.
\end{itemize}
This is, however, not the only possible example. We can also define component-wise spaces:
\begin{itemize}
    \item Given Banach function spaces $X_1,\ldots,X_n$ and an orthonormal basis $(v_k)_{k=1}^n$ of $\F^n$, we define $\mb{X}$ as those $f\in L^0(\Omega;\F^n)$ for which
    \[
    \|f\|_{\mb{X}}:=\sum_{k=1}^n\|f\cdot v_k\|_{X_k}<\infty,
    \]
    see Section~\ref{sec:componentwise} for an elaboration.
\end{itemize}

Our $\F^n$-directional Banach function spaces can be naturally extended to spaces of convex-set valued mappings $F:\R^d\to\mc{K}$. We define $L^0(\R^d;\mc{K})$ as the space of $F:\R^d\to\mc{K}$ that are measurable in the sense that for every measurable $E\subseteq\F^n$, the set
\[
F^{-1}(E):=\{x\in\R^d:F(x)\cap E\neq\emptyset\}
\]
is also measurable. Then we define $\mb{X}[\mc{K}]$ as the space of those $F\in L^0(\R^d;\mc{K})$ for which its space of measurable selections $S^0(\R^d;F)$ is a bounded subset of $\mb{X}$, with
\[
\|F\|_{\mb{X}[\mc{K}]}:=\sup_{f\in S^0(\R^d;F)}\|f\|_{\mb{X}}.
\]
This space naturally contains $\mb{X}$ through the isometric embedding $f\mapsto\mc{K}(f)$, see Proposition~\ref{prop:convexsetfequalsf} below. Working with the general definition of $\mb{X}[\mc{K}]$ in terms of the selections $S^0(\R^d;F)$ is facilitated by the so-called Filippov selection theorem. It states that if $\phi:\R^d\times\F^n\to\R$ is measurable in its first coordinate and continuous in its second, then for any $h\in L^0(\R^d)$ for which for a.e. $x\in\R^d$, $h(x)$ is of the form $\phi(x,u)$ for some $(x,u)$ satisfying $u\in F(x)$, then there is a selection $f\in S^0(\R^d;F)$ for which
\[
h(x)=\phi(x,f(x)).
\]
For example, this can be used to show that if $F(x)$ is a bounded set for a.e. $x\in\R^d$, then, for a given matrix weight $W:\R^d\to\F^{n\times n}$, there is a selection $f_0\in S^0(\R^d;F)$ such that
\[
|f_0(x)|=\sup_{u\in F(x)}|W(x)u|.
\]
As a consequence,  when $X$ is a Banach function space with the Fatou property, the space $X_W[\mc{K}]$ coincides exactly with our earlier defined space of $F\in L^0(\R^d;\mc{K})$ for which the function
\[
h(x):=\sup_{u\in F(x)}|W(x)u|
\]
belongs to $X$, see Proposition~\ref{prop:supattainmentXW} below.

\subsection{Main results} Our first main result is a generalization of the extrapolation theorem of Bownik and Cruz-Uribe \cite{BC23} to $\F^n$-directional Banach function spaces.
\begin{TheoremLetter}\label{thm:A}
Let $1\leq p\leq\infty$ and suppose
\[
T:\bigcup_{W\in A_p}L^p_W(\R^d;\F^n)\to L^0(\R^d;\F^n)
\]
is a map for which there is an increasing function $\phi:[0,\infty)\to[0,\infty)$ such that for all $W\in A_p$ and all $f\in L^p_W(\R^d;\F^n)$ we have
\[
\|Tf\|_{L^p_W(\R^d;\F^n)}\leq\phi([W]_p)\|f\|_{L^p_W(\R^d;\F^n)}.
\]
Let $\mb{X}$ be a K\"othe reflexive $\F^n$-directional Banach function space over $\R^d$ for which both $\mb{X}$ and $\mb{X}'$ are component-wise saturated, and
\[
M^{\mc{K}}:\mb{X}[\mc{K}]\to \mb{X}[\mc{K}],\quad M^{\mc{K}}:\mb{X}'[\mc{K}]\to \mb{X}'[\mc{K}].
\]
Then $Tf$ is well-defined for all $f\in\mb{X}$, and
\[
\|Tf\|_{\mb{X}}\lesssim_n \phi(C_n\|M^{\mc{K}}\|_{\mb{X}[\mc{K}]\to \mb{X}[\mc{K}]}^{\frac{1}{p'}}\|M^{\mc{K}}\|^{\frac{1}{p}}_{\mb{X}'[\mc{K}]\to \mb{X}'[\mc{K}]})\|f\|_{\mb{X}}.
\]
If $p=\infty$ or $p=1$, we can omit the bound $M^{\mc{K}}:\mb{X}'[\mc{K}]\to \mb{X}'[\mc{K}]$ or $M^{\mc{K}}:\mb{X}[\mc{K}]\to \mb{X}[\mc{K}]$ respectively.
\end{TheoremLetter}
This result is proven in a more general pairs of functions setting in Theorem~\ref{thm:BFSextrapolation} below. The component-wise saturation property is a directional analogue of the existence of a weak order unit. It can be formulated by saying that the space $\mb{X}[\mc{K}]$ contains a non-degenerate ellipsoid, i.e., there is a Hermitian positive definite matrix mapping $U:\R^d\to\F^{n\times n}$ for which the map $F(x):=U(x)\overline{B}$, where $\overline{B}$ is the unit ball in $\F^n$, belongs to $\mb{X}[\mc{K}]$. This condition is stronger than the non-degeneracy property, but is satisfied by our main examples, see Section~\ref{sec:directionalqbfs}. 

Our proof is new, and distinct from the one of Bownik and Cruz-Uribe in how we use our Rubio de Francia iteration algorithm. In their work, they apply the Rubio de Francia algorithm to an auxiliary operator
\[
TF(x):=\sup_{u\in M^{\mc{K}}F(x)}|W(x)u|W(x)^{-1}\overline{B},
\]
where $\overline{B}$ is the closed unit ball in $\F^n$. Moreover, they explicitly use the fact that this operator maps the space of scalar function multiples of the ellipsoid $W(x)^{-1}\overline{B}$ to itself. As we have no matrix weight to work with, our proof of Theorem~\ref{thm:A} instead relies on the Filippov selection theorem and a function norm approach, yielding a very elegant intrinsic argument. When $n=1$, our approach is similar to the one used for Lebesgue spaces in \cite{CMP11}; in particular, it does not recover the \emph{sharp} extrapolation theorem for matrix weighted Lebesgue spaces. We do emphasize that our approach can be adapted to give an intrinsic proof of the extrapolation theorem of \cite{BC23}, and a proof of this can be found in Appendix~\ref{app:lebesgueextrapolation}. We also note that when $n=1$, our result recovers the sharpest result known in the literature for general Banach function spaces of \cite{CMM22}, see also \cite[Section~4.7]{Ni23} for a comparison.

We apply this theorem to obtain a extrapolation theorems for matrix-weighted variable Lebesgue and Morrey spaces respectively in Theorem~\ref{thm:varlebesgueextrapolation} and Theorem~\ref{thm:morreyextrapolation} below. In particular, we recover the extrapolation theorem of Cruz-Uribe and Wang \cite[Theorem~2.7]{CW17} for weighted variable Lebesgue spaces in the case $n=1$.

If $T$ is a Calder\'on-Zygmund operator, then we can apply Theorem~\ref{thm:A} for $p=2$ to \eqref{eq:a2matrixtheoremintro} to conclude that
\begin{equation}\label{eq:badczobound}
\|T\|_{\mb{X}\to\mb{X}}\lesssim \|M^{\mc{K}}\|^{\frac{3}{2}}_{\mb{X}[\mc{K}]\to \mb{X}[\mc{K}]}\|M^{\mc{K}}\|^{\frac{3}{2}}_{\mb{X}'[\mc{K}]\to \mb{X}'[\mc{K}]}.
\end{equation}
This, however, is certainly not sharp. Our second main theorem shows how the power $\tfrac{3}{2}$ can be reduced to $1$ in this estimate using convex body domination, even without assuming the component-wise saturation property.

For a collection of cubes $\mc{P}$ and a locally integrably bounded $F:\R^d\to\mc{K}$, we define the averaging operator $T_{\mc{P}}$ as
\[
T_{\mc{P}}F(x):=\sum_{Q\in\mc{P}}\langle F\rangle_Q\ind_Q(x),
\]
where scalar multiples of a set are considered pointwise, and the sum is interpreted as a Minkowski sum. When $F=\mc{K}(f)$ for a function $f$ integrable over the cubes in $\mc{P}$, we have
\[
T_{\mc{P}}(\mc{K}(f))(x)=\sum_{Q\in\mc{P}}\llangle f\rrangle_Q\ind_Q(x).
\]
A collection of cubes $\mc{S}$ is called \emph{sparse} if for all $Q\in\mc{S}$ we have
\[
\Big|\bigcup_{\substack{Q'\in\mc{S}\\ Q'\subsetneq Q}}Q'\Big|\leq\tfrac{1}{2}|Q|.
\]
The boundedness properties of $T_{\mc{S}}$ in $\mb{X}[\mc{K}]$ for sparse collections $\mc{S}$ dictate precisely in which spaces $\mb{X}$ an operator satisfying convex body domination is bounded. Indeed, if 
\[
\widetilde{T}f(x)\in C_T\sum_{Q\in\mc{S}}\llangle f\rrangle_Q\ind(x),
\]
then, since $\|\mc{K}(f)\|_{\mb{X}[\mc{K}]}=\|f\|_{\mb{X}}$,
\begin{equation}\label{eq:convexbodydomestimateintro}
\|\widetilde{T}f\|_{\mb{X}}\lesssim \|T_{\mc{S}}(\mc{K}(f))\|_{\mb{X}[\mc{K}]}\leq \|T_{\mc{S}}\|_{\mb{X}[\mc{K}]\to\mb{X}[\mc{K}]}\|f\|_{\mb{X}}.
\end{equation}
Thus, to study bounds for $T$ in $\mb{X}$, we need to study the behavior of $T_{\mc{S}}$ in $\mb{X}[\mc{K}]$. We have the following result on the relationship between the boundedness of the sparse operator $T_{\mc{S}}$ and of $M^{\mc{K}}$:
\begin{TheoremLetter}\label{thm:B}
Let $\mb{X}$ be a K\"othe reflexive $\F^n$-directional Banach function space. Then the following are equivalent:
\begin{enumerate}[(i)]
    \item\label{it:thmB1} $T_{\mc{S}}:\mb{X}[\mc{K}]\to\mb{X}[\mc{K}]$ and $T_{\mc{S}}:\mb{X}'[\mc{K}]\to\mb{X}'[\mc{K}]$ uniformly for all sparse collection $\mc{S}$;
    \item\label{it:thmB2} $M^{\mc{K}}:\mb{X}[\mc{K}]\to\mb{X}[\mc{K}]$ and $M^{\mc{K}}:\mb{X}'[\mc{K}]\to\mb{X}'[\mc{K}]$.
\end{enumerate}
Moreover, in this case we have
\begin{align*}
\|M^{\mc{K}}\|_{\mb{X}[\mc{K}]\to\mb{X}[\mc{K}]}&\lesssim_{d,n}\|T_{\mc{S}}\|_{\mb{X}[\mc{K}]\to\mb{X}[\mc{K}]};\\
\|M^{\mc{K}}\|_{\mb{X}'[\mc{K}]\to\mb{X}'[\mc{K}]}&\lesssim_{d,n}\|T_{\mc{S}}\|_{\mb{X}'[\mc{K}]\to\mb{X}'[\mc{K}]}\\
\max\{\|T_{\mc{S}}\|_{\mb{X}[\mc{K}]\to\mb{X}[\mc{K}]},\|T_{\mc{S}}\|_{\mb{X}'[\mc{K}]\to\mb{X}'[\mc{K}]}\}&\lesssim_n \|M^{\mc{K}}\|_{\mb{X}[\mc{K}]\to\mb{X}[\mc{K}]}\|M^{\mc{K}}\|_{\mb{X}'[\mc{K}]\to\mb{X}'[\mc{K}]}.
\end{align*}
\end{TheoremLetter}
The proof of this result can be found in Section~\ref{sec:thmB}. Surprisingly, it is not clear if the two bounds in \ref{it:thmB1} are equivalent. This is in stark contrast with the fact that they \emph{are} equivalent when viewing the operator $T_{\mc{S}}$ as a linear operator on $\mb{X}$ and $\mb{X}'$ instead of their convex-set valued variants. Moreover, when $\mc{S}$ is pairwise disjoint, this symmetry \emph{is} true in $\mb{X}[\mc{K}]$, see Theorem~\ref{thm:pairwisedisjointavconvex} below.

To prove Theorem~\ref{thm:B} we need to establish a convex body domination result for $M^{\mc{K}}$, which is interesting in its own right. Here, for a collection of cubes $\mc{P}$, we define $M^{\mc{K}}_{\mc{P}}$ in the same way as $M^{\mc{K}}$, but with the union $\bigcup_Q\langle F\rangle_Q\ind_Q(x)$ taken only over the $Q\in\mc{P}$.
\begin{TheoremLetter}\label{thm:C}
Let $\mc{D}$ be a dyadic grid and let $\mc{F}\subseteq\mc{D}$ be finite. For each $F\in L^0(\R^d;\mc{K})$ satisfying $F\ind_Q\in L^1(\R^d;\mc{K})$ for all cubes $Q$, there is a sparse collection $\mc{S}\subseteq\mc{F}$ for which
\[
M^{\mc{K}}_{\mc{F}}F(x)\subseteq C_n M^{\mc{K}}_{\mc{S}}F(x)
\]
for a.e. $x\in\R^d$.
\end{TheoremLetter}
This theorem can be found below as Theorem~\ref{thm:convexbodydomofmax}. As a consequence of Theorem~\ref{thm:B} and \eqref{eq:convexbodydomestimateintro}, we obtain the following improvement of \eqref{eq:badczobound}:
\begin{CorollaryLetter}
Let $T$ be an operator for which there is a constant $C_T>0$ such that for every $f\in L^\infty_c(\R^d;\F^n)$ there is a sparse collection $\mc{S}$ for which
\[
Tf(x)\in C_T\sum_{Q\in\mc{S}}\llangle f\rrangle_Q\ind_Q(x).
\]
Let $\mb{X}$ be a K\"othe reflexive $\F^n$-directional Banach function space for which
\[
M^{\mc{K}}:\mb{X}[\mc{K}]\to \mb{X}[\mc{K}],\quad M^{\mc{K}}:\mb{X}'[\mc{K}]\to \mb{X}'[\mc{K}].
\]
Then, for all $f\in L^\infty_c(\R^d;\F^n)\cap\mb{X}$ we have
\[
\|Tf\|_{\mb{X}}\lesssim C_T\|M^{\mc{K}}\|_{\mb{X}[\mc{K}]\to\mb{X}[\mc{K}]}\|M^{\mc{K}}\|_{\mb{X}'[\mc{K}]\to\mb{X}'[\mc{K}]}\|f\|_{\mb{X}}.
\]
\end{CorollaryLetter}

Finally, we initiate a study into the boundedness of averaging operators related to the Muckenhoupt condition in $\F^n$-directional Banach function spaces $\mb{X}$. As we noted in \eqref{eq:muckenhouptmatrixweightintro}, the matrix $A_p$ condition can be recovered through the uniform boundedness of the averaging operators
\[
T_QF(x):=\langle F\rangle_Q\ind_Q(x)
\]
in $L^p_W(\R^d;\mc{K})$. Intriguingly, one does not need to bound $T_Q$ on $L^p_W(\R^d;\mc{K})$ to recover the matrix $A_p$ condition. Indeed, if we define $T_Qf:=\langle f\rangle_Q\ind_Q$ for $f\in L^0(\R^d;\F^n)$, where $\langle f\rangle_Q:=\tfrac{1}{|Q|}\int_Q\!f\,\mathrm{d}x$, then we already have
\[
\sup_Q\|T_Q\|_{L^p_W(\R^d;\F^n)\to L^p_W(\R^d;\F^n)}=[W]_p,
\]
this time with a strict equality. In general, given a pairwise disjoint collection of cubes $\mc{P}$ and $T_{\mc{P}}f:=\sum_{Q\in\mc{P}}\langle f\rangle_Q\ind_Q$, then, by a proof based on the John ellipsoid theorem, we have
\[
T_{\mc{P}}:\mb{X}\to\mb{X}\quad \Leftrightarrow\quad T_{\mc{P}}:\mb{X}[\mc{K}]\to\mb{X}[\mc{K}]
\]
with comparable operator norm, up to a constant depending only on $n$, see Theorem~\ref{thm:pairwisedisjointavconvex} below.

When $n=1$, the Muckenhoupt condition $X\in A$ is defined as
\[
[X]_A:=\sup_Q|Q|^{-1}\|\ind_Q\|_X\|\ind_Q\|_{X'}<\infty,
\]
and the quantity $|Q|^{-1}\|\ind_Q\|_X\|\ind_Q\|_{X'}$ is precisely the operator norm of $T_Q:X\to X$, see \cite[Section~3]{Ni24}. Generalizing this to $n>1$, we say that an $\F^n$-directional Banach function space $\mb{X}$ with the component-wise saturation property satisfies the Muckenhoupt condition $\mb{X}\in A$ if
\[
[\mb{X}]_A:=\sup_Q\|T_Q\|_{\mb{X}\to\mb{X}}<\infty.
\]
As is the case for $n=1$, we can determine a more or less exact expression for the operator norm $\|T_Q\|_{\mb{X}\to\mb{X}}$ in terms of indicator functions. If $\mb{X}$ is K\"othe reflexive, then we have $X\in A$ precisely if for all cubes $Q$ we have $\ind_Q u\in\mb{X},\mb{X}'$ for all $u\in\F^n$, and there is a constant $C>0$ such that for every $u\in\F^n$ and every cube $Q$ there is a non-zero $v\in\F^n$ for which
\[
\|\ind_Qu\|_{\mb{X}}\|\ind_Q v\|_{\mb{X}'}\leq C|u\cdot v|.
\]
Moreover, the optimal $C$ coincides with $[\mb{X}]_A$, see Section~\ref{sec:averagingoperators} below. When $\mb{X}=L^p_W(\R^d;\F^n)$, then this is precisely the Muckenhoupt condition $W\in A_p$. Several further characterizations of $\mb{X}\in A$ are given in Theorem~\ref{thm:muckenhouptdef} below.

To study the relationship between the boundedness of $M^{\mc{K}}$ in $\mb{X}[\mc{K}]$ and the Muckenhoupt condition $\mb{X}\in A$, we also define an analogue of the \emph{strong} Muckenhoupt condition. We say that $\mb{X}$ satisfies the strong Muckenhoupt condition $\mb{X}\in A_{\text{strong}}$ if
\[
[\mb{X}]_{A_{\text{strong}}}:=\sup_{\mc{P}}\|T_{\mc{P}}\|_{\mb{X}\to\mb{X}}<\infty
\]
where the supremum is taken over all pairwise disjoint collections of cubes $\mc{P}$. As noted in the above discussion, we have
\[
[\mb{X}]_A\eqsim_n \sup_Q\|T_Q\|_{\mb{X}[\mc{K}]\to\mb{X}[\mc{K}]},\quad [\mb{X}]_{A_{\text{strong}}}\eqsim_n \sup_{\mc{P}}\|T_{\mc{P}}\|_{\mb{X}[\mc{K}]\to\mb{X}[\mc{K}]}.
\]
Thus, it is not surprising that the (strong) Muckenhoupt condition is related to the bounds of $M^{\mc{K}}$ in $\mb{X}[\mc{K}]$. We define
\[
MF(x):=\overline{\bigcup_{Q\ni x}\langle F\rangle_Q},
\]
so that $M^{\mc{K}}F(x)$ is the smallest set in $\mc{K}$ containing $MF(x)$. We say that $M:\mb{X}[\mc{K}]\to \mb{X}_{\text{weak}}[\mc{K}]$ if there is a constant $C>0$ such that for all $F\in\mb{X}[\mc{K}]$ we have
\[
\sup_{u\in\F^n}\|\ind_{\{x\in\R^d:u\in MF(x)\}}u\|_{\mb{X}}\leq C\|F\|_{\mb{X}[\mc{K}]}.
\]
Moreover, we denote the optimal constant $C$ by $\|M\|_{\mb{X}[\mc{K}]\to\mb{X}_{\text{weak}}[\mc{K}]}$. Weak-type bounds for $M^{\mc{K}}$ in $L^p(\R^d;\F^n)$ were also considered in \cite{BC23}, but their definition of a weak-type space differs from ours: they considered the condition that the function
\[
h(x):=\sup_{u\in M^{\mc{K}}F(x)}|u|
\]
satisfies $h\in L^{p,\infty}(\R^d)$. We note that, in fact, we have
\[
h(x)=\sup_{u\in MF(x)}|u|,
\]
so their bound for $M^{\mc{K}}$ and $M$ coincide. However when working with our definition of the weak-type bound, this equivalence no longer holds, and we have to work with the smaller operator $M$. Nevertheless, it seems that our weak-type bound is better suited for the theory in the sense that it characterizes certain Muckenhoupt conditions. Indeed, we have the following result, extending \cite[Theorem~B]{Ni24}:
\begin{TheoremLetter}\label{thm:E}
Let $\mb{X}$ be a K\"othe reflexive $\F^n$-directional Banach function space over $\R^d$ for which both $\mb{X}$ and $\mb{X}'$ satisfy the component-wise saturation property. Consider the following statements:
\begin{enumerate}[(a)]
    \item $M^{\mc{K}}:\mb{X}[\mc{K}]\to\mb{X}[\mc{K}]$;
    \item $\mb{X}\in A_{\text{strong}}$;
    \item $M:\mb{X}[\mc{K}]\to\mb{X}_{\text{weak}}[\mc{K}]$;
    \item $\mb{X}\in A$.
\end{enumerate}
Then \ref{it:mweak1}$\Rightarrow$\ref{it:mweak2}$\Rightarrow$\ref{it:mweak3}$\Rightarrow$\ref{it:mweak4} with
\[
[\mb{X}]_{A}\leq\|M\|_{\mb{X}[\mc{K}]\to\mb{X}_{\text{weak}}[\mc{K}]}\lesssim_{d,n}[\mb{X}]_{A_{\text{strong}}}\leq\|M^{\mc{K}}\|_{\mb{X}[\mc{K}]\to\mb{X}[\mc{K}]}.
\]
Furthermore, if there is a $C\geq 1$ such that for all pairwise disjoint collections of cubes $\mc{P}$ and all $f\in\mb{X}$, $g\in\mb{X}'$ we have
\begin{equation}\label{eq:mbxing}
\sum_{Q\in\mc{P}}\|\ind_Q f\|_{\mb{X}}\|\ind_Q g\|_{\mb{X}'}\leq C\|f\|_{\mb{X}}\|g\|_{\mb{X}'},
\end{equation}
then \ref{it:mweak2}-\ref{it:mweak4} are equivalent, with
\[
[\mb{X}]_{A_{\text{strong}}}\eqsim\|M\|_{\mb{X}[\mc{K}]\to\mb{X}_{\text{weak}}[\mc{K}]}\eqsim [\mb{X}]_A.
\]   
\end{TheoremLetter}
This result can be found as Theorem~\ref{thm:mweak} below. The condition \eqref{eq:mbxing} when $n=1$ is often denoted as $X\in\mc{G}$. Moreover, if $X\in\mc{G}$, then for any matrix weight $W$, the estimate \eqref{eq:mbxing} is satisfied by $X_W$. We refer the reader to \cite{Ni24} and references therein for an overview of the condition $\mc{G}$.

Example of spaces $X$ satisfying $X\in\mc{G}$ include variable Lebesgue spaces with global $\log$-H\"older regular exponents, see Subsection~\ref{subsec:varlebesgue} below. The resulting corollary of Theorem~\ref{thm:E} can be found as Theorem~\ref{thm:varlebesguemuckenhoupt} below. In particular, an application of H\"older's inequality shows that $X=L^p(\R^d)\in\mc{G}$ with $C=1$, in which case this corollary yields:
\begin{CorollaryLetter}\label{cor:F}
Let $1\leq p\leq\infty$ and let $W:\R^d\to\F^{n\times n}$ be a matrix weight. Then the following are equivalent:
\begin{enumerate}[(i)]
    \item $W\in A_p$;
    \item $M:L^p_W(\R^d;\mc{K})\to L^p_W(\R^d;\mc{K})_{\text{weak}}$.
\end{enumerate}
Moreover, we have
\[
[W]_p\eqsim \|M\|_{L^p_W(\R^d;\mc{K})\to L^p_W(\R^d;\mc{K})_{\text{weak}}}.
\]
\end{CorollaryLetter}
Here $L^p_W(\R^d;\mc{K})_{\text{weak}}$ denotes the space $\mb{X}_{\text{weak}}[\mc{K}]$ for $\mb{X}=L^p_W(\R^d;\F^n)$. 

\subsection*{Organization}
This paper is organized as follows:
\begin{itemize}
    \item In Section~\ref{sec:convexbodies} we define the preliminary notions of convex bodies and convex-set valued functions that we will require throughout this work.
    \item In Section~\ref{sec:directionalqbfs} we introduce directional quasi-Banach function spaces and their convex-set valued analogues, and prove basic equivalences and results for these spaces. The directional analogue of the Lorentz-Luxemburg theorem is proven in Appendix~\ref{app:A}.
    \item In Section~\ref{sec:averagingoperators} we discuss the boundedness of averaging operators in directional Banach function spaces and their convex-set valued analogues in general $\sigma$-finite measure spaces. 
    \item In Section~\ref{sec:muckenhoupt} we discuss the Muckenhoupt condition and its relations to the bounds of the convex-set valued maximal operator.
    \item In Section~\ref{sec:extrapolation} we prove Theorem~\ref{thm:A}.
    \item In Section~\ref{sec:thmB} we prove Theorem~\ref{thm:B} and Theorem~\ref{thm:C}.
    \item In Section~\ref{sec:applications} we apply our main theorems to derive results for matrix-weighted variable Lebesgue and Morrey spaces, as well as for component-wise spaces.
\end{itemize}

\section{Convex bodies}\label{sec:convexbodies}
\begin{definition}
We let $\mc{C}$ denote the collection of non-empty closed sets $K\subseteq\F^n$. Moreover, we denote by $\mc{K}$ the subcollection of $\mc{C}$ of sets $K$ that satisfy the following additional properties:
\begin{itemize}
\item \emph{Convexity:} If $u,v\in K$, then for all $0\leq t\leq 1$ we have $(1-t)u+tv\in K$;
\item \emph{Symmetry:} If $u\in K$, then $\lambda u\in K$ for all $\lambda\in\F$ with $|\lambda|=1$;
\end{itemize}
We let $\mc{K}_b$ denote the \emph{bounded} sets in $\mc{K}$.
\end{definition}
For $K,L\in\mc{C}$ and $\lambda\in\F$ we define $K+L$  and $\lambda K$ through
\[
K+L:=\overline{\{u+v:u\in K,\, v\in L\}}\in\mc{C},\quad \lambda K:=\{\lambda u:u\in K\}\in\mc{C}.
\]
Taking the closure in the definition of $K+L$ is necessary, unless one of $K$ or $L$ is compact, in which case the Minkowski sum is closed.

As $\mc{K}$ is closed under arbitrary intersections, given a set $B\subseteq\F^n$, we define
\[
\mc{K}(B):=\bigcap\{K\in\mc{K}:B\subseteq K\}\in\mc{K},
\]
i.e., $\mc{K}(B)$ is the smallest set in $\mc{K}$ containing $B$. Note that if $B$ is bounded, then $\mc{K}(B)\in\mc{K}_b$.

Given a vector $u\in\F^n$, we write $\mc{K}(u):=\mc{K}(\{u\})$.
\begin{proposition}\label{prop:smallestkset}
Let $u\in\F^n$. Then $\mc{K}(u)=\{\lambda u:\lambda\in\F,\, |\lambda|\leq 1\}$.
\end{proposition}
\begin{proof}
Since $K:=\{\lambda u:\lambda\in\F,\, |\lambda|\leq 1\}\in\mc{K}$ and $u\in K$, we have $\mc{K}(u)\subseteq K$. Conversely, if $v\in K$, then $v=\lambda u$ for some $|\lambda|\leq 1$. Since $u\in \mc{K}(u)$, symmetry implies that also $|\lambda|^{-1}v,-|\lambda|^{-1}v\in\mc{K}(u)$. Hence, setting $t:=\tfrac{1}{2}(1-|\lambda|)\in[0,1]$, by convexity of $\mc{K}(u)$ we have
\[
v=(1-t)|\lambda|^{-1}v-t|\lambda|^{-1}v\in\mc{K}(u).
\]
The result follows.
\end{proof}

The space $\mc{C}_b$ of closed and bounded subsets of $\F^n$ is a complete metric space with respect to the Hausdorff distance
\[
d_{H}(K,L):=\max\{\sup_{u\in K}\inf_{v\in L}|u-v|,\sup_{v\in L}\inf_{u\in K}|u-v|\}.
\]
Moreover, $\mc{K}_b$ is a closed subset of $\mc{C}_b$ with respect to this metric.

\subsection{Convex-set valued mappings}
Given a $\sigma$-finite measure space $(\Omega,\mu)$, we define $L^0(\Omega;\mc{C})$ as the space of mappings $F:\Omega\to\mc{C}$ (modulo mappings a.e. equal to $\{0\}$) which are measurable in the sense that for every measurable set $E\subseteq\F^n$ the set
\[
F^{-1}(E):=\{x\in\Omega:F(x)\cap E\neq\emptyset\}
\]
is also measurable. We let $L^0(\Omega;\mc{K})$ denote the subset of $L^0(\Omega;\mc{C})$ consisting of the mappings $F:\Omega\to\mc{K}$, and similarly for $\mc{K}_b$. Measurability can be characterized as follows:
\begin{proposition}\label{prop:measurablemultifunction}
Let $F:\Omega\to\mc{C}$. Then the following are equivalent:
\begin{enumerate}[(i)]
\item\label{it:castaing1} $F\in L^0(\Omega;\mc{C})$;
\item\label{it:castaing2} $x\mapsto\inf_{u\in F(x)}|u-v|$ is measurable for all $v\in\F^n$;
\item\label{it:castaing3} $F^{-1}(U)$ is measurable for all open $U\subseteq\F^n$;
\item\label{it:castaing4} $F(x)=\overline{\{f_k(x):k\geq 1\}}$ a.e. for a sequence $(f_k)_{k\geq 1}$ in $L^0(\Omega;\F^n)$.
\end{enumerate}
Moreover, if $F:\Omega\to\mc{K}_b$, then the above are also equivalent to
\begin{enumerate}[(i)]\setcounter{enumi}{4}
    \item\label{it:castaing5} $F$ is measurable as a function from $\Omega$ into $\mc{K}_b$ equipped with the Borel $\sigma$-algebra induced by the Hausdorff distance.
\end{enumerate}
\end{proposition}
The equivalences \ref{it:castaing1}-\ref{it:castaing4} can be found in \cite[Theorem~8.1.4]{AF09}. For the equivalence with \ref{it:castaing5}, see \cite[Theorem~III.2]{CV77}.

The partial ordering with respect to inclusion of sets of $\mc{K}$ can be extended to $L^0(\Omega;\mc{K})$ as follows: for $F,G\in L^0(\Omega;\mc{K})$ we say that $F$  dominates $G$ if
\[
G(x)\subseteq F(x)
\]
for a.e. $x\in\Omega$. Moreover, given a function $f\in L^0(\Omega;\F^n)$, we define $\mc{K}(f):\Omega\to\mc{K}$ by
\[
\mc{K}(f)(x):=\mc{K}(f(x)).
\]
This is a measurable mapping, since for a countable dense subset $(\lambda_k)_{k\geq 1}$ of the set of $\{\lambda\in\F:|\lambda|\leq 1\}$ we have
\[
\mc{K}(f)(x)=\overline{\{\lambda_k f(x):k\geq 1\}}.
\]
For $f,g\in L^0(\Omega;\F^n)$ we say that $f$ dominates $g$ if
\[
\mc{K}(g)(x)\subseteq\mc{K}(f)(x)
\]
for a.e $x\in\Omega$. Note that this is not a strict partial ordering, since $\mc{K}(g)=\mc{K}(f)$ a.e. does not imply that $g=f$ a.e. When $n=1$, this partial ordering is precisely the relation $|g|\leq |f|$ a.e. For general $n$, $f$ dominates $g$ if and only if $|g(x)\cdot u|\leq|f(x)\cdot u|$ for all $u\in\F^n$ for a.e. $x\in\Omega$. This follows from the following result:
\begin{proposition}\label{prop:ideal}
Let $f,g\in L^0(\Omega;\F^n)$. The following are equivalent:
\begin{enumerate}[(i)]
    \item\label{it:ideal1} $\mc{K}(g)(x)\subseteq\mc{K}(f)(x)$ a.e;
    \item\label{it:ideal2} $g(x)\in\mc{K}(f)(x)$ a.e.;
    \item\label{it:ideal3} $|g(x)\cdot u|\leq|f(x)\cdot u|$ for all $u\in\F^n$ for a.e. $x\in\Omega$;
    \item\label{it:idealmatrix} $|Ag(x)|\leq|Af(x)|$ for all $A\in\F^{n\times n}$ for a.e. $x\in\Omega$;
    \item\label{it:ideal4} $g=hf$ for some $h\in L^\infty(\Omega)$ with $\|h\|_{L^\infty(\Omega)}\leq 1$.
\end{enumerate}
\end{proposition}
For the proof we use the following lemma:
\begin{lemma}\label{lem:collinear}
Let $v,w\in\F^n$. Then the following are equivalent:
\begin{enumerate}[(i)]
    \item\label{it:collinear1} $|v\cdot u|\leq|w\cdot u|$ for all $u\in\F^n$;
    \item\label{it:collinear2} $\mc{K}(v)\subseteq\mc{K}(w)$.
\end{enumerate}
\end{lemma}
\begin{proof}
For \ref{it:collinear2}$\Rightarrow$\ref{it:collinear1}, note that by Proposition~\ref{prop:smallestkset} there is a $|\lambda|\leq 1$ for which $v=\lambda w$. Hence, for all $u\in\F^n$ we have
\[
|v\cdot u|=|\lambda||w\cdot u|\leq|w\cdot u|,
\]
as desired.

For \ref{it:collinear1}$\Rightarrow$\ref{it:collinear2}, note that the inequality implies that $w^{\perp}\subseteq v^{\perp}$, and, hence,
\[
\text{span}\{v\}=(v^{\perp})^{\perp}\subseteq(w^{\perp})^{\perp}=\text{span}\{w\}.
\]
Thus, there is a $\lambda\in\F$ such that $v=\lambda w$. Then
\[
|\lambda||v|^2=|\lambda||v\cdot v|\leq|\lambda||w\cdot v|=|v|^2,
\]
so $|\lambda|\leq 1$, proving the result by Proposition~\ref{prop:smallestkset}.
\end{proof}
\begin{proof}[Proof of Proposition~\ref{prop:ideal}]
The equivalence \ref{it:ideal1}$\Leftrightarrow$\ref{it:ideal3} follows from Lemma~\ref{lem:collinear}. Moreover, for \ref{it:ideal1}$\Leftrightarrow$\ref{it:ideal2}, note that $g(x)\in\mc{K}(f)(x)$ implies, per definition of $\mc{K}(g)$, that $\mc{K}(g)(x)\subseteq\mc{K}(f)(x)$. 

The implication \ref{it:ideal4}$\Rightarrow$\ref{it:idealmatrix} follows by noting that $|Ag(x)|=|h(x)||Af(x)|\leq|Af(x)|$ for all $A\in\F^{n\times n}$ for a.e. $x\in\Omega$. For \ref{it:idealmatrix}$\Rightarrow$\ref{it:ideal3}, note that for each $u\in\F^n$, the matrix $A\in\F^{n\times n}$ whose first row is given by $(\overline{u_1},\ldots,\overline{u_n})$ and is zero elsewhere, satisfies
\[
|g(x)\cdot u|=|Ag(x)|\leq |Af(x)|=|f(x)\cdot u|,
\]
as desired. Finally, for \ref{it:ideal2}$\Rightarrow$\ref{it:ideal4}, note that by Proposition~\ref{prop:smallestkset} for a.e. $x\in\Omega$ there is a $\lambda_x\in\F$, $|\lambda_x|\leq 1$, such that $g(x)=\lambda_x f(x)$, with $\lambda_x=0$ if $f(x)=0$. It remains to check that $h(x):=\lambda_x$ is measurable. Writing $f=(f_1,\ldots,f_n)$, $g=(g_1,\ldots g_n)$, we have $\supp(f)=\bigcup_{k=1}^n\supp(f_k)$ and $h(x)=\tfrac{g_k(x)}{f_k(x)}$ whenever $x\in\supp(f_k)$. Define $E_1:=\supp(f_1)$ and iteratively define $E_k:=\supp(f_k)\backslash \bigcup_{j=1}^{k-1} E_j$ for $k=2,\ldots,n$. Then
\[
h(x)=\sum_{k=1}^n\frac{g_k(x)}{f_k(x)}\ind_{E_k}(x),
\]
which is measurable. The assertion follows.
\end{proof}

\begin{remark}
Denoting the standard basis of $\F^n$ by $(e_k)_{k=1}^n$, the map
\[
\iota: L^0(\Omega\times\{1,\ldots,n\})\to L^0(\Omega;\F^n),\quad f\mapsto\sum_{k=1}^nf(\cdot,k)e_k,
\]
gives a natural one-to-one correspondence between $L^0(\Omega;\F^n)$ and $L^0(\Omega\times\{1,\ldots,n\})$, where $\{1,\ldots,n\}$ is equipped with the counting measure, and $\Omega\times\{1,\ldots,n\}$ with the product measure. One could be tempted to define $f\leq g$ on $L^0(\Omega;\F^n)$ by asking for the component-wise ordering
\begin{equation}\label{eq:wrongordering}
|f(x)\cdot e_k|\leq |g(x)\cdot e_k|
\end{equation}
for all $k\in\{1,\ldots,n\}$ for a.e. $x\in\Omega$. By Proposition~\ref{prop:ideal}, this partial ordering is weaker than the one we assume. The biggest advantage of ordering vectors the way that we have, is that if $\mc{K}(u)\subseteq\mc{K}(v)$, then $|Au|\leq|Av|$ for any matrix $A\in\F^{n\times n}$. This fails under \eqref{eq:wrongordering}.
\end{remark}

We can extend the set operations on $\mc{K}$ to $L^0(\Omega;\mc{K})$ pointwise. That is, for $F,G\in L^0(\Omega;\mc{K})$ and $\lambda\in\F$ we can define $F+G$ and $\lambda F$ for a.e. $x\in\Omega$ by
\[
(F+G)(x):=\overline{\{u+v:u\in F(x),\, v\in G(x)\}},\quad (\lambda F)(x):=\lambda F(x).
\]
\begin{definition}
Let $F\in L^0(\Omega;\mc{K})$. We define the space of measurable selections of $F$ by
\[
S^0(\Omega;F):=\{f\in L^0(\Omega;\F^n):f(x)\in F(x)\text{ for a.e. $x\in\Omega$}\}.
\]
\end{definition}
\begin{proposition}\label{prop:closedconvexsections}
Let $F\in L^0(\Omega;\mc{K})$. Then the set $S^0(\Omega;F)$ is a non-empty, convex, and symmetric subset of $L^0(\Omega;\F^n)$ that is closed with respect to convergence in measure.
\end{proposition}
\begin{proof}
To see that $S^0(\Omega;F)$ is non-empty, note that it contains the $0$ selection $f=0$. For convexity, if $f,g\in S^0(\Omega;F)$, then for a.e. $x\in\Omega$ and all $0\leq t\leq 1$ we have $(1-t)f(x)+tg(x)\in F(x)$, since $F(x)$ is convex. Thus, $(1-t)f+t g\in S^0(\Omega;F)$. Symmetry is proven analogously.

To see that $S^0(\Omega;F)$ is closed, let $(f_k)_{k\geq 1}$ be a sequence in $S^0(\Omega;F)$ that converges in measure to a function $f\in L^0(\Omega;\F^n)$. Then there is a subsequence $(f_{k_j})_{j\geq 1}$ that converges pointwise a.e. to $f$. Hence, for a.e. $x\in\Omega$,
\[
f(x)=\lim_{j\to\infty}f_{k_j}(x)\in\overline{F(x)}=F(x),
\]
since $F(x)$ is closed. Thus, $f\in S^0(\Omega;F)$, proving the assertion.
\end{proof}

\begin{definition}
We define $L^1(\Omega;\mc{K})$ as the space of $F\in L^0(\Omega;\mc{K})$ for which $S^0(\Omega;F)$ is a bounded subset of $L^1(\Omega;\F^n)$, and we write
\[
\|F\|_{L^1(\Omega;\mc{K})}:=\sup_{f\in S^0(\Omega;F)}\|f\|_{L^1(\Omega;\F^n)}.
\]
For $F\in L^1(\Omega;\mc{K})$ we define the \emph{Aumann integral} of $F$ as
\[
\int_\Omega\!F\,\mathrm{d}\mu:=\Big\{\int_\Omega\!f\,\mathrm{d}\mu:f\in S^0(\Omega;F)\Big\}.
\]
\end{definition}
Since $L^1(\Omega;\mc{K})$ is not a vector space due to the fact that Minkowski addition does not have an additional inverse, $\|\cdot\|_{L^1(\Omega;\mc{K})}$ is not a norm in the conventional sense. However, it does satisfy the property that if $F,G\in L^1(\Omega;\mc{K})$, then $F+G\in L^1(\Omega;\mc{K})$ with
\[
\|F+G\|_{L^1(\Omega;\mc{K})}\leq\|F\|_{L^1(\Omega;\mc{K})}+\|G\|_{L^1(\Omega;\mc{K})}
\]
We give proofs and further elaborations on these properties in Subsections~\ref{subsec:XK} and \ref{subsec:sumsinXK}.

If $F\in L^1(\Omega;\mc{K})$ then, since the embedding $L^1(\Omega;\F^n)\subseteq L^0(\Omega;\F^n)$ is continuous when $L^0(\Omega;\F^n)$ is endowed with the topology of convergence in measure and $L^1(\Omega;\F^n)$ with its usual norm topology, it follows from Proposition~\ref{prop:closedconvexsections} that $S^0(\Omega;F)$ is non-empty, convex, symmetric, and a closed subset in $L^1(\Omega;\F^n)$ with respect to the norm topology.

We say that $F$ is integrably bounded if there exists a $0\leq k\in L^1(\Omega)$ for which for a.e. $x\in\Omega$ we have
\[
F(x)\subseteq\{k(x)u\in\F^n :|u|\leq 1\}.
\]
It is shown in \cite[Theorem~3.15]{BC23} that if $F$ is integrably bounded, then $\int_\Omega\!F\,\mathrm{d}\mu\in\mc{K}_b$. This is the case for any $F\in L^1(\Omega;\mc{K})$:
\begin{proposition}\label{prop:integrablybounded}
Let $F\in L^0(\Omega;\mc{K})$. Then the following are equivalent:
\begin{enumerate}[(i)]
\item\label{it:intbdd2} $F\in L^1(\Omega;\mc{K})$;
    \item\label{it:intbdd1} $F$ is integrably bounded;
    \item\label{it:intbdd3} $h\in L^1(\Omega)$, where $h(x):=\sup_{u\in F(x)}|u|$.
\end{enumerate}
Moreover, in this case we have $\|F\|_{L^1(\Omega;\mc{K})}=\|h\|_{L^1(\Omega)}$.
\end{proposition}
Note that the function $h$ is measurable, as it satisfies $h(x)=\sup_{k\geq 1}|f_k(x)|$ for any sequence $(f_k)_{k\geq 1}$ in $L^0(\Omega;\F^n)$ satisfying
\[
F(x)=\overline{\{f_k(x):k\geq 1\}}.
\]
For the proof, we will need \cite[Lemma~3.9]{BC23, KNV24}:
\begin{lemma}\label{lem:measurablesupremumconvexset}
Let $F\in L^0(\Omega;\mc{K}_b)$. Then there is a $f_0\in S^0(\Omega;F)$ for which for a.e. $x\in\Omega$ we have $|f_0(x)|=\sup_{u\in F(x)}|u|$.
\end{lemma}
\begin{proof}[Proof of Proposition~\ref{prop:integrablybounded}]
For \ref{it:intbdd3}$\Rightarrow$\ref{it:intbdd1}, note that for any $v\in F(x)$ we have
\[
|v|\leq h(x).
\]
Hence, for $u:=h(x)^{-1}v$ if $h(x)\neq 0$ and $u:=0$ if $h(x)=0$, we have $|u|\leq 1$ and $v=h(x)u$. We conclude that
\[
F(x)\subseteq\{h(x)u\in\F^n :|u|\leq 1\}.
\]
As $h\in L^1(\Omega)$, this proves the assertion.

For \ref{it:intbdd1}$\Rightarrow$\ref{it:intbdd2}, note that for any $f\in S^0(\Omega;F)$ we have $|f(x)|\leq k(x)$, where $k$ is the function in the definition of integrably bounded. Hence, $f\in L^1(\Omega;\F^n)$ with $\|f\|_{L^1(\Omega;\F^n)}\leq\|k\|_{L^1(\Omega)}$. We conclude that $F\in L^1(\Omega;\mc{K})$ with
\[
\|F\|_{L^1(\Omega;\mc{K})}\leq\|k\|_{L^1(\Omega)},
\]
as desired.

To prove \ref{it:intbdd2}$\Rightarrow$\ref{it:intbdd3}, define
\[
F_k(x):=\{u\in F(x):|u|\leq k\}.
\]
Then $F_k\in L^0(\Omega;\mc{K}_b)$, so by Lemma~\ref{lem:measurablesupremumconvexset} there is a selection $f_k\in S^0(\Omega;F_k)$ for which
\[
h_k(x):=|f_k(x)|=\sup_{u\in F_k(x)}|u|.
\]
Note that $h_k(x)\uparrow h(x):=\sup_{u\in F(x)}|u|$ a.e., and that
\[
\sup_{k\geq 1}\|h_k\|_{L^1(\Omega)}=\sup_{k\geq 1}\|f_k\|_{L^1(\Omega;\F^n)}\leq\|F\|_{L^1(\Omega;\mc{K})}<\infty.
\]
Hence, by the monotone convergence theorem, $h\in L^1(\Omega)$ with
\[
\|h\|_{L^1(\Omega)}=\sup_{k\geq 1}\|h_k\|_{L^1(\Omega)}\leq\|F\|_{L^1(\Omega;\mc{K})}.
\]
The result follows.
\end{proof}

\begin{corollary}\label{cor:integrablybounded}
Let $F\in L^1(\Omega;\mc{K})$. Then there is an $f_0\in S^0(\Omega;F)$ such that
\[
\|F\|_{L^1(\Omega,\mc{K})}=\|f_0\|_{L^1(\Omega;\F^n)}.
\]
\end{corollary}
\begin{proof}
By Proposition~\ref{prop:integrablybounded} the function $h(x)=\sup_{u\in F(x)}|u|$ lies in $L^1(\Omega)$. This implies that $|h(x)|<\infty$ a.e., and, hence $F\in L^0(\Omega;\mc{K}_b)$. If we let $f_0$ be as in Lemma~\ref{lem:measurablesupremumconvexset}. Then, by Proposition~\ref{prop:integrablybounded},
\[
\|F\|_{L^1(\Omega;\mc{K})}=\|h\|_{L^1(\Omega)}=\|f_0\|_{L^1(\Omega;\F^n)}.
\]
This proves the result.
\end{proof}

\section{Directional quasi-Banach function spaces}\label{sec:directionalqbfs}
We say that $X$ is a \emph{quasi-Banach function space over $\Omega$} if it is a complete quasi-normed vector space $X\subseteq L^0(\Omega)$ that satisfies:
\begin{itemize}
    \item \emph{The ideal property}: for all $f\in X$ and $g\in L^0(\Omega)$ with $|g|\leq |f|$ a.e. we have $g\in X$ with $\|g\|_X\leq\|f\|_X$;
\item\emph{The saturation property}: for every $E\subseteq\Omega$ with $\mu(E)>0$ there is an $F\subseteq E$ with $\mu(F)>0$ and $\ind_F\in X$. 
\end{itemize}
We let $K_X$ denote the optimal constant in the quasi-triangle inequality
\[
\|f+g\|_X\leq K_X(\|f\|_X+\|g\|_X).
\]
If $K_X=1$, i.e., the quasi-norm of $X$ is a norm, then we call $X$ a Banach function space over $\Omega$. The saturation property is equivalent to the property that the seminorm
\[
\|g\|_{X'}:=\sup_{\|f\|_X=1}\int_\Omega\!|fg|\,\mathrm{d}x
\]
is a norm, or to the existence of a \emph{weak order unit}, i.e., a function $\rho\in X$ satisfying $\rho(x)>0$ for a.e. $x\in\Omega$. This result as well as an overview of this property can be found in \cite{LN23b}. In particular, as it is used several times throughout this text, we record part of \cite[Proposition~2.5]{LN23b} here:
\begin{proposition}\label{prop:weakorderunit}
Let $X\subseteq L^0(\Omega)$ be a quasi-Banach space satisfying the ideal property. Then the following are equivalent:
\begin{enumerate}[(i)]
    \item $X$ satisfies the saturation property;
    \item $X$ has a weak order unit, i.e., there is a function $\rho\in X$ satisfying $\rho>0$ a.e.;
    \item If $g\in L^0(\Omega)$ satisfies $\int_\Omega fg\,\mathrm{d}\mu=0$ for all $f\in X$, then $g=0$ a.e.
\end{enumerate}
\end{proposition}

\begin{definition}
We say that $\mb{X}$ is an $\F^n$-directional quasi-Banach function space over $\Omega$ if it is a complete quasi-normed subspace of $L^0(\Omega;\F^n)$ and satisfies the following properties:
\begin{itemize}
    \item \emph{The directional ideal property:} For all $f\in\mb{X}$ and $g\in L^0(\Omega;\F^n)$ satisfying $\mc{K}(g)(x)\subseteq\mc{K}(f)(x)$ a.e., we have $g\in\mb{X}$ with $\|g\|_{\mb{X}}\leq\|f\|_{\mb{X}}$;
    \item\emph{Non-degeneracy:} If $g\in L^0(\Omega;\F^n)$ satisfies $\int_\Omega\!f\cdot g\,\mathrm{d}\mu=0$ for all $f\in\mb{X}$, then $g=0$.
\end{itemize}
\end{definition}
We let $K_{\mb{X}}$ denote the optimal constant in the quasi-triangle inequality. If $K_{\mb{X}}=1$, i.e., the quasi-norm of $\mb{X}$ is a norm, then we call $\mb{X}$ an $\F^n$-directional Banach function space over $\Omega$. Completeness is equivalent to the Riesz-Fischer property, i.e., for every $(f_k)_{k\geq 1}$ for which $C:=\sum_{k=1}^\infty K_{\mb{X}}^k\|f_k\|_{\mb{X}}<\infty$, the partial sums $\sum_{k=1}^K f_k$ have a limit $f\in\mb{X}$ with $\|f\|_{\mb{X}}\leq K_X C$.

For verifying the directional ideal property, one can use any of the equivalences of Proposition~\ref{prop:ideal}.

Given an $\F^n$-directional quasi-Banach function space $\mb{X}$ over $\Omega$, we define $\mb{X}'$ as the space of $g\in L^0(\Omega;\F^n)$ for which
\[
\|g\|_{\mb{X}'}:=\sup_{\|f\|_{\mb{X}}=1}\int_\Omega\!|f\cdot g|\,\mathrm{d}\mu<\infty.
\]
The non-degeneracy property is equivalent to the assertion that the seminorm $\|\cdot\|_{\mb{X}'}$ is a norm. This follows from the following proposition:
\begin{proposition}\label{prop:normoutsideintegral}
Let $\mb{X}$ be an $\F^n$-directional quasi-Banach function space over $\Omega$. Then we have $g\in\mb{X}'$ if and only if there is a $C\geq 0$ such that for all $f\in\mb{X}$ we have
\[
\Big|\int_\Omega\!f\cdot g\,\mathrm{d}\mu\Big|\leq C\|f\|_{\mb{X}}.
\]
Moreover, in this case the smallest possible $C$ satisfies $\|g\|_{\mb{X}'}=C$.
\end{proposition}
\begin{proof}
If $g\in\mb{X}'$, we have
\[
\Big|\int_\Omega\!f\cdot g\,\mathrm{d}\mu\Big|\leq\int_\Omega\!|f\cdot g|\,\mathrm{d}\mu\leq\|g\|_{\mb{X}'}\|f\|_{\mb{X}},
\]
so it remains to prove the converse. Let $f\in\mb{X}$ and define 
\[
\widetilde{f}:=\ind_{\supp(f\cdot g)}\frac{|f\cdot g|}{f\cdot g}f.
\]
Then, for all $v\in\F^n$, we have $|\widetilde{f}(x)\cdot v|\leq |f(x)\cdot v|$, so that, by the directional ideal property of $\mb{X}$, we have $\widetilde{f}\in\mb{X}$ with $\|\widetilde{f}\|_{\mb{X}}\leq\|f\|_{\mb{X}}$. Hence,
\[
\int_\Omega\!|f\cdot g|\,\mathrm{d}\mu=\Big|\int_\Omega\!\widetilde{f}\cdot g\,\mathrm{d}\mu\Big|\leq C\|\widetilde{f}\|_{\mb{X}}\leq C\|f\|_{\mb{X}}.
\]
We conclude that $g\in\mb{X}'$, with $\|g\|_{\mb{X}'}\leq C$. The assertion follows.
\end{proof}

We prove several characterizations of non-degeneracy in Proposition~\ref{prop:saturation} in Appendix~\ref{app:A}.

Given a complete quasi-normed subspace $\mb{X}\subseteq L^0(\Omega;\F^n)$ and $g\in L^0(\Omega;\F^n)$, we can define a space of scalar-valued functions $\mb{X}_g$ as the space of functions $h\in L^0(\supp(g))$ for which $hg\in\mb{X}$, and set
\[
\|h\|_{\mb{X}_g}:=\|hg\|_{\mb{X}}.
\]
The saturation property of the spaces $\mb{X}_g$ for various collections of functions $g$ can be used to define different notions of saturation for the space $\mb{X}$.

We define the \emph{directional saturation property} of $\mb{X}$ as follows:
\begin{itemize}
    \item\emph{Directional saturation:} For all non-zero $g\in L^0(\Omega;\F^n)$ there is a measurable set $E\subseteq\text{supp}(g)$ with $\mu(E)>0$ such that $\ind_E g\in\mb{X}$.
\end{itemize}
We also define a component-wise saturation property of $\mb{X}$ by specializing to the constant functions $g=e_k$ for $k\in\{1,\ldots,n\}$, where $(e_k)_{k=1}^n$ is the standard basis of $\F^n$:
\begin{itemize}
    \item \emph{Component-wise saturation:} The space $\mb{X}_{e_k}$ is saturated for all $k\in\{1,\ldots,n\}$.
\end{itemize}
The component-wise saturation property implies non-degeneracy, since
\[
\int_\Omega\!f\cdot g\,\mathrm{d}\mu=0
\]
for all $f\in\mb{X}$ implies that
\[
\int_\Omega\!h(e_k\cdot g)\,\mathrm{d}\mu=0
\]
for all $h\in \mb{X}_{e_k}$. Thus, by Proposition~\ref{prop:weakorderunit}, we find that $g=\sum_{k=1}^n(g\cdot e_k)e_k=0$ a.e., as desired.

When $n=1$, directional saturation, component-wise saturation, and non-degeneracy coincide. However, for $n>1$, the directional saturation property is stronger. We have the following characterizations: 
\begin{proposition}\label{prop:directionalsaturation}
Let $\mb{X}$ be an $\F^n$-directional quasi-Banach function space over $\Omega$. The following are equivalent:
\begin{enumerate}[(i)]
    \item\label{it:dirsat1} $\mb{X}$ satisfies the directional saturation property;
    \item\label{it:dirsat2} $\mb{X}_g$ is a quasi-Banach function space over $\supp(g)$ for all non-zero $g\in L^0(\Omega;\F^n)$;
    \item\label{it:dirsat3} For all $g\in L^0(\Omega;\F^n)$ there is a sequence $(f_k)_{k\geq 1}$ in $\mb{X}$ for which $|f_k(x)\cdot u|\uparrow |g(x)\cdot u|$ for all $u\in\F^n$ for a.e. $x\in\Omega$.
\end{enumerate}
\end{proposition}
\begin{proof}

For \ref{it:dirsat1}$\Rightarrow$\ref{it:dirsat2}, let $g\in L^0(\Omega;\F^n)$ be non-zero and let $E\subseteq\supp(g)$ with $\mu(E)>0$. Then $\ind_E g$ is non-zero, so there is an $F\subseteq E$ with $\mu(F)>0$ and $\ind_Fg=\ind_F\ind_Eg\in\mb{X}$. Thus, $\ind_F\in\mb{X}_g$, proving that $\mb{X}_g$ is saturated and, hence, a quasi-Banach function space over $\supp(g)$.

To see \ref{it:dirsat2}$\Rightarrow$\ref{it:dirsat3}, for $g\in L^0(\Omega;\F^n)$ it follows from \cite[Proposition~2.5(ii)]{LN23b} that there is an increasing sequence of sets $(E_k)_{k\geq 1}$ with $\bigcup_{k=1}^\infty E_k=\supp(g)$ and $\ind_{E_k}\in\mb{X}_g$ for all $k\geq 1$. The result then follows by setting $f_k:=\ind_{E_k}g\in\mb{X}$ and noting that $|f_k\cdot u|=\ind_{E_k}|g\cdot u|\uparrow|g\cdot u|$ for all $u\in\F^n$ a.e.

Finally, for \ref{it:dirsat3}$\Rightarrow$\ref{it:dirsat1}, let $g\in L^0(\Omega;\F^n)$ be non-zero, and pick $(f_k)_{k\geq 1}$ as in \ref{it:dirsat3}. Then, by Proposition~\ref{prop:ideal}\ref{it:ideal4} and \ref{it:dirsat3}, there is a sequence $(h_k)_{k\geq 1}$ in $L^\infty(\Omega)$ with $0\leq |h_k|\uparrow \ind_{\supp(g)}$ a.e. and $f_k=h_k g$. This means that for $K$ large enough, the set $E:=\{|h_K|>\tfrac{1}{2}\}$ has positive measure. Since
\[
|\ind_E g\cdot u|\leq 2|h_K||g\cdot u|=|2f_K\cdot u|
\]
for all $u\in\F^n$ a.e., it follows from the directional ideal property of $\mb{X}$ and the equivalence of \ref{it:ideal1} and \ref{it:ideal3}
in Proposition~\ref{prop:ideal} that $\ind_E g\in\mb{X}$. The assertion follows.
\end{proof}

We also prove the following characterization of component-wise saturation:
\begin{proposition}\label{prop:componentsaturation}
Let $\mb{X}$ be an $\F^n$-directional quasi-Banach function space over $\Omega$. The following are equivalent:
\begin{enumerate}[(i)]
    \item\label{it:componentsaturation1} $\mb{X}$ satisfies the component-wise saturation property;
    \item\label{it:componentsaturation2} There is a measurable Hermitian and positive definite matrix-valued mapping 
    \[
    U:\Omega\to\F^{n\times n}
    \]
    such that if $g\in L^0(\Omega;\F^n)$ with $|U(x)^{-1}g(x)|\leq 1$ a.e., then $g\in\mb{X}$ with $\|g\|_{\mb{X}}\leq 1$;
    \item\label{it:componentsaturation3} For all non-zero $u\in\F^n$, the space $\mb{X}_u$ is saturated.
\end{enumerate}
\end{proposition}
\begin{proof}
For \ref{it:componentsaturation1}$\Rightarrow$\ref{it:componentsaturation2}, let $\rho_k>0$ be a weak order unit in $\mb{X}_{e_k}$ (which exists by Proposition~\ref{prop:weakorderunit}) with
\[
\|\rho_k\|_{\mb{X}_{e_k}}=\tfrac{1}{n}K_{\mb{X}}^{-k}
\]
for all $k\in\{1,\ldots,n\}$, and define $U(x)$ so that $U(x)e_k=\rho_k(x)e_k$. Then $U$ is a.e. Hermitian and positive definite. Moreover, if $|U(x)^{-1}g(x)|\leq 1$, then
\begin{align*}
g(x)&=\sum_{k=1}^n (g(x)\cdot e_k)e_k=\sum_{k=1}^n(U(x)^{-1}g(x)\cdot U(x)e_k)e_k\\
&=\sum_{k=1}^n(U(x)^{-1}g(x)\cdot e_k)\rho_k(x)e_k.
\end{align*}
Since $\rho_ke_k\in\mb{X}$, and $|U(x)^{-1}g(x)\cdot e_k|\leq 1$, it follows from the directional ideal property and the quasi-triangle inequality that $g\in\mb{X}$, with
\[
\|g\|_{\mb{X}}\leq\sum_{k=1}^n K_{\mb{X}}^k\|\rho_k\|_{\mb{X}_{e_k}}=1.
\]

For \ref{it:componentsaturation2}$\Rightarrow$\ref{it:componentsaturation3}, let $0\neq u\in\F^n$ and set 
\[
\rho_u(x):=|U(x)^{-1}u|^{-1}.
\]
Since $|U(x)^{-1}\rho_u(x) u|=\rho_u(x)|U(x)^{-1}u|=1$, we have $\rho_u u\in\mb{X}$. Since $\rho_u>0$ a.e., we conclude that $\mb{X}_u$ is saturated, as desired.

Finally, the implication \ref{it:componentsaturation3}$\Rightarrow$\ref{it:componentsaturation1} follows by specializing $u=e_k$, $k\in\{1,\ldots,n\}$.
\end{proof}

We say that an $\F^n$-directional quasi-Banach function space $\mb{X}$ over $\Omega$ satisfies the \emph{Fatou property} if:
\begin{itemize}
    \item For all $(f_k)_{k\geq 1}$ in $\mb{X}$ for which there is an $f\in L^0(\Omega;\F^n)$ such that $f_k\to f$ a.e. and $\liminf_{k\to\infty}\|f_k\|_{\mb{X}}<\infty$, we have $f\in\mb{X}$, with 
    \[
    \|f\|_{\mb{X}}\leq\liminf_{k\to\infty}\|f_k\|_{\mb{X}}.
    \]
\end{itemize}
We say that $\mb{X}$ satisfies the \emph{monotone convergence property} if:
\begin{itemize}
    \item For all $(f_k)_{k\geq 1}$ in $\mb{X}$ with $\mc{K}(f_k)\subseteq\mc{K}(f_{k+1})$ a.e. for all $k\geq 1$, $\sup_{k\geq 1}\|f_k\|_{\mb{X}}<\infty$, and
\[
\bigcup_{k=1}^\infty\mc{K}(f_k)=\mc{K}(f)
\]
for some $f\in L^0(\Omega;\F^n)$, we have $f\in\mb{X}$ with $\|f\|_{\mb{X}}=\sup_{k\geq 1}\|f_k\|_{\mb{X}}$.
\end{itemize}
Moreover, we say that $\mb{X}$ is \emph{K\"othe reflexive} if $\mb{X}''=\mb{X}$ with equal norm.

When $n=1$, the Fatou and monotone convergence properties coincide by \cite[Lemma~3.5]{LN23b} and, if $\mb{X}$ is a Banach function space, these notions coincide with K\"othe reflexivity by the Lorentz-Luxemburg theorem, see, e.g., \cite[Theorem~71.1]{Za67}. However, for $n>1$, the Fatou property is stronger than the monotone convergence property: the Fatou property is defined through pointwise a.e. convergence and, hence, this includes sequences for which $f_k$ and $f$ do not share a direction at any point.

Note that for any $\F^n$-directional quasi-Banach function space $\mb{X}$ over $\Omega$, the space $\mb{X}'$ satisfies the Fatou property by Fatou's lemma of integration theory. 

The following result is an $\F^n$-directional analogue of the Lorentz-Luxemburg theorem:
\begin{theorem}[The directional Lorentz-Luxemburg theorem]\label{thm:lorentzluxemburg}
Let $\mb{X}$ be an $\F^n$-directional Banach function space over $\Omega$. Then consider the statements:
\begin{enumerate}[(a)]
    \item\label{it:lorentzluxemburga} $\mb{X}$ satisfies the Fatou property;
    \item\label{it:lorentzluxemburgb} $\mb{X}$ is K\"othe reflexive.
\end{enumerate}
Then \ref{it:lorentzluxemburgb}$\Rightarrow$\ref{it:lorentzluxemburga}. If $\mb{X}$ satisfies the directional saturation property, then also \ref{it:lorentzluxemburga}$\Rightarrow$\ref{it:lorentzluxemburgb}.
\end{theorem}
The proof of this result can be found below as Theorem~\ref{thm:lorentzluxemburgapp} in Appendix~\ref{app:A}.

\subsection{Convex-set valued quasi-Banach function spaces}\label{subsec:XK}
Let $\mb{X}$ be an $\F^n$-directional quasi-Banach function space over $\Omega$. We define $\mb{X}[\mc{K}]$ as the space of $F\in L^0(\Omega;\mc{K})$ for which $S^0(\Omega;F)$ is a bounded subset of $\mb{X}$. Moreover, we set
\[
\|F\|_{\mb{X}[\mc{K}]}:=\sup_{f\in S^0(\Omega;F)}\|f\|_{\mb{X}}.
\]
The space $\mb{X}[\mc{K}]$ satisfies the ideal property in the sense that if $F\in\mb{X}[\mc{K}]$, then for any $G\in L^0(\Omega;\mc{K})$ for which 
\[
G(x)\subseteq F(x)
\]
for a.e. $x\in\Omega$, we have $G\in\mb{X}[\mc{K}]$ with $\|G\|_{\mb{X}[\mc{K}]}\leq\|F\|_{\mb{X}[\mc{K}]}$. This follows from the observation that $S^0(\Omega;G)\subseteq S^0(\Omega;F)$.

The space $\mb{X}[\mc{K}]$ naturally contains $\mb{X}$ through the embedding $f\mapsto\mc{K}(f)$. Indeed, we have the following result:
\begin{proposition}\label{prop:convexsetfequalsf}
Let $\mb{X}$ be an $\F^n$-directional quasi-Banach function space over $\Omega$ and $f\in L^0(\Omega;\F^n)$. Then $f\in\mb{X}$ if and only if $\mc{K}(f)\in\mb{X}[\mc{K}]$, with
\[
\|\mc{K}(f)\|_{\mb{X}[\mc{K}]}=\|f\|_{\mb{X}}.
\]
\end{proposition}
\begin{proof}
If $\mc{K}(f)\in\mb{X}[\mc{K}]$, then, since $f\in S^0(\Omega;\mc{K}(f))$, we have $f\in\mb{X}$ with
\[
\|f\|_{\mb{X}}\leq\sup_{g\in S^0(\Omega;\mc{K}(f))}\|g\|_{\mb{X}}=\|\mc{K}(f)\|_{\mb{X}[\mc{K}]}.
\]
Conversely, if $f\in\mb{X}$, then for any $g\in S^0(\Omega;\mc{K}(f))$ we have $g(x)\in\mc{K}(f)(x)$ for a.e. $x\in\Omega$. Thus, by Proposition~\ref{prop:ideal} and the directional ideal property of $\mb{X}$, we have $g\in\mb{X}$ with $\|g\|_{\mb{X}}\leq\|f\|_{\mb{X}}$. Taking a supremum over all $g\in S^0(\Omega;\mc{K}(f))$, we conclude that $\mc{K}(f)\in\mb{X}[\mc{K}]$ with
\[
\|\mc{K}(f)\|_{\mb{X}[\mc{K}]}=\sup_{g\in S^0(\Omega;\mc{K}(f))}\|g\|_{\mb{X}}\leq\|f\|_{\mb{X}}.
\]
The assertion follows.
\end{proof}

Given an $\F^n$-directional quasi-Banach function space $\mb{X}$ over $\Omega$, for a sequence $(F_k)_{k\geq 1}$ in $\mb{X}[\mc{K}]$ we write $F_k\uparrow F$, if $F_k(x)\subseteq F_{k+1}(x)$ a.e. for all $k\geq 1$, and if
\[
F(x)=\overline{\bigcup_{k=1}^\infty F_k(x)}.
\]
Note that $F\in L^0(\Omega;\mc{K})$. Indeed, for any open $U\subseteq\F^n$ we have
\[
F^{-1}(U)=\bigcup_{k=1}^\infty F_k^{-1}(U),
\]
which is measurable by measurability of the $F_k$.
\begin{proposition}\label{prop:convexsetmonotoneconvergence}
Let $\mb{X}$ be an $\F^n$-directional quasi-Banach function space $\mb{X}$ over $\Omega$ with the monotone convergence property. If $(F_k)_{k\geq 1}$ is a bounded sequence in $\mb{X}[\mc{K}]$ and $F_k\uparrow F$, then $F\in\mb{X}[\mc{K}]$ with
\[
\|F\|_{\mb{X}[\mc{K}]}=\sup_{k\geq 1}\|F_k\|_{\mb{X}[\mc{K}]}.
\]
\end{proposition}
\begin{proof}
Let $f\in S^0(\Omega;F)$. Since $\mc{K}(f)\cap F_k\in L^0(\Omega;\mc{K}_b)$, by Lemma~\ref{lem:measurablesupremumconvexset} there is an $f_k\in S^0(\Omega;F_k)$ for which $|f_k(x)|=\sup_{u\in \mc{K}(f)(x)\cap F_k(x)}|u|$. Note that
\[
\mc{K}(f_k)=\mc{K}(f)\cap F_k\uparrow \mc{K}(f)\cap F=\mc{K}(f).
\]
Hence, by the monotone convergence property of $\mb{X}$, $f\in\mb{X}$ with
\[
\|f\|_{\mb{X}}=\sup_{k\geq 1}\|f_k\|_{\mb{X}}\leq\sup_{k\geq1}\|F_k\|_{\mb{X}[\mc{K}]}.
\]
We conclude that $F\in\mb{X}[\mc{K}]$, with
\[
\|F\|_{\mb{X}[\mc{K}]}\leq \sup_{k\geq 1}\|F_k\|_{\mb{X}[\mc{K}]}.
\]
This proves the assertion.
\end{proof}

\subsection{Matrix weights}\label{sec:matrixweights}
A measurable mapping $W:\Omega\to \F^{n\times n}$ is called a \emph{matrix weight} if for a.e. $x\in\Omega$ the matrix $W(x)$ is Hermitian and positive definite, i.e., for all non-zero $u\in\F^n$ we have $W(x)u\cdot u>0$.

Given a matrix weight $W:\Omega\to\F^{n\times n}$ and a quasi-Banach function space $X$ over $\Omega$, we define the $\F^n$-directional quasi-Banach function space $X_W$ as the space of functions $f\in L^0(\Omega:\F^n)$ for which
\[
|Wf|\in X,
\]
and set
\[
\|f\|_{X_W}:=\||Wf|\|_X.
\]
Then $\|\cdot\|_{X_W}$ is indeed a quasi-norm, since $\|f\|_{X_W}=0$ is equivalent to $Wf=0$ a.e., which is equivalent to $f=W^{-1}Wf=0$ a.e. To see that $X_W$ satisfies the directional ideal property, suppose that $f\in X_W$ and $g\in L^0(\Omega;\F^n)$ satisfies $\mc{K}(g)\subseteq\mc{K}(f)$ a.e. Then, by Proposition~\ref{prop:ideal}, there is a $h\in L^\infty(\Omega)$ with $\|h\|_{L^\infty(\Omega)}\leq 1$ such that $g=h f$. Hence,
\[
|Wg|=|h||Wf|\leq|Wf|
\]
a.e., so that by the ideal property of $X$ we have $g\in X_W$, with 
\[
\|g\|_{X_W}=\||Wg|\|_X\leq\||Wf|\|_X=\|f\|_{X_W}.
\]
The non-degeneracy property follows from the fact that $(X_W)'=(X')_{W^{-1}}$:
\begin{proposition}\label{prop:matrixweightedspaceduality}
Let $W:\Omega\to\F^{n\times n}$ be a matrix weight, and let $X$ be a quasi-Banach function space over $\Omega$. Then $(X_W)'=(X')_{W^{-1}}$.
\end{proposition}
In particular, note that if $X$ is a Banach function space with the Fatou property, then $(X_W)''=(X'')_W=X_W$ by the Lorentz-Luxemburg theorem.
\begin{proof}
First, suppose that $g\in (X')_{W^{-1}}$. Then for all $f\in X_W$ we have
\begin{align*}
\int_\Omega\!|f\cdot g|\,\mathrm{d}\mu
&=\int_\Omega\!|Wf\cdot W^{-1}g|\,\mathrm{d}\mu
\leq \int_\Omega\!|Wf||W^{-1}g|\,\mathrm{d}\mu\\
&\leq\||Wf|\|_X\||W^{-1}g\|_{X'}=\|f\|_{X_W}\|g\|_{(X')_{W^{-1}}}.
\end{align*}
Hence, $g\in (X_W)'$ with $\|g\|_{(X_W)'}\leq\|g\|_{(X')_{W^{-1}}}$. Conversely, suppose $g\in (X_W)'$. Let $k\in X$ and define 
\[
h:=\begin{cases}
    k\frac{W^{-1}g}{|W^{-1}g|} & \text{if $W^{-1}g\neq 0$;}\\
    0 & \text{if $W^{-1}g=0$}.
\end{cases}
\]
Then $|h|\leq |k|$ a.e., so $|h|\in X$ by the ideal property of $X$. Setting $f:=W^{-1}h\in X_W$, we have 
\[
\|f\|_{X_W}=\||h|\|_X\leq\|k\|_X.
\]
Hence,
\begin{align*}
\int_\Omega\!|k||W^{-1}g|\,\mathrm{d}\mu&=\int_\Omega\!|h\cdot W^{-1}g|\,\mathrm{d}\mu=\int_\Omega\!|f\cdot g|\,\mathrm{d}\mu\\
&\leq\|f\|_{X_W}\|g\|_{(X_W)'}\leq\|k\|_X\|g\|_{(X_W)'}.
\end{align*}
Thus, $g\in (X')_{W^{-1}}$ with $\|g\|_{(X')_{W^{-1}}}\leq\|g\|_{(X_W)'}$. The result follows.
\end{proof}
As a matter of fact, $X_W$ satisfies the directional saturation property. Indeed, for any non-zero $g\in L^0(\Omega;\F^n)$ we have
\[
\|h\|_{(X_W)_g}=\|hg\|_{X_W}=\|h|Wg|\|_X=\|h\|_{X(|Wg|)},
\]
showing that $(X_W)_g=X(|Wg|)$. Here, for a function $u\geq 0$, we define $X(u)$ through $\|h\|_{X(u)}:=\|hu\|_X$. Since $|Wg|>0$ a.e. on $\supp(g)$, the space $(X_W)_g$ is saturated by \cite[Proposition~3.17]{LN23b}. The directional saturation property of $X_W$ now follows from Proposition~\ref{prop:directionalsaturation}.

A very useful result for matrix weighted spaces is Filippov's selection theorem. The following version can be found in \cite[Theorem~8.2.10]{AF09}:
\begin{theorem}[Filippov]\label{thm:filippov}
Let $F\in L^0(\Omega;\mc{C})$ and $\phi:\Omega\times\F^n\to\R$ a function for which for all $u\in\F^n$ the function $x\mapsto\phi(x,u)$ is measurable and for a.e. $x\in\Omega$ the function $u\mapsto\phi(x,u)$ is continuous. For each $h\in L^0(\Omega)$ satisfying
\[
h(x)\in\{\phi(x,u):u\in F(x)\}
\]
for a.e. $x\in\Omega$, there is a selection $f\in S^0(\Omega;F)$ for which 
\[
h(x)=\phi(x,f(x))
\]
for a.e. $x\in\Omega$.
\end{theorem}

As a consequence, we can show that if $X$ satisfies the Fatou property, then any $F\in X_W[\mc{K}]$ satisfies $F(x)\in\mc{K}_b$ a.e.
\begin{proposition}\label{prop:matrixweightfatouimpliesbounded}
Let $X$ be a quasi-Banach function space over $\Omega$ with the Fatou property, and let $W:\Omega\to\F^{n\times n}$ be a matrix weight. Then $X_W[\mc{K}]\subseteq L^0(\Omega;\mc{K}_b)$.    
\end{proposition}
\begin{proof}
Let $F\in X_W[\mc{K}]$ and define
\[
F_k(x):=\{u\in F(x):|u|\leq k\}.
\]
Setting
\[
\phi(x,u):=|W(x)u|,\quad h_k(x)=\sup_{u\in F_k(x)}\phi(x,u),
\]
we note that since for a given $x\in\Omega$ the function $u\mapsto\phi(x,u)$ is continuous, and the set $F_k(x)$ is compact in $\F^n$, the supremum defining $h_k(x)$ is achieved and, therefore, belongs to $\{\phi(x,u):u\in F_k(x)\}$. Thus, by Theorem~\ref{thm:filippov}, there is a selection $f_k\in S^0(\Omega;F_k)$ such that
\[
h_k(x)=|W(x)f_k(x)|=\sup_{u\in F_k(x)}|W(x)u|.
\]
Since
\[
\sup_{k\geq1}\|h_k\|_X=\sup_{k\geq 1}\|f_k\|_{X_W}\leq\|F\|_{X_W[\mc{K}]},
\]
it follows from the Fatou property of $X$ that $h(x):=\sup_{k\geq 1}h_k(x)=\sup_{u\in F(x)}|W(x)u|$ satisfies $h\in X$ with
\[
\|h\|_X=\sup_{k\geq 1}\|h_k\|_X\leq \|F\|_{X_W[\mc{K}]}.
\]
This implies that $h(x)<\infty$ a.e. and, hence, 
\[
\sup_{u\in F(x)}|u|\leq|W(x)^{-1}|h(x)<\infty
\]
a.e., as desired.
\end{proof}
We can also prove the following result:
\begin{proposition}\label{prop:supattainmentXW}
Let $X$ be a quasi-Banach function space over $\Omega$, let $W:\Omega\to\F^{n\times n}$ be a matrix weight, and let $F\in L^0(\Omega;\mc{K}_b)$. Then the following are equivalent:
\begin{enumerate}[(i)]
    \item\label{it:supattainmentXW1} $F\in X_W[\mc{K}]$;
    \item\label{it:supattainmentXW2} $h\in X$, where $h(x):=\sup_{u\in F(x)}|W(x)u|$.
\end{enumerate}
Moreover, in this case there is an $f_0\in S^0(\Omega; F)$ such that
\[
\|F\|_{X_W[\mc{K}]}=\|f_0\|_{X_W}=\|h\|_X.
\]
\end{proposition}
Note that if $X$ satisfies the Fatou property, then any $F\in X_W[\mc{K}]$ satisfies $F\in L^0(\Omega;\mc{K}_b)$ by Proposition~\ref{prop:matrixweightfatouimpliesbounded}.
\begin{proof}[Proof of Proposition~\ref{prop:supattainmentXW}]
For \ref{it:supattainmentXW1}$\Rightarrow$\ref{it:supattainmentXW2}, we may use an analogous reasoning to the one in the proof of Proposition~\ref{prop:matrixweightfatouimpliesbounded} to apply Theorem~\ref{thm:filippov} to
\[
\phi(x,u)=|W(x)u|,\quad h(x)=\sup_{u\in F(x)}\phi(x,u),
\]
to find an $f_0\in S^0(\Omega;F)$ for which
\[
|W(x)f_0(x)|=\sup_{u\in F(x)}|W(x)u|
\]
a.e. Then $f_0\in X_W$, so $h=|Wf_0|\in X$, with
\[
\|h\|_X=\|f_0\|_{X_W}\leq\|F\|_{X_W[\mc{K}]},
\]
as desired. Conversely, for \ref{it:supattainmentXW2}$\Rightarrow$\ref{it:supattainmentXW1}, if $f\in S^0(\Omega;F)$, then $|Wf|\leq h$ a.e. Hence, by the ideal property of $X$, we have $f\in X_W$ with $\|f\|_{X_W}=\||Wf|\|_X\leq\|h\|_X$. Thus, $F\in X_W[\mc{K}]$ with
\[
\|F\|_{X_W[\mc{K}]}=\sup_{f\in S^0(\Omega;F)}\|f\|_{X_W}\leq\|h\|_X.
\]
The assertion follows.
\end{proof}

\subsection{Component-wise directional quasi-Banach function spaces}\label{sec:componentwise}
Not all directional quasi-Banach function spaces $\mb{X}$ are of the form $\mb{X}=X_W$. Here, we give another way of constructing such a space from given scalar-valued quasi-Banach function spaces.

Suppose that for each $k\in\{1,\ldots,n\}$ we are given a quasi-Banach function space $X_k$ over $\Omega$, and $(v_1,\ldots,v_n)$ is an orthonormal basis of $\F^n$. We define $\mb{X}$ as those $f\in L^0(\Omega;\F^n)$ for which
\[
\|f\|_{\mb{X}}:=\sum_{k=1}^n\|f\cdot v_k\|_{X_k}<\infty.
\]
Then this is a quasi-Banach function space with
\[
K_{\mb{X}}\leq\max_{k\in\{1,\ldots,n\}}K_{X_k}.
\]
The directional ideal property follows from the ideal properties of the $X_k$ and Proposition~\ref{prop:ideal}. Since we have
\[
\int_\Omega\!f\cdot g\,\mathrm{d}\mu=\sum_{k=1}^n\int_\Omega (f\cdot v_k)(v_k\cdot g)\,\mathrm{d}\mu,
\]
we find that $\mb{X}'$ is given by
\[
\|g\|_{\mb{X}'}=\max_{k\in\{1,\ldots,n\}}\|g\cdot v_k\|_{X_k'}.
\]
In particular, $\mb{X}$ satisfies the non-degeneracy property, since each of the $X_k$ do. As a matter of fact, we will show that $\mb{X}$ and $\mb{X}'$ satisfy the component-wise saturation property.

As proving the result for $\mb{X}'$ is analogous to proving it for $\mb{X}$, it suffices to do the latter. We show that $\mb{X}$ satisfies the second characterization in Proposition~\ref{prop:componentsaturation}.  For each $k\in\{1,\ldots,n\}$, let $0<\rho_k\in X_k$ be a weak order unit (which exists by Proposition~\ref{prop:weakorderunit}) satisfying
\[
\|\rho_k\|_{X_k}=\tfrac{1}{n}.
\]
We define the mapping $U:\Omega\to\F^{n\times n}$ through
\[
U(x)v_k:=\rho_k(x)v_k.
\]
Now, if $g\in L^0(\Omega;\F^n)$ satisfies $|U(x)^{-1}g(x)|\leq 1$, then
\[
|g(x)\cdot v_k|=|U(x)^{-1}g(x)\cdot v_k|\rho_k(x)\leq\rho_k(x),
\]
so by the ideal property of the $X_k$ we have
\[
\sum_{k=1}^n\|g\cdot v_k\|_{X_k}\leq \sum_{k=1}^n\|\rho_k\|_{X_k}=1.
\]
Hence, $g\in\mb{X}$ with $\|g\|_{\mb{X}}\leq 1$, as desired.

\bigskip

Next, we note that if all of the $X_k$ have the Fatou property, but not all are equal, then there is no Banach function space $X$ over $\Omega$ for which $\|f\|_{\mb{X}}=\||f|\|_{X}$ for all $f\in\mb{X}$. Indeed, if this were to be the case, then for any $v_k$ we have
\[
\|h\|_{X_k}=\|hv_k\|_{\mb{X}}=\|h\|_X,
\]
proving that $X=X_k$ by \cite[Proposition~3.10]{LN23b}. Thus, all the $X_k$ are equal.

Similarly, it need not be the case that there is a matrix-valued mapping $W:\Omega\to\F^{n\times n}$ for which $\|f\|_{\mb{X}}=\||Wf|\|_X$ for all $f\in\mb{X}$. Indeed, set $X_1:=L^\infty(\Omega)$, $X_2:=L^1(\Omega)$. Then, setting 
\[
w_1(x):=|W(x)v_1|^{-1},\quad w_2(x):=|W(x)v_2|^{-1},
\]
by a similar calculation as above, we find 
\[
L^\infty_{w_1}(\Omega)=X=L^1_{w_2}(\Omega).
\]
However, as $L^\infty_{w_1}(\Omega)$ is not order-continuous whereas $L^1_{w_2}(\Omega)$ is, this is impossible.

\subsection{Infinite sums of convex-set valued mappings}\label{subsec:sumsinXK}
Since we are interested in constructing a version of the Rubio de Francia algorithm in $\mb{X}[\mc{K}]$, we need to be able to define infinite sums of convex-set valued mappings. We first show that if $F,G\in L^0(\Omega;\mc{K})$, then their closed Minkowski sum $F+G$ is again in $L^0(\Omega;\mc{K})$.
\begin{proposition}\label{prop:sumofconvexmeasurable}
Let $F,G\in L^0(\Omega;\mc{K})$. Then
\[
(F+G)(x)=\overline{\{u+v:u\in F(x),\, v\in G(x)\}}
\]
satisfies $F+G\in L^0(\Omega;\mc{K})$. Moreover, if $h\in S^0(\Omega;F+G)$, then there are sequences $(f_k)_{k\geq 1}$ and $(g_k)_{k\geq 1}$ respectively in $S^0(\Omega;F)$ and $S^0(\Omega;G)$ such that
\[
f_k+g_k\to h
\]
a.e. Moreover, if $\mb{X}$ is an $\F^n$-directional quasi-Banach function space with the Fatou property and $F,G\in\mb{X}[\mc{K}]$, then $F+G\in\mb{X}[\mc{K}]$ with
\[
\|F+G\|_{\mb{X}[\mc{K}]}\leq K_{\mb{X}}(\|F\|_{\mb{X}[\mc{K}]}+\|G\|_{\mb{X}[\mc{K}]}).
\]
\end{proposition}
\begin{proof}
By Proposition~\ref{prop:measurablemultifunction} we can pick measurable selections $(f_k)_{k\geq 1}$ and $(g_k)_{k\geq 1}$ of $F$ and $G$ respectively such that
\[
F(x)=\overline{\{f_k(x):k\geq 1\}},\quad G(x)=\overline{\{g_k(x):k\geq 1\}}
\]
for a.e. $x\in\Omega$. Then the functions $h_{j,k}:=f_j+g_k\in S^0(\Omega;F+G)$ satisfy
\[
(F+G)(x)=\overline{\{h_{j,k}(x):j,k\geq 1\}}
\]
a.e. As $(h_{j,k})_{j,k\geq 1}$ is countable, we conclude that $F+G\in L^0(\Omega;\mc{K})$.

Now, suppose $h\in S^0(\Omega;F+G)$. By applying Theorem~\ref{thm:filippov} to
\[
\phi(x,(u,v))=|h(x)-(u+v)|,\quad h_K(x)=\min_{j,k\in\{1,\ldots, K\}}|h(x)-(f_k(x)+g_j(x))|,
\]
we find a selection $(\widetilde{f}_K,\widetilde{g}_K)\in S^0(\Omega;F\times G)=S^0(\Omega;F)\times S^0(\Omega,G)$ for which
\[
|h(x)-(\widetilde{f}_K(x)+\widetilde{g}_K(x))|=\min_{j,k\in\{1,\ldots, K\}}|h(x)-(f_k(x)+g_j(x))|
\]
a.e. Since the right-hand side converges to
\[
\inf_{j,k\geq 1}|h(x)-(f_k(x)+g_j(x))|=0,
\]
we conclude that $\widetilde{f}_K+\widetilde{g}_K\to h$ a.e., as desired. If $\mb{X}$ has the Fatou property and $F,G\in\mb{X}[\mc{K}]$, then also $h\in\mb{X}$ with
\begin{align*}
\|h\|_{\mb{X}}&\leq\liminf_{K\to\infty}\|\widetilde{f}_K+\widetilde{g}_K\|_{\mb{X}}\leq\sup_{K\geq 1}K_{\mb{X}}(\|\widetilde{f}_K\|_{\mb{X}}+\|\widetilde{g}_K\|_{\mb{X}})\\
&\leq K_{\mb{X}}(\|F\|_{\mb{X}[\mc{K}]}+\|G\|_{\mb{X}[\mc{K}]}).
\end{align*}
Hence, $F+G\in\mb{X}[\mc{K}]$ with
\[
\|F+G\|_{\mb{X}[\mc{K}]}\leq K_{\mb{X}}(\|F\|_{\mb{X}[\mc{K}]}+\|G\|_{\mb{X}[\mc{K}]}),
\]
proving the result.
\end{proof}

Given a sequence $(F_k)_{k\geq 1}$ in $\mb{X}[\mc{K}]$, their partial sums are defined through the closed Minkowski sum
\[
S_K(x):=\sum_{k=1}^K F_k(x)=\overline{\Big\{\sum_{k=1}^Ku_k:u_k\in F_k(x)\Big\}}.
\]
If $S_K\uparrow F$, then we write $\sum_{k=1}^\infty F_k:=F$. Per definition, this means that
\[
\sum_{k=1}^\infty F_k(x)=\overline{\bigcup_{K=1}^\infty S_K(x)}=\overline{\bigcup_{K=1}^\infty \Big\{\sum_{k=1}^Ku_k:u_k\in F_k(x)\Big\}}.
\]
By Proposition~\ref{prop:sumofconvexmeasurable} we have $S_K\in L^0(\Omega;\mc{K})$. Hence, we also have $\sum_{k=1}^\infty F_k\in L^0(\Omega;\mc{K})$.
\begin{theorem}[Directional Riesz-Fischer]\label{thm:directionalrieszfischer}
Let $\mb{X}$ be an $\F^n$-directional quasi-Banach function space over $\Omega$ with the Fatou property, and let $(F_k)_{k\geq 1}$ be a sequence in $\mb{X}[\mc{K}]$ satisfying
\[
\sum_{k=1}^\infty K_{\mb{X}}^k\|F_k\|_{\mb{X}[\mc{K}]}<\infty.
\]
Then $F:=\sum_{k=1}^\infty F_k\in\mb{X}[\mc{K}]$ with
\[
\|F\|_{\mb{X}[\mc{K}]}\leq \sum_{k=1}^\infty K_{\mb{X}}^k\|F_k\|_{\mb{X}[\mc{K}]}.
\]
\end{theorem}
\begin{proof}
By inductively applying Proposition~\ref{prop:sumofconvexmeasurable}, we find that $S_K\in\mb{X}[\mc{K}]$ with
\[
\|S_K\|_{\mb{X}[\mc{K}]}\leq\sum_{k=1}^K K_{\mb{X}}^k\|F_k\|_{\mb{X}[\mc{K}]}\leq \sum_{k=1}^\infty K_{\mb{X}}^k\|F_k\|_{\mb{X}[\mc{K}]}.
\]
Since $S_K\uparrow F$ and the Fatou property of $\mb{X}$ implies the monotone convergence property, it follows from Proposition~\ref{prop:convexsetmonotoneconvergence} that $F\in\mb{X}[\mc{K}]$ with
\[
\|F\|_{\mb{X}[\mc{K}]}=\sup_{K\geq 1}\|S_K\|_{\mb{X}[\mc{K}]}\leq\sum_{k=1}^\infty K_{\mb{X}}^k\|F_k\|_{\mb{X}[\mc{K}]}.
\]
This proves the assertion.
\end{proof}

The space $\mb{X}[\mc{K}]$ is not a quasi-normed vector space, as it lacks an additive structure (there is no additive inverse to Minkowski summation of sets). Nonetheless, it is a quasi-metric space with respect to
\[
d_{\mb{X}[\mc{K}]}(F,G):=\max\{\sup_{f\in S^0(\Omega;F)}\inf_{g\in S^0(\Omega;G)}\|f-g\|_{\mb{X}},\sup_{g\in S^0(\Omega;G)}\inf_{f\in S^0(\Omega;F)}\|f-g\|_{\mb{X}}\}.
\]
Let $\mc{K}_b(\mb{X})$ denote the collection of non-empty closed, bounded, convex, and symmetric subsets of $\mb{X}$. Then $d_{\mb{X}[\mc{K}]}$ is exactly the Hausdorff quasi-distance inherited from $\mc{K}_b[\mb{X}]$ through the embedding 
\[
\mb{X}[\mc{K}]\hookrightarrow \mc{K}_b[\mb{X}],\quad F\mapsto \overline{S^0(\Omega;F)}.
\]
Note that $d_{\mb{X}[\mc{K}]}(F,\{0\})=\|F\|_{\mb{X}[\mc{K}]}$ and
\begin{equation}\label{eq:hausdorffdistanceproperties}
d_{\mb{X}[\mc{K}]}(F+H,G+H)\leq d_{\mb{X}[\mc{K}]}(F,G),\quad d_{\mb{X}[\mc{K}]}(\lambda F,\lambda G)=|\lambda|d_{\mb{X}[\mc{K}]}(F,G)
\end{equation}
for all $F,G,H\in\mb{X}[\mc{K}]$ and $\lambda\in\F$. 

In \cite{BC23} the metric defined on $L^p_W(\R^d;\mc{K})$ is through a pointwise weighted Hausdorff distance. In general, for matrix weighted spaces $X_W$ where $X$ has the Fatou property, this approach is equivalent. Indeed, setting
\[
d_W(F,G)(x):=\max\{\sup_{u\in F(x)}\inf_{v\in G(x)}|W(x)(u-v)|,\sup_{v\in G(x)}\inf_{u\in F(x)}|W(x)(u-v)|\},
\]
we have the following:
\begin{proposition}\label{prop:hausdorffmatrixweight}
Let $X$ be a quasi-Banach function space over $\Omega$ with the Fatou property and let $W:\Omega\to\F^{n\times n}$ be a matrix weight. Then
\[
d_{X_W[\mc{K}]}(F,G)\eqsim \|d_W(F,G)\|_X.
\]
\end{proposition}
\begin{proof}[Proof of Proposition~\ref{prop:hausdorffmatrixweight}]
Note that by Proposition~\ref{prop:matrixweightfatouimpliesbounded} the quasi-distance $d_W$ is well-defined. Let $f\in S^0(\Omega;F)$ and let
\[
\phi(x,v):=|W(x)f(x)-W(x)v|,\quad h(x):=\inf_{v\in G(x)}\phi(x,v)
\]
Note that the infimum defining $h$ is the distance between a point and a closed set in $\F^n$ and, thus, is achieved, i.e., $h(x)\in\{\phi(x,v):v\in G(x)\}$. Hence, by Theorem~\ref{thm:filippov}, we can find a $g\in S^0(\Omega;G)$ such that
\[
|W(x)f(x)-W(x)g(x)|=d_W(f,G)(x):=\inf_{v\in G(x)}|W(x)f(x)-W(x)v|
\]
a.e. Hence, by the ideal property of $X$,
\[
\|f-g\|_{X_W}=\||W(f-g)|\|_X\leq\|d_W(f,G)\|_X.
\]
We conclude that
\[
\sup_{f\in S^0(\Omega;F)}\inf_{g\in S^0(\Omega;G)}\|f-g\|_{X_W}\leq \|d_W(f,G)\|_X
\]
Noting that $d_W(f,G)\leq d_W(F,G)$ a.e. and by repeating the same argument with the roles of $F$ and $G$ reversed, we conclude that
\[
d_{X_W[\mc{K}]}(F,G)\leq \|d_W(F,G)\|_X.
\]
For the converse inequality, just like above we may apply Theorem~\ref{thm:filippov} with
\[
\phi(x,u)=\inf_{v\in G(x)}|W(x)u-W(x)v|,\quad h(x)=\sup_{u\in F(x)}\phi(x,u),
\]
where measurability of $\phi(\cdot,u)$ now follows from Proposition~\ref{prop:measurablemultifunction}. This yields an $f_0\in S^0(\Omega;F)$ for which
\[
\inf_{v\in G(x)}|W(x)f_0(x)-W(x)v|=\sup_{u\in F(x)}\inf_{v\in G(x)}|W(x)(u-v)|
\]
a.e. Thus, for any $g\in S^0(\Omega;G)$ we have
\[
\sup_{u\in F(x)}\inf_{v\in G(x)}|W(x)(u-v)|\leq |W(x)f_0(x)-W(x)g(x)|.
\]
a.e. By the ideal property of $X$ this implies
\[
\|\sup_{u\in F(x)}\inf_{v\in G(x)}|W(x)(u-v)|\|_X\leq\|f_0-g\|_{X_W}.
\]
Taking an infimum over all $g\in S^0(\Omega;G)$, we conclude that
\begin{align*}
\|\sup_{u\in F}\inf_{v\in G}|W(u-v)|\|_X&\leq\inf_{g\in S^0(\Omega;G)}\|f_0-g\|_{X_W}\\
&\leq\sup_{f\in S^0(\Omega;F)}\inf_{g\in S^0(\Omega;G)}\|f-g\|_{X_W}.
\end{align*}
Repeating the argument with the roles of $F$ and $G$ reversed, we conclude that
\begin{align*}
\|d_W(F,G)\|_X&\leq K_X\|\sup_{u\in F}\inf_{v\in G}|W(u-v)|\|_X+K_X\|\sup_{v\in G}\inf_{u\in F}|W(u-v)|\|_X\\
&\leq 2K_X d_{X_W[\mc{K}]}(F,G).
\end{align*}
This proves the assertion.
\end{proof}

Completeness of $(X_W,d_{X_W[\mc{K}]})$ can be proven analogously to \cite[Theorem~4.8]{BC23}. Even without assuming the Fatou property, the space $X_W$ satisfies the directional Riesz-Fischer property of Theorem~\ref{thm:directionalrieszfischer}, as long as we assume that every $F\in X_W[\mc{K}]$ satisfies $F(x)\in\mc{K}_b$ a.e.:
\begin{proposition}
Let $X$ be a quasi-Banach function space over $\Omega$, let $W:\Omega\to\F^{n\times n}$ be a matrix weight, and assume that $X_W[\mc{K}]\subseteq L^0(\Omega;\mc{K}_b)$. If $(F_k)_{k\geq 1}$ is a sequence in $X_W[\mc{K}]$ satisfying
\[
C:=\sum_{k=1}^\infty K_X^k\|F_k\|_{X_W[\mc{K}]}<\infty,
\]
then $F:=\sum_{k=1}^\infty F_k\in X_W[\mc{K}]$ with $\|F\|_{X_W[\mc{K}]}\leq K_X C$.
\end{proposition}
\begin{proof}
For $f\in S^0(\Omega;F)$ we have
\begin{equation}\label{eq:dirrieszfischermatrixweight1}
|W(x)f(x)|\leq\sum_{k=1}^\infty \sup_{u_k\in F_k(x)}|W(x)u_k|
\end{equation}
a.e. Since, by Proposition~\ref{prop:supattainmentXW}, the sequence $h_k(x):=\sup_{u_k\in F_k(x)}|W(x)u_k|$ satisfies $\|h_k\|_X=\|F_k\|_{X_W[\mc{K}]}$, we have $\sum_{k=1}^\infty K_X^k\|h_k\|_X<\infty$. Thus, by the Riesz-Fischer property of $X$, the sum on the right-hand side of \eqref{eq:dirrieszfischermatrixweight1} belongs to $X$ with
\[
\Big\|\sum_{k=1}^\infty h_k\Big\|_X\leq K_X\sum_{k=1}^\infty K_X^k\|h_k\|_X=:C.
\]
Hence, by the ideal property of $X$, $f\in X_W$ with
\[
\|f\|_{X_W}=\||Wf|\|_X\leq C.
\]
Thus, $F\in X_W[\mc{K}]$ with $\|F\|_{X_W[\mc{K}]}\leq C$, as asserted.
\end{proof}
We leave it as an open problem whether the Fatou property is necessary for Theorem~\ref{thm:directionalrieszfischer} or not.

\subsection{Weak-type directional quasi-Banach function spaces}
Given an $\F^n$-directional quasi-Banach function space $\mb{X}$ over $\Omega$, we define a weak-type analogue $\mb{X}_{\text{weak}}$ as the space of $f\in L^0(\Omega;\F^n)$ for which $\ind_{\mc{K}(f)^{-1}(\{u\})}u\in\mb{X}$ for all $u\in\F^n$, with
\[
\|f\|_{\mb{X}_{\text{weak}}}:=\sup_{u\in\F^n}\|\ind_{\mc{K}(f)^{-1}(\{u\})}u\|_{\mb{X}}<\infty.
\]
When $n=1$, this space satisfies the quasi-triangle inequality as a consequence of the fact that for any $K,L\in\mc{K}$ we have $K+L\subseteq 2(K\cup L)$. However, this inclusion fails when $n>2$, and, hence, $\mb{X}_{\text{weak}}$ might not be a vector space. Moreover, we note that if $\|f\|_{\mb{X}_{\text{weak}}}=0$, then we can conclude that 
\begin{equation}\label{eq:weaktypezero}
\mu(\{x\in\Omega:u\in\mc{K}(f)(x)\})=0
\end{equation}
for all non-zero $u\in\F^n$, which does not imply that $f=0$ a.e. For example, the mapping
\[
f:\R\to\R^2,\quad f(x):=(\cos x,\sin x)
\]
satisfies \eqref{eq:weaktypezero}, but is not equal to $0$ a.e. Despite these difficulties, $\mb{X}_{\text{weak}}$ does still satisfy the directional ideal and non-degeneracy properties, and inherits the monotone convergence property from $\mb{X}$ in case $\mb{X}$ has it.
\begin{proposition}
Let $\mb{X}$ be an $\F^n$-directional quasi-Banach function space over $\Omega$. Then $\mb{X}_{\text{weak}}$ satisfies the directional ideal and non-degeneracy properties. Moreover, we have $\mb{X}\subseteq\mb{X}_{\text{weak}}$ with
\[
\|f\|_{\mb{X}_{\text{weak}}}\leq\|f\|_{\mb{X}},
\]
and
\[
\|F\|_{\mb{X}_{\text{weak}}[\mc{K}]}=\sup_{u\in\F^n}\|\ind_{F^{-1}(\{u\})}u\|_{\mb{X}}.
\]
Finally, if $\mb{X}$ satisfies the monotone convergence property, then so does $\mb{X}_{\text{weak}}$.
\end{proposition}
\begin{proof}
Suppose that $\mc{K}(g)\subseteq\mc{K}(f)$ a.e., say, on a set $E\subseteq\Omega$ whose complement has measure $0$. Then $\ind_{\mc{K}(g)^{-1}(\{u\})}(x)\leq\ind_{\mc{K}(f)^{-1}(\{u\})}(x)$ for all $x\in E$ for all $u\in\F^n$. Thus, we have $\ind_{\mc{K}(g)^{-1}(\{u\})}\leq\ind_{\mc{K}(f)^{-1}(\{u\})}$ a.e. for all $u\in\F^n$, so the directional ideal property of $\mb{X}_{\text{weak}}$ follows from the directional ideal property of $\mb{X}$ and Proposition~\ref{prop:ideal}. 

For the non-degeneracy property, we first show that $\mb{X}\subseteq\mb{X}_{\text{weak}}$. Note that since
\[
\ind_{\mc{K}(f)^{-1}(\{u\})}(x)u=\ind_{\{x\in\Omega:u\in \mc{K}(f)(x)\}}(x)u\in\mc{K}(f)(x)
\]
a.e. for all $u\in\F^n$, it follows from the directional ideal property that $\mb{X}\subseteq\mb{X}_{\text{weak}}$ with
\[
\|f\|_{\mb{X}_{\text{weak}}}\leq\|f\|_{\mb{X}},
\]
as desired. Now, suppose $g\in L^0(\Omega;\F^n)$ and that $\int_{\Omega}\!f\cdot g\,\mathrm{d}\mu=0$ for all $f\in\mb{X}_{\text{weak}}$. Then, since $\mb{X}\subseteq\mb{X}_{\text{weak}}$, this also holds for all $f\in\mb{X}$. Thus, by the non-degeneracy property of $\mb{X}$, we have $g=0$ a.e., as desired.

If $F\in\mb{X}_{\text{weak}}[\mc{K}]$, then, as above, for any $f\in S^0(\Omega;F)$ and $u\in\F^n$ we have $\ind_{\mc{K}(f)^{-1}(\{u\})}\leq \ind_{F^{-1}(\{u\})}$ a.e., so, by the directional ideal property of $\mb{X}$ and Proposition~\ref{prop:ideal},
\[
\|\ind_{\mc{K}(f)^{-1}(\{u\})}u\|_{\mb{X}}\leq \|\ind_{F^{-1}(\{u\})}u\|_{\mb{X}}\leq\sup_{u\in\F^n}\|\ind_{F^{-1}(\{u\})}u\|_{\mb{X}}.
\]
Taking a supremum over all $u\in\F^n$ then shows $F\in\mb{X}_{\text{weak}}[\mc{K}]$ with 
\[
\|F\|_{\mb{X}_{\text{weak}}[\mc{K}]}\leq\sup_{u\in\F^n}\|\ind_{F^{-1}(\{u\})}u\|_{\mb{X}}.
\]
Conversely, let $u\in\F^n$ and define $f:=\ind_{F^{-1}(\{u\})}u$. Then $f\in S^0(\Omega;F)$, and
\[
\mc{K}(f)^{-1}(\{u\})=F^{-1}(\{u\}).
\]
Hence,
\[
\|\ind_{F^{-1}(\{u\})}u\|_{\mb{X}}=\|\ind_{\mc{K}(f)^{-1}(\{u\})}u\|_{\mb{X}}\leq\|f\|_{\mb{X}_{\text{weak}}}\leq\|F\|_{\mb{X}_{\text{weak}}[\mc{K}]}.
\]
Taking a supremum over all $u\in\F^n$ now proves the desired identity.

Next, suppose $\mb{X}$ has the monotone convergence property. If $(f_k)_{k\geq 1}$ is such that $(\mc{K}(f_k))_{k\geq 1}$ is an increasing sequence with union $\mc{K}(f)$, then
\[
\ind_{\mc{K}(f_k)^{-1}(\{u\})}\uparrow \ind_{\mc{K}(f)^{-1}(\{u\})}
\]
for all $u\in\F^n$. Hence, $\ind_{\mc{K}(f)^{-1}(\{u\})}u\in\mb{X}$, with
\[
\|\ind_{\mc{K}(f)^{-1}(\{u\})}u\|_{\mb{X}}=\sup_{k\geq 1}\|\ind_{\mc{K}(f_k)^{-1}(\{u\})}\|_{\mb{X}}\leq\sup_{k\geq 1}\|f_k\|_{\mb{X}_{\text{weak}}}.
\]
Thus, taking a supremum over $u\in\F^n$, proves that $f\in\mb{X}_{\text{weak}}$. We conclude that $\mb{X}_{\text{weak}}$ has the monotone convergence property.
\end{proof}

\begin{remark}
When $X$ is a quasi-Banach function space over $\Omega$ and $W:\Omega\to\F^n$ is a matrix weight, then we can define an alternative weak-type space as follows. We define $X_{\text{weak}}$ through
\[
\|f\|_{X_{\text{weak}}}:=\sup_{t>0}\|t\ind_{\{|f|>t\}}\|_X.
\]
As this is again a quasi-Banach function space over $\Omega$, the space $(X_{\text{weak}})_W$ is a well-defined $\F^n$-directional quasi-Banach function space over $\Omega$ consisting of the $f\in L^0(\Omega;\F^n)$ for which $|Wf|\in X_{\text{weak}}$.

The difference between the two spaces can be more easily illustrated when $n=1$. If $w$ is a weight, $X=L_w^p(\R^d)$, and $f\in L^p_w(\R^d)$, then for any $u\in\F^n$ we have 
\[
\{x\in\R^d:u\in\mc{K}(f)(x)\}=\{x\in\R^d:|u|\leq |f(x)|\}.
\]
Hence, $L^p_w(\R^d)_{\text{weak}}$ as we have defined in this section is given by
\[
\|f\|_{L^p_w(\R^d)_{\text{weak}}}=\sup_{u\in\F^n}\|\ind_{\mc{K}(f)^{-1}(u)}u\|_{L^p_w(\R^d)}=\sup_{t>0} tw^p(\{|f|>t\})^{\frac{1}{p}}.
\]
On the other hand, we have
\[
\|f\|_{(L^p(\R^d)_{\text{weak}})_w}=\|wf\|_{L^{p,\infty}(\R^d)}=\sup_{t>0}t|\{|wf|>t\}|^{\frac{1}{p}}.
\]
Thus, the different approaches respectively yield a weak-type space with respect to a change of measure, and a weak-type space with respect to a weight as a multiplier.
\end{remark}

\section{Averaging operators}\label{sec:averagingoperators}

\subsection{Directional averaging operators} 
Throughout this section we let $E\subseteq\Omega$ be a measurable set for which $0<\mu(E)<\infty$. For $f\in L^0(\Omega)$ integrable in $E$ we set
\[
T_Ef:=\langle f\rangle_E\ind_E,\quad \langle f\rangle_E:=\frac{1}{\mu(E)}\int_E\!f\,\mathrm{d}\mu.
\]
This can be extended component-wise to functions $f\in L^0(\Omega;\F^n)$, which we also denote by $T_Ef$.

Given an $\F^n$-directional quasi-Banach function space $\Omega$, we note that for any $u,v\in\F^n$ with $\ind_E u\in\mb{X}$, $\ind_E v\in\mb{X}'$, we have
\[
\mu(E)^{-1}\|\ind_E u\|_{\mb{X}}\|\ind_E v\|_{\mb{X}'}\geq\frac{1}{\mu(E)}\int_\Omega\!|u\ind_E\cdot v\ind_E|\,\mathrm{d}\mu=|u\cdot v|.
\]
A converse inequality characterizes the boundedness $T_E:\mb{X}\to\mb{X}$. When writing $T_E:\mb{X}\to\mb{X}$ we mean that $T_E$ is a bounded operator from $\mb{X}$ to $\mb{X}$ and, in particular, that every $f\in\mb{X}$ is integrable over $E$, so $T_Ef$ is well-defined. To facilitate this, throughout this section we will assume that our space satisfy the component-wise saturation property. By Proposition~\ref{prop:componentsaturation} this property is equivalent to stating that the space $\mb{X}_u$, containing all $h\in L^0(\Omega)$ with $hu\in\mb{X}$, is saturated.
\begin{proposition}\label{prop:averagingoperatornorm}
Let $\mb{X}$ be an $\F^n$-directional quasi-Banach function space over $\Omega$ with the Fatou property and the component-wise saturation property. If $E\subseteq\Omega$ with $0<\mu(E)<\infty$, then the following assertions hold:
\begin{enumerate}[(a)]
    \item\label{it:avopa} If $T_E:\mb{X}\to\mb{X}$, then $\ind_E u\in\mb{X},\mb{X}'$ for all $u\in\F^n$. Moreover, for all $v\in\F^n$  there is a non-zero $u\in\F^n$ for which
    \[
    \mu(E)^{-1}\|\ind_E u\|_{\mb{X}}\|\ind_Ev\|_{\mb{X}'}\leq\|T_E\|_{\mb{X}\to\mb{X}}|u\cdot v|.
    \]
    \item\label{it:avopb} If $\ind_E u\in\mb{X},\mb{X}'$ for all $u\in\F^n$ and if there is a $C>0$ such that for all $u\in\F^n$ there is a non-zero $v\in\F^n$ for which
    \[
    \mu(E)^{-1}\|\ind_E u\|_{\mb{X}}\|\ind_E v\|_{\mb{X}'}\leq C|u\cdot v|,
    \]
    then $T_E:\mb{X}\to\mb{X}$ with
    \[
\|T_E\|_{\mb{X}\to\mb{X}}\leq C.
    \]
\end{enumerate}
\end{proposition}
\begin{proof}
For \ref{it:avopa}, assume that $T_E:\mb{X}\to\mb{X}$. Let $0\neq u\in\F^n$. By Proposition~\ref{prop:componentsaturation}, the space $\mb{X}_u$ is saturated, so there exists a $0<\rho_u\in\mb{X}_u$ of norm $1$. Hence $\rho_u u\in\mb{X}$, and
\[
\langle\rho_u\rangle_E\|\ind_E u\|_{\mb{X}}=\|T_E(\rho_u u)\|_{\mb{X}}\leq\|T_E\|_{\mb{X}\to\mb{X}}.
\]
Since $\langle\rho_u\rangle_E>0$, we conclude that $\ind_E u\in\mb{X}$.

Next, we claim that $\inf_{|u|=1}\|\ind_E u\|_{\mb{X}}>0$. Indeed, if not, then there is a sequence of $(u_k)_{k\geq 1}$ with $|u_k|=1$ for which
\[
\lim_{k\to\infty}\|\ind_E u_k\|_{\mb{X}}=0.
\]
As the unit circle in $\F^n$ is compact, there is a $u\in\F^n$ with $|u|=1$ such that a subsequence $(u_{k_j})_{j\geq 1}$ satisfies $u_{k_j}\to u$. By the Fatou property of $\mb{X}$, this implies that
\[
\|\ind_E u\|_{\mb{X}}\leq\liminf_{j\to\infty}\|\ind_E u_{k_j}\|_{\mb{X}}=0,
\]
which, by non-degeneracy of $\mb{X}$, is impossible, proving the claim.

Now, let $0\neq v\in\F^n$ and $f\in\mb{X}$ with $\|f\|_{\mb{X}}=1$. Setting $u:=\langle f\rangle_E$, we have
\begin{align*}
\Big|\int_\Omega\!f\cdot \ind_E v\,\mathrm{d}\mu\Big|
&=\mu(E)|\langle f\rangle_E\cdot v|
=\mu(E)\frac{|u\cdot v|}{\|\ind_E u\|_{\mb{X}}}\|\langle f\rangle_E\ind_E\|_{\mb{X}}\\
&\leq \mu(E)\|T_E\|_{\mb{X}\to\mb{X}}\frac{|u\cdot v|}{\|\ind_E u\|_{\mb{X}}}\\
&\leq \mu(E)\|T_E\|_{\mb{X}\to\mb{X}}\frac{|v|}{\inf_{|u|=1}\|\ind_E u\|_{\mb{X}}}.
\end{align*}
Thus, by Proposition~\ref{prop:normoutsideintegral}, we have $\ind_Ev\in\mb{X}'$, as desired. For the last statement, let $v\in\F^n$. Then for each $k\geq 1$ there is a $f_k\in\mb{X}$ of norm $1$ such that
\[
\|\ind_E v\|_{\mb{X}'}\leq (1+\tfrac{1}{k})\Big|\int_E\!f_k\cdot v\,\mathrm{d}\mu\Big|.
\]
Setting $u_k:=\langle f_k\rangle_E\neq 0$, the same computation as above proves that
\begin{align*}
\mu(E)^{-1}\|\ind_E u_k\|_{\mb{X}}\|\ind_E v\|_{\mb{X}'}
&\leq(1+\tfrac{1}{k})\|T_E\|_{\mb{X}\to\mb{X}}|u_k\cdot v|.
\end{align*}
By homogeneity, we may assume that $|u_k|=1$. Then, by compactness, there is a convergent subsequence $(u_{k_j})_{j\geq 1}$ with limit $u\in\F^n$, $|u|=1$. Since then $\ind_Eu_{k_j}\to\ind_E u$ a.e., it follows from the Fatou property that
\begin{align*}
    \mu(E)^{-1}\|\ind_E u\|_{\mb{X}}\|\ind_E v\|_{\mb{X}'}&\leq 
    \liminf_{j\to\infty}\mu(E)^{-1}\|\ind_E u_{k_j}\|_{\mb{X}}\|\ind_E v\|_{\mb{X}'}\\
    &\leq \liminf_{j\to\infty}(1+\tfrac{1}{k_j})\|T_E\|_{\mb{X}\to\mb{X}}|u_{k_j}\cdot v|\\
    &=\|T_E\|_{\mb{X}\to\mb{X}}|u\cdot v|.
\end{align*}
This proves \ref{it:avopa}.

Finally, we prove \ref{it:avopb}. Since $\ind_E v\in\mb{X}'$ for all $v\in\F^n$, it follows that for $f\in\mb{X}$ we have
\[
\int_E\!|f|\,\mathrm{d}\mu\leq\sum_{k=1}^n\int_\Omega\!|f\cdot \ind_E e_k|\,\mathrm{d}\mu\leq\sum_{k=1}^n\|f\|_{\mb{X}}\|\ind_E e_k\|_{\mb{X}'},
\]
where $(e_k)_{k=1}^n$ denotes the standard basis of $\F^n$. Hence, $f\ind_E\in L^1(\Omega;\F^n)$, so $u:=\langle f\rangle_E\in\F^n$ is well-defined. Per assumption, we can pick a $v\in\F^n$ with $\|\ind_E v\|_{\mb{X}'}=1$ for which, by Proposition~\ref{prop:normoutsideintegral},
\[
\|T_Ef\|_{\mb{X}}=\|\ind_E u\|_{\mb{X}}\leq C\mu(E)|u\cdot v|\leq C\int_\Omega\!|f\cdot\ind_E v|\,\mathrm{d}\mu\leq C\|f\|_{\mb{X}},
\]
as asserted.
\end{proof}

For the following result, recall that by the directional Lorentz-Luxemburg theorem (Theorem~\ref{thm:lorentzluxemburg}), if $\mb{X}$ is K\"othe reflexive, then it has the Fatou property.
\begin{corollary}\label{cor:averagingoperator}
Let $\mb{X}$ be a K\"othe reflexive $\F^n$-directional Banach function space over $\Omega$ for which both $\mb{X}$ and $\mb{X}'$ have the component-wise saturation property. If $E\subseteq\Omega$ with $0<\mu(E)<\infty$, then the following are equivalent:
\begin{enumerate}[(i)]
    \item\label{it:avop1} $T_E:\mb{X}\to\mb{X}$;
    \item\label{it:avop2} $T_E:\mb{X}'\to\mb{X}'$
    \item\label{it:avop3} $\ind_E u\in\mb{X},\mb{X}'$ for all $u\in\F^n$ and there is a $C_1>0$ such that for all $u\in\F^n$ there is a non-zero $v\in\F^n$ for which
    \[
    \mu(E)^{-1}\|\ind_E u\|_{\mb{X}}\|\ind_E v\|_{\mb{X}'}\leq C_1|u\cdot v|;
    \]
    \item\label{it:avop4} $\ind_E u\in\mb{X},\mb{X}'$ for all $u\in\F^n$ and there is a $C_2>0$ such that for all $v\in\F^n$ there is a non-zero $u\in\F^n$ for which
    \[
    \mu(E)^{-1}\|\ind_E u\|_{\mb{X}}\|\ind_E v\|_{\mb{X}'}\leq C_2|u\cdot v|.
    \]
\end{enumerate}
Moreover, the optimal constants $C_1$, $C_2$ satisfy
\[
\|T_E\|_{\mb{X}\to\mb{X}}=\|T_E\|_{\mb{X}'\to\mb{X}'}=C_1=C_2.
\]
\end{corollary}
\begin{proof}
Assuming \ref{it:avop1}, assertion \ref{it:avop4} holds by Proposition~\ref{prop:averagingoperatornorm} with $C_2\leq\|T_E\|_{\mb{X}\to\mb{X}}$. Since $\mb{X}''=\mb{X}$, if we assume \ref{it:avop4}, then another application of Proposition~\ref{prop:averagingoperatornorm} proves \ref{it:avop2} with $\|T_E\|_{\mb{X}'\to\mb{X}'}\leq C_2$ . Repeating this argument with the roles of $\mb{X}$ and $\mb{X}'$ reversed, we have proven the chain of implications 
\[
\ref{it:avop1}\Rightarrow\ref{it:avop4}\Rightarrow\ref{it:avop2}\Rightarrow\ref{it:avop3}\Rightarrow\ref{it:avop1},
\]
with
\[
\|T_E\|_{\mb{X}\to\mb{X}}\leq C_1\leq\|T_E\|_{\mb{X}'\to\mb{X}'}\leq C_2\leq\|T_E\|_{\mb{X}\to\mb{X}}.
\]
The result follows.
\end{proof}

If $\mb{X}$ is an $\F^n$ directional Banach function space over $\Omega$ that satisfies the property that $\ind_E u\in\mb{X}$ for all $u\in\F^n$, then we can define a norm on $\F^n$ through
\[
u\mapsto\|\ind_E u\|_{\mb{X}}.
\]
Using the John ellipsoid theorem (see \cite[Proposition~1.2]{Go03}), there is a Hermitian positive definite matrix $A_{\mb{X},E}\in\F^{n\times n}$ that satisfies
\[
\|\ind_E u\|_{\mb{X}}\leq|A_{\mb{X},E}u|\leq n^{\frac{1}{2}}\|\ind_E u\|_{\mb{X}}.
\]
We call the matrix $A_{\mb{X},E}$ the \emph{reducing matrix in $\mb{X}$ of $E$}. Reducing matrices can be used to give another characterization of the boundedness of $T_E$:
\begin{proposition}\label{prop:reducingmatrixavop}
Let $\mb{X}$ be an $\F^n$-directional Banach function space over $\Omega$ with the Fatou property and the component-wise saturation property. Then $T_E:\mb{X}\to\mb{X}$ if and only if $\ind_E u\in\mb{X},\mb{X}'$ for all $u\in\mb{F}^n$. Moreover, we have
\[
\|T_E\|_{\mb{X}\to\mb{X}}\eqsim_n \mu(E)^{-1}\|A_{\mb{X},E}A_{\mb{X}',E}\|_{\F^{n\times n}}.
\]
\end{proposition}
We note that since $A_{\mb{X},E}$ and $A_{\mb{X}',E}$ are Hermitian, we have
\begin{align*}
\|A_{\mb{X},E}A_{\mb{X}',E}\|_{\F^{n\times n}}
&=\|(A_{\mb{X},E}A_{\mb{X}',E})^\ast\|_{\F^{n\times n}}\\
&=\|A_{\mb{X}',E}^\ast A_{\mb{X},E}^\ast\|_{\F^{n\times n}}\\
&=\|A_{\mb{X}',E}A_{\mb{X},E}\|_{\F^{n\times n}}.
\end{align*}
\begin{proof}[Proof of Proposition~\ref{prop:reducingmatrixavop}]
First assume that $T_E:\mb{X}\to\mb{X}$. Then by Proposition~\ref{prop:averagingoperatornorm}\ref{it:avopa} we have $\ind_E u\in\mb{X},\mb{X}'$ for all $u\in\mb{F}^n$. Moreover, if $w\in\F^n$ and $u:=A_{\mb{X}',E}w$, there is a non-zero $v\in\F^n$ such that
\begin{align*}
n^{-\frac{1}{2}}|A_{\mb{X},E}A_{\mb{X}',E}w|
&\leq\|\ind_E u\|_{\mb{X}}
\leq\mu(E)\|T_E\|_{\mb{X}\to\mb{X}}\frac{|u\cdot v|}{\|\ind_E v\|_{\mb{X}'}}\\
&\leq \mu(E)\|T_E\|_{\mb{X}\to\mb{X}}\frac{|w\cdot A_{\mb{X}',E}v|}{n^{-\frac{1}{2}}|A_{\mb{X}',E}v|}\\
&\leq n^{\frac{1}{2}}\mu(E)\|T_E\|_{\mb{X}\to\mb{X}}|w|.
\end{align*}
Hence, $\mu(E)^{-1}\|A_{\mb{X},E}A_{\mb{X}',E}\|_{\F^{n\times n}}\leq n\|T_E\|_{\mb{X}\to\mb{X}}$.

For the converse, let $f\in\mb{X}$ not equal to $0$ a.e. and set $v:=\langle A_{\mb{X},E}f\rangle_E$. Then, by Proposition~\ref{prop:normoutsideintegral},
\begin{align*}
\mu(E)|\langle A_{\mb{X},E}f\rangle_E|^2&=\int_E\!A_{\mb{X},E}f\cdot v\,\mathrm{d}\mu=\int_E\!f\cdot A_{\mb{X},E}v\,\mathrm{d}\mu\\
&\leq\|f\|_{\mb{X}}\|\ind_E A_{\mb{X},E}v\|_{\mb{X}'}
\leq |A_{\mb{X}',E}A_{\mb{X},E}v|\|f\|_{\mb{X}}\\
&\leq \|A_{\mb{X}',E}A_{\mb{X},E}\|_{\F^{n\times n}}|\langle A_{\mb{X},E}f\rangle_E|\|f\|_{\mb{X}},
\end{align*}
so that
\[
\|T_E f\|_{\mb{X}}\leq |\langle A_{\mb{X},E} f\rangle_E|\leq \mu(E)^{-1}\|A_{\mb{X}',E}A_{\mb{X},E}\|_{\F^{n\times n}}\|f\|_{\mb{X}}.
\]
Since
\[
\|A_{\mb{X}',E}A_{\mb{X},E}\|_{\F^{n\times n}}=\|(A_{\mb{X}',E}A_{\mb{X},E})^\ast\|_{\F^{n\times n}}=\|A_{\mb{X},E}A_{\mb{X}',E}\|_{\F^{n\times n}},
\]
this proves the assertion.
\end{proof}

\subsection{Convex-set valued averaging operators}
Given a measurable set $E\subseteq\Omega$ with $0<\mu(E)<\infty$ and $F\in L^0(\Omega;\mc{K})$, we define $F\ind_E$ pointwise by
\[
(F\ind_E)(x):=\ind_E(x)F(x)=\begin{cases}
F(x) & \text{if $x\in E$;}\\
\{0\} & \text{if $x\notin E$.}
\end{cases}
\]
If $F\ind_E\in L^1(\Omega;\mc{K})$, we use the Aumann integral to define
\[
\langle F\rangle_E:=\frac{1}{\mu(E)}\int_\Omega\!F\ind_E\,\mathrm{d}\mu=\{\langle f\rangle_E:f\in S^0(\Omega;F\ind_E)\}.
\]
The averaging operator $T_E$ of $F$ is defined as
\[
T_EF(x):=\langle F\rangle_E\ind_E(x)=\{T_Ef(x):f\in S^0(\Omega;F\ind_E)\}.
\]
We note that if for a function $f\in L^0(\Omega;\F^n)$ with $f\ind_E\in L^1(\Omega;\F^n)$ we define
\[
\llangle f\rrangle_E:=\{\langle hf\rangle_E:h\in L^\infty(\Omega),\,\|h\|_{L^\infty(\Omega)}\leq 1\},
\]
then we have 
\[
T_E(\mc{K}(f))=\llangle f\rrangle_E\ind_E.
\]
This follows from the observation that by Proposition~\ref{prop:ideal} the measurable selections of $\mc{K}(f)$ are given by
\[
S^0(\Omega;\mc{K}(f))=\{hf:\|h\|_{L^\infty(\Omega)}\leq 1\}.
\]
Perhaps unexpectedly, it turns out that the boundedness properties of $T_E$ in a space $\mb{X}$ are equivalent to its boundedness properties in $\mb{X}[\mc{K}]$. We will prove this equivalence for more general averaging operators. Given a collection $\mc{P}$ of measurable sets in $\Omega$ satisfying $0<\mu(E)<\infty$ for all $E\in\mc{P}$, we define $T_{\mc{P}}:=\sum_{E\in\mc{P}}T_E$, i.e., 
\[
T_{\mc{P}}F:=\sum_{E\in\mc{P}}\langle F\rangle_E\ind_E. 
\]
Then we have the following result.
\begin{theorem}\label{thm:pairwisedisjointavconvex}
Let $\mb{X}$ be an $\F^n$-directional Banach function space over $\Omega$, and let $\mc{P}$ be a pairwise disjoint collection of measurable sets in $\Omega$ for which $0<\mu(E)<\infty$ for all $E\in\mc{P}$. Then the following are equivalent:
\begin{enumerate}[(i)]
    \item\label{it:convexsettovector1} $T_{\mc{P}}:\mb{X}\to\mb{X}$;
    \item\label{it:convexsettovector2} $T_{\mc{P}}:\mb{X}[\mc{K}]\to\mb{X}[\mc{K}]$.
\end{enumerate}
Moreover, in this case we have
\[
\|T_{\mc{P}}\|_{\mb{X}\to\mb{X}}\eqsim_n\|T_{\mc{P}}\|_{\mb{X}[\mc{K}]\to\mb{X}[\mc{K}]}.
\]
\end{theorem}
\begin{proof}
For \ref{it:convexsettovector2}$\Rightarrow$\ref{it:convexsettovector1}, Let $f\in\mb{X}$. Since, by Proposition~\ref{prop:convexsetfequalsf}, we have $\mc{K}(f)\in\mb{X}[\mc{K}]$ with $\|\mc{K}(f)\|_{\mb{X}[\mc{K}]}=\|f\|_{\mb{X}}$, noting that $T_{\mc{P}}f\in S^0(\Omega;T_{\mc{P}}(\mc{K}(f)))$, we have
\begin{align*}
\|T_{\mc{P}} f\|_{\mb{X}}&\leq\|T_{\mc{P}}(\mc{K}(f))\|_{\mb{X}}\leq\|T_{\mc{P}}\|_{\mb{X}[\mc{K}]\to\mb{X}[\mc{K}]}\|\mc{K}(f)\|_{\mb{X}[\mc{K}]}\\
&=\|T_{\mc{P}}\|_{\mb{X}[\mc{K}]\to\mb{X}[\mc{K}]}\|f\|_{\mb{X}}.
\end{align*}
Thus, $T_{\mc{P}}:\mb{X}\to\mb{X}$ with $\|T_{\mc{P}}\|_{\mb{X}\to\mb{X}}\leq \|T_{\mc{P}}\|_{\mb{X}[\mc{K}]\to\mb{X}[\mc{K}]}$.

For \ref{it:convexsettovector1}$\Rightarrow$\ref{it:convexsettovector2}, let $F\in\mb{X}[\mc{K}]$. We claim that $\langle F\rangle_E$ is a bounded set in $\F^n$ for all $E\in\mc{P}$. Indeed, since $T_E=\ind_ET_{\mc{P}}:\mb{X}\to\mb{X}$, the mapping $u\mapsto \|\ind_E u\|_{\mb{X}}$ is a norm on $\F^n$. Thus, there is a constant $c=c_{\mb{X},E}>0$ for which $\|\ind_E u\|_{\mb{X}}\geq c|u|$ for all $u\in\F^n$. Let $f\in S^0(\Omega;F\ind_E)$ not equal to $0$ a.e. and set $u=\frac{\langle f\rangle_E}{|\langle f\rangle_E|}$. Then
\[
c|\langle f\rangle_E|\leq |\langle f\rangle_E|\|\ind_E u\|_{\mb{X}}=\|T_E f\|_{\mb{X}}\leq\|T_E\|_{\mb{X}\to\mb{X}}\|F\|_{\mb{X}[\mc{K}]},
\]
proving the claim. We conclude that $\langle F\rangle_E\in\mc{K}_b$. Thus, by the John ellipsoid theorem (see \cite[Section~2.2.1]{NPTV17} for a more precise definition), there is a matrix $A_E\in\F^{n\times n}$ for which
\begin{equation}\label{eq:convexsettovector1}
\{A_Eu:|u|\leq 1\}\subseteq\langle F\rangle_E\subseteq \{A_Eu:|u|\leq n^{\frac{1}{2}}\}.
\end{equation}
Denoting the columns of $A_E$ by $(v^E_k)_{k=1}^n$, since $v^E_k=A_Ee_k\in\langle F\rangle_E$ by the first inclusion in \eqref{eq:convexsettovector1}, we have $v^E_k=\langle f^E_k\rangle_E$ for some $f^E_k\in S^0(\Omega;F\ind_E)$. Now, let $g\in S^0(\Omega;T_{\mc{P}}F)$. Then, we have $g=\sum_{E\in\mc{P}}g_E$ with $g_E=g\ind_E\in S^0(\Omega;T_EF)$, and by the second inclusion in \eqref{eq:convexsettovector1}, $g_E$ is of the form
\[
g_E(x)=\sum_{k=1}^n h^E_k(x)\langle f^E_k\rangle_E\ind_E(x),
\]
where $h^E=(h^E_1,\ldots,h^E_n)$ satisfies $|h^E(x)|\leq n^{\frac{1}{2}}$. Since $\mc{P}$ is pairwise disjoint, setting $h_k(x):=\sum_{E\in\mc{P}}h_k^E(x)\ind_E(x)$ and $f_k(x):=\sum_{E\in\mc{P}}f_k^E(x)\ind_E(x)$, we have $|h_k(x)|\leq n^{\frac{1}{2}}$ and
\[
g(x)=\sum_{E\in\mc{P}}\sum_{k=1}^n h_k^E(x)\langle f_k^E\rangle_E\ind_E(x)=\sum_{k=1}^n h_k(x)T_{\mc{P}}f_k(x).
\]
Hence, by the directional ideal property and the fact that $f_1,\ldots,f_n\in S^0(\Omega;F)$,
\begin{align*}
\|g\|_{\mb{X}}&\leq\sum_{k=1}^n\|h_kT_{\mc{P}}f_k\|_{\mb{X}}\leq n^{\frac{1}{2}}\sum_{k=1}^n\|T_{\mc{P}}f_k\|_{\mb{X}}\\
&\leq n^{\frac{1}{2}}\|T_{\mc{P}}\|_{\mb{X}\to\mb{X}}\sum_{k=1}^n\|f_k\|_{\mb{X}}
\leq n^{\frac{3}{2}}\|T_{\mc{P}}\|_{\mb{X}\to\mb{X}}\|F\|_{\mb{X}[\mc{K}]}.
\end{align*}
This proves that $T_{\mc{P}}F\in\mb{X}[\mc{K}]$ with
\[
\|T_{\mc{P}}F\|_{\mb{X}[\mc{K}]}\lesssim_n\|T_{\mc{P}}\|_{\mb{X}\to\mb{X}}\|F\|_{\mb{X}[\mc{K}]},
\]
as desired.    
\end{proof}
As a consequence, we obtain the following characterizations:
\begin{corollary}\label{cor:avopconvex}
Let $\mb{X}$ be a K\"othe reflexive $\F^n$-directional Banach function space over $\Omega$ for which both $\mb{X}$ and $\mb{X}'$ satisfy the component-wise saturation property. Then the following are equivalent:
\begin{enumerate}[(i)]
    \item\label{it:convexsettovector1'} $T_E:\mb{X}[\mc{K}]\to\mb{X}[\mc{K}]$;
    \item\label{it:convexsettovector2'} $T_E:\mb{X}'[\mc{K}]\to\mb{X}'[\mc{K}]$
    \item\label{it:convexsettovector3'} $T_E:\mb{X}\to\mb{X}$;
    \item\label{it:convexsettovector4'} $T_E:\mb{X}'\to\mb{X}'$
    \item $\ind_E u\in\mb{X},\mb{X}'$ for all $u\in\F^n$ and there is a $C_1>0$ such that for all $u\in\F^n$ there is a non-zero $v\in\F^n$ for which
    \[
    \mu(E)^{-1}\|\ind_E u\|_{\mb{X}}\|\ind_E v\|_{\mb{X}'}\leq C_1|u\cdot v|;
    \]
    \item  $\ind_E u\in\mb{X},\mb{X}'$ for all $u\in\F^n$ and there is a $C_2>0$ such that for all $v\in\F^n$ there is a non-zero $u\in\F^n$ for which
    \[
    \mu(E)^{-1}\|\ind_E u\|_{\mb{X}}\|\ind_E v\|_{\mb{X}'}\leq C_2|u\cdot v|.
    \]
    \item\label{it:reducingop7} $\ind_E u\in\mb{X},\mb{X}'$ for all $u\in\mb{F}^n$.
\end{enumerate}
Moreover, in this case all the associated optimal constants (where for \ref{it:reducingop7} the associated constant is $\|A_{\mb{X},E}A_{\mb{X}',E}\|_{\F^{n\times n}}$) are equivalent up to factors only depending on $n$.
\end{corollary}
\begin{proof}
The equivalences \ref{it:convexsettovector1'}$\Leftrightarrow$\ref{it:convexsettovector3'} and \ref{it:convexsettovector2'}$\Leftrightarrow$\ref{it:convexsettovector4'} follow from Theorem~\ref{thm:pairwisedisjointavconvex} with $\mc{P}=\{E\}$. The remaining equivalences follow from Corollary~\ref{cor:averagingoperator} and Proposition~\ref{prop:reducingmatrixavop}.
\end{proof}

\section{The Muckenhoupt condition}\label{sec:muckenhoupt}
We now specify to $\Omega=\R^d$ with the Lebesgue measure. By a cube $Q$ we will mean a cube in $\R^d$ whose sides are parallel to the coordinate axes. We define the Muckenhoupt condition of an $\F^n$-directional quasi-Banach function space through the averaging operators $T_Qf=\langle f\rangle_Q\ind_Q$.
\begin{definition}
We say that an $\F^n$-directional quasi-Banach function space $\mb{X}$ over $\R^d$ satisfies the \emph{Muckenhoupt condition} if $\ind_Q u\in\mb{X},\mb{X}'$ for all cubes $Q$ and all $u\in\F^n$, and $T_Q:\mb{X}\to\mb{X}$ for all cubes $Q$ with
\[
[\mb{X}]_A:=\sup_Q\|T_Q\|_{\mb{X}\to\mb{X}}<\infty.
\]
In this case we write $\mb{X}\in A$.
\end{definition}

By Corollary~\ref{cor:avopconvex} we have the following characterizations of the Muckenhoupt condition:
\begin{theorem}\label{thm:muckenhouptdef}
Let $\mb{X}$ be a K\"othe reflexive $\F^n$-directional Banach function space over $\R^d$ for which both $\mb{X}$ and $\mb{X}'$ are component-wise saturated. Then the following are equivalent:
\begin{enumerate}[(i)]
    \item\label{it:muckenhouptdef1} $\mb{X}\in A$;
    \item\label{it:muckenhouptdef2} $\mb{X}'\in A$;
    \item\label{it:muckenhouptdef3} $T_Q:\mb{X}[\mc{K}]\to\mb{X}[\mc{K}]$  for all cubes $Q$, and
    \[
    [\mb{X}]_{A[\mc{K}]}:=\sup_Q\|T_Q\|_{\mb{X}[\mc{K}]\to\mb{X}[\mc{K}]}<\infty;
    \]
    \item $T_Q:\mb{X}'[\mc{K}]\to\mb{X}'[\mc{K}]$ for all cubes $Q$, and
    \[
    [\mb{X}']_{A[\mc{K}]}:=\sup_Q\|T_Q\|_{\mb{X}'[\mc{K}]\to\mb{X}'[\mc{K}]}<\infty;
    \]
    \item\label{it:muckenhouptdef4} $\ind_Q u\in\mb{X},\mb{X}'$ for all cubes $Q$ and $u\in\F^n$, and
    \[
    [\mb{X}]_{A_R}:=\sup_Q|Q|^{-1}\|A_{\mb{X},Q}A_{\mb{X}',Q}\|_{\F^{n\times n}}<\infty;
    \]
    \item\label{it:muckenhouptdef5} $\ind_Q u\in\mb{X},\mb{X}'$ for all cubes $Q$ and $u\in\F^n$, and there is a $C_1>0$ such that for all $u\in\F^n$ there is a non-zero $v\in\F^n$ for which
    \[
    |Q|^{-1}\|\ind_Q u\|_{\mb{X}}\|\ind_Q v\|_{\mb{X}'}\leq C_1|u\cdot v|;
    \]
    \item\label{it:muckenhouptdef6}  $\ind_Q u\in\mb{X},\mb{X}'$ for all cubes $Q$ and $u\in\F^n$, and there is a $C_2>0$ such that for all $v\in\F^n$ there is a non-zero $u\in\F^n$ for which
    \[
    |Q|^{-1}\|\ind_Q u\|_{\mb{X}}\|\ind_Q v\|_{\mb{X}'}\leq C_2|u\cdot v|.
    \]
\end{enumerate}
Moreover, for the optimal $C_1$, $C_2$ we have
\[
[\mb{X}]_A=[\mb{X}']_A=C_1=C_2
\]
and
\[
[\mb{X}]_A\eqsim_n[\mb{X}]_{A_R}\eqsim_n[\mb{X}]_{A[\mc{K}]}\eqsim_n[\mb{X}']_{A[\mc{K}]}.
\]
\end{theorem}
\begin{proof}
The equivalences of \ref{it:muckenhouptdef1}-\ref{it:muckenhouptdef2}, \ref{it:muckenhouptdef5}-\ref{it:muckenhouptdef6} with equal constants follow from Corollary~\ref{cor:averagingoperator} and taking a supremum over all cubes $Q$. The equivalences of \ref{it:muckenhouptdef1} and \ref{it:muckenhouptdef3}-\ref{it:muckenhouptdef4} follow from Corollary~\ref{cor:avopconvex}.
\end{proof}

As is done in \cite{Ni24}, for $\mb{X}$ with $\ind_Q u\in\mb{X},\mb{X}'$ for all cubes $Q$ and all $u\in\F^n$, we can define the \emph{strong} Muckenhoupt condition $\mb{X}\in A_{\text{strong}}$ through imposing that
\[
[\mb{X}]_{A_{\text{strong}}}:=\sup_{\mc{P}}\|T_{\mc{P}}\|_{\mb{X}\to\mb{X}}<\infty,
\]
where the supremum is taken over all pairwise disjoint collections of cubes $\mc{P}$, and where $T_{\mc{P}}f=\sum_{Q\in\mc{P}}\langle f\rangle_Q\ind_Q$. Moreover, we also set
\[
[\mb{X}]_{A_{\text{strong}}[\mc{K}]}:=\sup_{\mc{P}}\|T_{\mc{P}}\|_{\mb{X}[\mc{K}]\to\mb{X}[\mc{K}]}
\]
Analogously to the Muckenhoupt condition, we have the following result:
\begin{proposition}
Let $\mb{X}$ be an $\F^n$-directional Banach function space over $\Omega$. Then $\mb{X}\in A_{\text{strong}}$ if and only if $T_{\mc{P}}:\mb{X}[\mc{K}]\to\mb{X}[\mc{K}]$ uniformly for all pairwise disjoint collections $\mc{P}$, with
\[
[\mb{X}]_{A_{\text{strong}}}\eqsim_n[\mb{X}]_{A_{\text{strong}}[\mc{K}]}.
\]
\end{proposition}
\begin{proof}
This follows directly from Theorem~\ref{thm:pairwisedisjointavconvex}.
\end{proof}

\begin{remark}
The proof of Theorem~\ref{thm:pairwisedisjointavconvex} fails for \emph{sparse} collections $\mc{P}$. Thus, it is not clear if such an equivalence holds for sparse collections as well.
\end{remark}

We define $L^1_{\text{loc}}(\R^d;\mc{K})$ to be the space of $F\in L^0(\R^d;\mc{K})$ for which $F\ind_Q\in L^1(\R^d;\mc{K})$ for all cubes $Q$.
\begin{proposition}\label{prop:locallyintegrablyboundedinA}
We have $L^1_{\text{loc}}(\R^d;\mc{K})\subseteq L^0(\R^d;\mc{K}_b)$.
\end{proposition}
\begin{proof}
Let $F\in L^1_{\text{loc}}(\R^d;\mc{K})$. Then for each cube $Q$ we have $F\ind_Q\in L^1(\R^d;\mc{K})$, so by Proposition~\ref{prop:matrixweightfatouimpliesbounded} we have $F(x)\in\mc{K}_b$ for a.e. $x\in Q$. Partitioning $\R^d$ into a countable collection of cubes, we conclude that $F(x)\in\mc{K}_b$ for a.e. $x\in\R^d$, as desired.
\end{proof}

\subsection{The convex-set valued maximal operator}
Given a collection of cubes $\mc{P}$ and $F\in L^0(\Omega;\mc{K})$ satisfying $F\in L^1_{\text{loc}}(\R^d;\mc{K})$, we define
\[
M^{\mc{K}}_{\mc{P}}F(x):=\sup_{Q\in\mc{P}}\langle F\rangle_Q\ind_Q(x),
\]
where the supremum is taken with respect to the partial ordering given by set inclusion in $\mc{K}$, i.e., $M^{\mc{K}}_{\mc{P}}F(x)$ is the smallest closed convex set containing $\bigcup_{Q\in\mc{P}}T_QF(x)$. When $\mc{P}$ is the collection of all cubes, then we drop the subscript $\mc{P}$ and denote the operator by $M^{\mc{K}}$.

Moreover, we define
\[
M_{\mc{P}}F(x):=\overline{\bigcup_{Q\in\mc{P}}\langle F\rangle_Q\ind_Q(x)},
\]
and drop subscript $\mc{P}$ when we consider the collection of all cubes. Unlike $M_{\mc{P}}^{\mc{K}}F$, $M_{\mc{P}}F$ is not necessarily convex-set valued. However, given a matrix weight $W:\R^d\to\F^{n\times n}$, we have
\[
\sup_{u\in M_{\mc{P}}^{\mc{K}}F(x)}|W(x)u|=\sup_{u\in M_{\mc{P}}F(x)}|W(x)u|.
\]
Using Proposition~\ref{prop:supattainmentXW}, this implies that if $X$ is a Banach function space with the Fatou property, then to prove $M_{\mc{P}}^{\mc{K}}:\mb{X}[\mc{K}]\to\mb{X}[\mc{K}]$, it suffices to consider $M_{\mc{P}}$ without having to worry about the convex hull.

Note that when $\mc{P}$ is finite, the set $M_{\mc{P}}F(x)=\bigcup_{Q\in\mc{P}}T_QF(x)$ is bounded, from which we conclude that $M^{\mc{K}}_{\mc{P}}F(x)\in\mc{K}_b$. However, if $\mc{P}$ is infinite, this need not be the case. Nonetheless, if an $\F^n$-directional quasi-Banach function space $\mb{X}$ over $\R^d$ satisfies the monotone convergence property, then we have
\[
M^{\mc{K}}:\mb{X}[\mc{K}]\to\mb{X}[\mc{K}]
\]
if and only if we have
\[
M^{\mc{K}}_{\mc{F}}:\mb{X}[\mc{K}]\to\mb{X}[\mc{K}]
\]
uniformly for all finite collections of cubes $\mc{F}$, and
\[
\|M^{\mc{K}}\|_{\mb{X}[\mc{K}]\to\mb{X}[\mc{K}]}=\sup_{\mc{F}}\|M_{\mc{K}}^{\mc{K}}\|_{\mb{X}[\mc{K}]\to\mb{X}[\mc{K}]}.
\]

For a matrix weight $W:\R^d\to\F^{n\times n}$, we define the \emph{Christ-Goldberg maximal operator} $M_W$ by
\[
M_Wf(x):=\sup_Q\langle |W(x)W^{-1}f|\rangle_Q\ind_Q(x),
\]
where the supremum is taken over all cubes $Q$ in $\R^d$. We analogously define $M_{\mc{P},W}f$ when the supremum is taken over a collection of cubes $\mc{P}$. Given a (quasi-)Banach function space $X$ over $\R^d$, we write $X[\F^n]$ for the space of $f\in L^0(\R^d;\F^n)$ satisfying $|f|\in X$, with $\|f\|_{X[\F^n]}:=\||f|\|_X$.
\begin{proposition}\label{prop:christgoldbergmax}
Let $X$ be a Banach function space over $\R^d$ with the Fatou property, let $W:\R^d\to\F^{n\times n}$ be a matrix weight, and let $\mc{P}$ be a collection of cubes. Then the following are equivalent:
\begin{enumerate}[(i)]
    \item\label{it:christgoldberg1} $M^{\mc{K}}_{\mc{P}}:X_W[\mc{K}]\to X_W[\mc{K}]$;
    \item\label{it:christgoldberg2} $M_{\mc{P},W}:X[\F^n]\to X$.
\end{enumerate}
Moreover, in this case we have
\[
\|M_{\mc{P},W}\|_{X[\F^n]\to X}\eqsim_n\|M_{\mc{P}}^{\mc{K}}\|_{X_W[\mc{K}]\to X_W[\mc{K}]}.
\]
\end{proposition}
\begin{proof}
For \ref{it:christgoldberg2}$\Rightarrow$\ref{it:christgoldberg1}, by monotone convergence, it suffices to prove the bound for $M^{\mc{K}}_{\mc{F}}$ uniformly in finite collections of cubes $\mc{F}\subseteq\mc{P}$. Let $F\in X_W[\mc{K}]$. By Proposition~\ref{prop:supattainmentXW} we have $M^{\mc{K}}_{\mc{F}}F\in X_W[\mc{K}]$ precisely when
\[
\sup_{u\in M^{\mc{K}}_{\mc{F}}F(x)}|W(x)u|=\sup_{u\in M_{\mc{F}}F(x)}|W(x)u|\in X.
\]
Since for $u\in M_{\mc{F}}F(x)$ there is a cube $Q\in\mc{F}$ containing $x$ and an $f\in S^0(\R^d;F)$ for which $u=\langle f\rangle_Q$, we have
\begin{equation}\label{eq:christgoldberg1}
|W(x)u|\leq\langle |W(x)f|\rangle_Q\leq\langle \sup_{u\in F}|W(x)u|\rangle_Q.
\end{equation}
Exactly as in \cite[Theorem~6.9]{BC23}, there is a constant $C_n$ and mappings $f_1,\ldots,f_n\in S^0(\R^d;F)$ for which
\[
F(x)\subseteq C_n\sum_{k=1}^n\mc{K}(f_k)(x)
\]
for a.e. $x\in\R^d$. This yields
\[
\langle \sup_{u\in F}|W(x)u|\rangle_Q\lesssim_n\sum_{k=1}^n\langle |W(x)f_k|\rangle_Q.
\]
Combining this with \eqref{eq:christgoldberg1}, we conclude that
\begin{align*}
\|M_{\mc{F}}^{\mc{K}}F\|_{X_W[\mc{K}]}&\lesssim_n\sum_{k=1}^n\|M_{\mc{F},W} (Wf_k)\|_X\leq\sum_{k=1}^n\|M_W\|_{X[\F^n]\to X}\||Wf_k|\|_X\\
&\lesssim_n\|M_W\|_{X[\F^n]\to X}\|F\|_{X_W[\mc{K}]},
\end{align*}
as desired.

For the converse implication \ref{it:christgoldberg1}$\Rightarrow$\ref{it:christgoldberg2}, let $f\in X[\F^n]$ not equal to $0$ a.e., let $Q\in\mc{P}$, let $x\in Q$, and set $\widetilde{f}:=\frac{|W(x)\langle W^{-1}f\rangle_Q|}{\langle|W(x)W^{-1}f|\rangle_Q}f$. Then
\[
\langle|W(x)W^{-1}f|\rangle_Q=|W(x)\langle W^{-1}\widetilde{f}\rangle_Q|\leq\sup_{u\in M^{\mc{K}}(\mc{K}(W^{-1}f))(x)}|W(x)u|.
\]
Hence, by Proposition~\ref{prop:supattainmentXW} and Proposition~\ref{prop:convexsetfequalsf},
\[
\|M_{\mc{P},W} f\|_X\leq \|M_{\mc{P}}^{\mc{K}}(\mc{K}(W^{-1}f))\|_{X_W[\mc{K}]}\leq\|M_{\mc{P}}^{\mc{K}}\|_{X_W[\mc{K}]\to X_W[\mc{K}]}\|f\|_{X[\F^n]}.
\]
The assertion follows.
\end{proof}

Just like in the scalar-valued case, to prove the boundedness of $M^{\mc{K}}$ and $M$, it suffices to prove the boundedness $M^{\mc{K}}_{\mc{D}^\alpha}$ for $3^d$ translated dyadic grids $(\mc{D}^\alpha)_{\alpha=1}^{3^d}$ in $\R^d$. More precisely, we have
\begin{equation}\label{eq:3dlatticeinclusion}
MF(x)\subseteq C_d\bigcup_{\alpha=1}^{3^d}M_{\mc{D}^\alpha}F(x)
\end{equation}
and
\[
M^{\mc{K}}F(x)\subseteq C_d\sum_{\alpha=1}^{3^d}M^{\mc{K}}_{\mc{D}^\alpha}F(x).
\]
For the latter inclusion, taking the sum is required, as the union of the $M^{\mc{K}}_{\mc{D}^\alpha}F(x)$ need not be in $\mc{K}$. This makes the inclusion \eqref{eq:3dlatticeinclusion} much more suited for working in weak-type spaces $\mb{X}_{\text{weak}}[\mc{K}]$, since the quasi-triangle inequality generally fails here. With a bit of abuse of notation, we write
\[
M:\mb{X}[\mc{K}]\to \mb{X}_{\text{weak}}[\mc{K}],
\]
if there is a $C>0$ for which
\[
\sup_{u\in\F^n}\|\ind_{MF^{-1}(\{u\})}u\|_{\mb{X}}\leq C\|F\|_{\mb{X}[\mc{K}]}
\]
for all $F\in\mb{X}[\mc{K}]$, and the smallest possible $C$ is denoted by $\|M\|_{\mb{X}[\mc{K}]\to \mb{X}_{\text{weak}}[\mc{K}]}$. As an analogue to \cite[Theorem~B]{Ni24}, we prove the following result:
\begin{theorem}\label{thm:mweak}
Let $\mb{X}$ be an $\F^n$-directional Banach function space over $\R^d$ with the Fatou property for which $\mb{X}$ and $\mb{X}'$ satisfy the component-wise saturation property. Consider the following statements:
\begin{enumerate}[(a)]
    \item\label{it:mweak1} $M^{\mc{K}}:\mb{X}[\mc{K}]\to\mb{X}[\mc{K}]$;
    \item\label{it:mweak2} $\mb{X}\in A_{\text{strong}}$;
    \item\label{it:mweak3} $M:\mb{X}[\mc{K}]\to\mb{X}_{\text{weak}}[\mc{K}]$;
    \item\label{it:mweak4} $\mb{X}\in A$.
\end{enumerate}
Then \ref{it:mweak1}$\Rightarrow$\ref{it:mweak2}$\Rightarrow$\ref{it:mweak3}$\Rightarrow$\ref{it:mweak4} with
\[
[\mb{X}]_{A}\leq\|M\|_{\mb{X}[\mc{K}]\to\mb{X}_{\text{weak}}[\mc{K}]}\lesssim_{d,n}[\mb{X}]_{A_{\text{strong}}}\leq\|M^{\mc{K}}\|_{\mb{X}[\mc{K}]\to\mb{X}[\mc{K}]}.
\]
Furthermore, if there is a $C\geq 1$ such that for all pairwise disjoint collections of cubes $\mc{P}$ and all $f\in\mb{X}$, $g\in\mb{X}'$ we have
\[
\sum_{Q\in\mc{P}}\|\ind_Q f\|_{\mb{X}}\|\ind_Q g\|_{\mb{X}'}\leq C\|f\|_{\mb{X}}\|g\|_{\mb{X}'},
\]
then \ref{it:mweak2}-\ref{it:mweak4} are equivalent, with
\[
[\mb{X}]_{A_{\text{strong}}}\leq C[\mb{X}]_A.
\]
\end{theorem}
\begin{proof}
For \ref{it:mweak1}$\Rightarrow$\ref{it:mweak2}, let $\mc{P}$ be a pairwise disjoint collection of cubes and let $f\in\mb{X}$. If $x\in\R^d$ satisfies $x\in Q$ for some $Q\in\mc{P}$, then $T_{\mc{P}}f(x)=\langle f\rangle_Q\ind_Q(x)\in \bigcup_Q\llangle f\rrangle_Q\ind_Q(x)$. Thus, we have
\[
T_{\mc{P}}f(x)\in M^{\mc{K}}(\mc{K}(f))(x)
\]
a.e. Hence, by the ideal property in $\mb{X}[\mc{K}]$ and Proposition~\ref{prop:convexsetfequalsf}, we have
\[
\|T_{\mc{P}}f\|_{\mb{X}}\leq\|M^{\mc{K}}(\mc{K}(f))\|_{\mb{X}[\mc{K}]}\leq\|M^{\mc{K}}\|_{\mb{X}[\mc{K}]\to\mb{X}[\mc{K}]}\|f\|_{\mb{X}}.
\]
Thus, $\mb{X}\in A_{\text{strong}}$ and $[\mb{X}]_{A_{\text{strong}}}\leq\|M^{\mc{K}}\|_{\mb{X}[\mc{K}]\to\mb{X}[\mc{K}]}$.

For \ref{it:mweak2}$\Rightarrow$\ref{it:mweak3}, first note that by \eqref{eq:3dlatticeinclusion}, it suffices to prove the bound for $M_{\mc{D}}$ for each dyadic grid $\mc{D}=\mc{D}^\alpha$. Moreover, by the monotone convergence property, it suffices to bound $M_{\mc{F}}$ uniformly for each finite $\mc{F}\subseteq\mc{D}$. Let $u\in\F^n$ be non-zero and suppose that $x\in M_{\mc{F}}F^{-1}(\{u\})$. Then there is a $Q\in\mc{F}$ with $x\in Q$ such that $u\in\langle F\rangle_Q$. We let $\mc{P}$ denote the maximal cubes in $\mc{F}$ for which $u\in\langle F\rangle_Q$. Then
\[
\ind_{M_{\mc{F}}F^{-1}(\{u\})}(x)u\in \sum_{Q\in\mc{P}}\langle F\rangle_Q\ind_Q(x)
\]
a.e., so, by the directional ideal property of $\mb{X}[\mc{K}]$ and the fact that $\mc{P}$ is pairwise disjoint,
\[
\|\ind_{M_{\mc{F}}F^{-1}(\{u\})}u\|_{\mb{X}}\leq\|T_{\mc{P}}F\|_{\mb{X}[\mc{K}]}\leq [\mb{X}]_{A_{\text{strong}}[\mc{K}]}\|F\|_{\mb{X}}.
\]
Taking a supremum over all $u\in\F^n$ now proves the result.

For \ref{it:mweak3}$\Rightarrow$\ref{it:mweak4}, let $Q$ be a cube and let $f\in\mb{X}$. If $x\in Q$, then $\langle f\rangle_Q\in M(\mc{K}(f))(x)$. Hence, for $u=\langle f\rangle_Q$ we have
\[
T_Qf(x)=\langle f\rangle_Q\ind_Q(x)\in \ind_{M(\mc{K}(f))^{-1}(\{u\})}(x)u
\]
a.e. It follows from the ideal property of $\mb{X}[\mc{K}]$ and Proposition~\ref{prop:convexsetfequalsf} that
\[
\|T_Qf\|_{\mb{X}}\leq\|M(\mc{K}(f))\|_{\mb{X}_{\text{weak}}[\mc{K}]}\leq\|M\|_{\mb{X}[\mc{K}]\to\mb{X}_{\text{weak}}[\mc{K}]}\|f\|_{\mb{X}},
\]
proving that $\mb{X}\in A$ with $[\mb{X}]_A\leq\|M\|_{\mb{X}[\mc{K}]\to\mb{X}_{\text{weak}}[\mc{K}]}$. The result follows.

Finally, let $\mb{X}\in A$ and let $\mc{P}$ be a pairwise disjoint collection of cubes. Then
\begin{align*}
\Big|\int_{\R^d}T_{\mc{P}}f\cdot g\,\mathrm{d}x\Big|
&\leq \sum_{Q\in\mc{P}}\int_{\R^d}\!|T_Qf\cdot \ind_Qg|\,\mathrm{d}x
\leq \sum_{Q\in\mc{P}}\|T_Qf\|_{\mb{X}}\|\ind_Q g\|_{\mb{X}'}\\
&\leq [\mb{X}]_A\sum_{Q\in\mc{P}}\|\ind_Q f\|_{\mb{X}}\|\ind_Q g\|_{\mb{X}'}
\leq C[\mb{X}]_A\|f\|_{\mb{X}}\|g\|_{\mb{X}'}.
\end{align*}
Thus, by Proposition~\ref{prop:normoutsideintegral} and Lemma~\ref{lem:lorentzluxemburg}, $\mb{X}\in A_{\text{strong}}$ with $[\mb{X}]_{A_{\text{strong}}}\leq C[\mb{X}]_A$.
\end{proof}
For matrix weighted spaces, this yields the following corollary:
\begin{corollary}\label{cor:propertygmatrixweight}
Let $X$ be a Banach function space over $\R^d$ for which there is a $C\geq 1$ such that for all pairwise disjoint $\mc{P}$ and all $f\in X$, $g\in X'$ we have
\begin{equation}\label{eq:propertyg}
\sum_{Q\in\mc{P}}\|\ind_Q f\|_X\|\ind_Q g\|_{X'}\leq C\|f\|_X\|g\|_{X'}.
\end{equation}
If $W:\R^d\to\F^{n\times n}$ is a matrix weight, then then the following are equivalent:
\begin{enumerate}[(i)]
    \item $X_W\in A$;
    \item $X_W\in A_{\text{strong}}$;
    \item $M:X_W[\mc{K}]\to (X_W)_{\text{weak}}[\mc{K}]$.
\end{enumerate}
Moreover, in this case we have
\[
[X_W]_A\eqsim [X_W]_{A_{\text{strong}}}\eqsim_{d,n}\|M\|_{X_W[\mc{K}]\to (X_W)_{\text{weak}}[\mc{K}]}.
\]
\end{corollary}
\begin{proof}
By Proposition~\ref{prop:matrixweightedspaceduality} we have
\begin{align*}
\sum_{Q\in\mc{P}}\|\ind_Q f\|_{X_W}\|\ind_Q g\|_{(X_W)'}
&=\sum_{Q\in\mc{P}}\|\ind_Q |Wf|\|_X\|\ind_Q |W^{-1}g|\|_{X'}\\
&\leq C\||Wf|\|_X\||W^{-1}g|\|_{X'}=C\|f\|_{X_W}\|g\|_{(X_W)'}.
\end{align*} 
Hence, the result follows from Theorem~\ref{thm:mweak}.
\end{proof}

If $X=L^p(\R^d)$, then \eqref{eq:propertyg} is satisfies with $C=1$ by H\"older's inequality. Thus, Corollary~\ref{cor:propertygmatrixweight} implies that
\[
\|M\|_{L^p_W(\R^d;\mc{K})\to L^p_W(\R^d;\mc{K})_{\text{weak}}}\eqsim_{d,n}[W]_p.
\]
This proves Corollary~\ref{cor:F}.

In the work \cite{Ni24} the property \eqref{eq:propertyg} is referred to as property $X\in\mc{G}$, and we refer the reader to this work for an overview of further spaces with this property.

Next, we prove a convex body domination result for $M^{\mc{K}}$:
\begin{theorem}\label{thm:convexbodydomofmax}
Let $\mc{D}$ be a dyadic grid and let $\mc{F}\subseteq\mc{D}$ be finite. For each $F\in L^1_{\text{loc}}(\R^d;\mc{K})$ there is a sparse collection $\mc{S}\subseteq\mc{F}$ for which
\[
M^{\mc{K}}_{\mc{F}}F(x)\subseteq 2n^{\frac{5}{2}} M^{\mc{K}}_{\mc{S}}F(x)
\]
for a.e. $x\in\R^d$.
\end{theorem}
\begin{proof}
Let $\mc{F}^\ast$ denote the maximal cubes in $\mc{F}$ and fix $Q_0\in\mc{F}^\ast$. By the John ellipsoid theorem there is an orthonormal basis $(e_k)_{k=1}^n$ of $\F^n$ and $\lambda_1,\ldots\lambda_n\geq 0$ such that
\[
\mc{E}(Q_0):=\Big\{\sum_{k=1}^n u_k\lambda_k e_k:|u|\leq 1\Big\}
\]
satisfies
\[
\mc{E}(Q_0)\subseteq\langle F\rangle_{Q_0}\subseteq n^{\frac{1}{2}}\mc{E}(Q_0).
\]
For each $k\in\{1,\ldots,n\}$ we let $\text{ch}_k(Q_0)$ denote the collection of maximal cubes $Q\in\mc{F}$ contained in $Q_0$ satisfying
\[
|\langle F\cdot e_k\rangle_{Q}|> 2n\langle |F\cdot e_k|\rangle_{Q_0},
\]
where, for a cube $P\in\mc{F}$,
\[
|\langle F\cdot e_k\rangle_P|:=\sup_{f\in S^0(\R^d;F)}|\langle f\cdot e_k\rangle_P|,\quad \langle |F\cdot e_k|\rangle_{P}:=\langle x\mapsto \sup_{u\in F(x)}|u\cdot e_k|\rangle_{P}.
\]
Moreover, we let $\text{ch}(Q_0)$ denote the maximal cubes in $\bigcup_{k=1}^n\text{ch}_k(Q_0)$. Applying this same procedure to each of the cubes in $\text{ch}(Q_0)$, we iteratively obtain the collections $\mc{S}_{j+1}(Q_0):=\bigcup_{Q\in\mc{S}_j(Q_0)}\text{ch}(Q)$, where $\mc{S}_0(Q_0):=\{Q_0\}$. Setting $\mc{S}:=\bigcup_{Q_0\in\mc{F}^\ast}\bigcup_{j=0}^\infty\mc{S}_j(Q_0)$, we note that for each $Q\in\mc{S}$ we have
\begin{align*}
\sum_{Q'\in\text{ch}(Q)}|Q'|
&\leq\sum_{k=1}^n\sum_{Q'\in\text{ch}_k(Q)}|Q'|\\
&\leq \sum_{k=1}^n\frac{1}{2n\langle |F\cdot e_k|\rangle_Q}\sum_{Q'\in\text{ch}_k(Q)}\int_{Q'}\sup_{u\in F(x)}|u\cdot e_k|\,\mathrm{d}x\\
&\leq\sum_{k=1}^n\frac{1}{2n\langle |F\cdot e_k|\rangle_Q}\int_Q\sup_{u\in F(x)}|u\cdot e_k|\,\mathrm{d}x=\frac{|Q|}{2}.
\end{align*}
Hence, $\mc{S}$ is sparse.

For each $Q'\in\mc{F}$ we let $\pi_{\mc{S}}(Q')$ denote the smallest cube $Q\in \mc{S}$ containing it. If $Q'\in\mc{F}$ satisfies $\pi_{\mc{S}}(Q')=Q$ and $Q'\neq Q$, then $Q'$ fails the stopping condition for all $k\in\{1,\ldots,n\}$, i.e.,
\[
|\langle F\cdot e_k\rangle_{Q'}|\leq 2n\langle |F\cdot e_k|\rangle_Q,
\]
where $(e_k)_{k=1}^n$ is the orthonormal basis associated to $\mc{E}(Q)$. Furthermore, this estimate remains true when $Q'=Q$. Thus, we have
\[
\langle F\rangle_{Q'}=\sum_{k=1}^n\langle F\cdot e_k\rangle_{Q'}e_k\subseteq 2n\sum_{k=1}^n\langle |F\cdot e_k|\rangle_Q\mc{K}(e_k)\subseteq 2n^{\frac{5}{2}}\langle F\rangle_Q.
\]
To see this last inclusion, pick an $f\in S^0(\R^d;F)$ for which $\langle |F\cdot e_k|\rangle_Q=\langle |f\cdot e_k|\rangle_Q$. Multiplying $f$ by a unit, we may assume that $f\cdot e_k\geq 0$. Since $\langle f\rangle_Q\in\langle F\rangle_Q\subseteq n^{\frac{1}{2}}\mc{E}(Q)$, we can write
\[
\langle f\rangle_Q=n^{\frac{1}{2}}\sum_{k=1}^n u_k\lambda_k e_k,
\]
where $|u_k|\leq 1$. Hence,
\[
\langle |F\cdot e_k|\rangle_Qe_k=(\langle f\rangle_Q\cdot e_k)e_k= n^{\frac{1}{2}} u_k\lambda_ke_k\in n^{\frac{1}{2}}\mc{E}(Q)\subseteq n^{\frac{1}{2}}\langle F\rangle_Q,
\]
as desired.

In conclusion, we obtain
\begin{align*}
\bigcup_{Q\in\mc{F}}\langle F\rangle_Q\ind_Q(x)&=\bigcup_{Q\in\mc{S}}\bigcup_{\substack{Q'\in\mc{F}\\\pi_{\mc{S}}(Q')=Q}}\langle F\rangle_{Q'}\ind_{Q'}(x)
\subseteq2n^{\frac{5}{2}} \bigcup_{Q\in\mc{S}}\langle F\rangle_Q\ind_Q(x)\subseteq2n^{\frac{5}{2}} M^{\mc{K}}_{\mc{S}}F(x).
\end{align*}
As $M^{\mc{K}}_{\mc{S}}F(x)\in\mc{K}$, we find that $M^{\mc{K}}_{\mc{F}}F(x)\subseteq 2n^{\frac{5}{2}} M^{\mc{K}}_{\mc{S}}F(x)$, as desired.
\end{proof}

We end this section with a sufficient condition for bounding $M^{\mc{K}}$ in a matrix weighted space $X_W$ and its dual $(X')_{W^{-1}}$. 

If $X$ is a Banach function space over $\R^d$ and $W:\R^d\to\F^{n\times n}$ is a matrix weight, then we define the \emph{Goldberg auxiliary maximal operator} $M_{X_W}$ as
\[
M_{X_W}f(x):=\sup_Q\Big(\frac{1}{\|\ind_Q\|_X}\int_Q\! |A_{X_W,Q}^{-1}W(x)f(x)|\,\mathrm{d}x\Big)\ind_Q(x).
\]
If $F\in L^0(\R^d;\mc{K})$, we similarly define
\[
M_{X_W}F(x):=\sup_Q\Big(\frac{1}{\|\ind_Q\|_X}\int_Q\! |A_{X_W,Q}^{-1}W(x)F(x)|\,\mathrm{d}x\Big)\ind_Q(x),
\]
where for a matrix $A\in\F^{n\times n}$ we have written $|AF(x)|:=\sup_{u\in F(x)}|Au|$. Our result is as follows:
\begin{theorem}\label{thm:auxiliarygoldbergreduction}
Let $X$ be a Banach function space over $\R^d$ with the Fatou property and let $W$ be a matrix weight. Suppose that $X\in A$, $X_W\in A$, and that
\[
M_{(X')_{W^{-1}}}:X[\F^n]\to X,\quad\text{and}\quad M_{X_W}:X'[\F^n]\to X'.
\]
Then
\[
M^{\mc{K}}:X_W[\mc{K}]\to X_W[\mc{K}]\quad\text{and}\quad M^{\mc{K}}:(X')_{W^{-1}}[\mc{K}]\to (X')_{W^{-1}}[\mc{K}].
\]
with
\begin{align*}
\max\big\{&\|M^{\mc{K}}\|_{X_W[\mc{K}]\to X_W[\mc{K}]},\|M^{\mc{K}}\|_{(X')_{W^{-1}}[\mc{K}]\to (X')_{W^{-1}}[\mc{K}]}\big\}\\
&\lesssim_{n,d}[X]_A[X_W]_A\|M_{(X')_{W^{-1}}}\|_{X[\F^n]\to X}\|M_{X_W}\|_{X'[\F^n]\to X'}.
\end{align*}
\end{theorem}
\begin{proof}
By symmetry, it suffices to bound $M^{\mc{K}}$ in $X_W[\mc{K}]$. Furthermore, by a $3^d$ lattice argument and monotone convergence, it suffices to bound $M^{\mc{K}}_{\mc{F}}$ uniformly for all finite collections $\mc{F}$. Using Theorem~\ref{thm:convexbodydomofmax}, we may further replace $\mc{F}$ by a finite sparse collection $\mc{S}$. Then, defining 
\[
E_Q:=Q\backslash \bigcup_{\substack{Q'\in\mc{S}\\ Q'\subsetneq Q}}Q'
\]
for $Q\in\mc{S}$, we have $E_Q\subseteq Q$ and $|Q|\leq 2|E_Q|$ by the definition of sparseness, and $(E_Q)_{Q\in\mc{S}}$ is pairwise disjoint. Thus, for $F\in X_W[\mc{K}]$, $g\in S^0(\R^d;M_{\mc{S}}^{\mc{K}}F)$, and $h\in (X')_{W^{-1}}$, we have
\begin{align*}
\int_{\R^d}&\!|g(x)\cdot h(x)|\,\mathrm{d}x
\leq\sum_{Q\in\mc{S}}\int_Q\!|\langle F\rangle_Q\cdot h(x)|\,\mathrm{d}x\\
&\leq \sum_{Q\in\mc{S}}\int_Q\!|A_{X_W,Q}A_{(X')_{W^{-1}},Q}\langle A_{(X')_{W^{-1}},Q}^{-1}F\rangle_Q\cdot A_{X_W,Q}^{-1}h(x)|\,\mathrm{d}x\\
&\leq \sum_{Q\in\mc{S}}\|A_{X_W,Q}A_{(X')_{W^{-1}},Q}\|\langle |A_{(X')_{W^{-1}},Q}^{-1}F|\rangle_Q \int_Q\! |A_{X_W,Q}^{-1}h|\,\mathrm{d}x\\
&\leq[X]_A[X_W]_{A_R}\sum_{Q\in\mc{S}}\Big(\frac{1}{\|\ind_Q\|_{X'}}\int_Q\! |A_{(X')_{W^{-1}},Q}^{-1}F|\,\mathrm{d}x\Big)\Big(\frac{1}{\|\ind_Q\|_X}\int_Q |A_{X_W,Q}^{-1}h|\,\mathrm{d}x\Big)|Q|\\
&\lesssim [X]_A[X_W]_{A_R}\sum_{Q\in\mc{S}}\int_{E_Q}\!M_{(X')_{W^{-1}}}(WF)M_{X_W}(W^{-1}h)\,\mathrm{d}x\\
&\leq [X]_A[X_W]_{A_R}\int_{\R^d}\!M_{(X')_{W^{-1}}}(WF)M_{X_W}(W^{-1}h)\,\mathrm{d}x\\
&\leq [X]_A[X_W]_{A_R}\|M_{(X')_{W^{-1}}}(WF)\|_{X}\|M_{X_W}(W^{-1}h)\|_{X'}.
\end{align*}
With the same reduction as is done in the proof of Proposition~\ref{prop:christgoldbergmax}, we find that
\[
\|M_{(X')_{W^{-1}}}(WF)\|_{X}\lesssim_n\|M_{(X')_{W^{-1}}}\|_{X[\F^n]\to X}\|F\|_{X_W[\mc{K}]}.
\]
Thus, taking a supremum over all $h\in (X')_{W^{-1}}$ of norm $1$ and over all $g\in S^0(\R^d;M^{\mc{K}}_{\mc{S}}F)$, we conclude that
\[
\|M^{\mc{K}}_{\mc{S}}F\|_{X_W[\mc{K}]}\lesssim_n [X]_A[X_W]_{A_R}\|M_{(X')_{W^{-1}}}\|_{X_W\to X}\|M_{X_W}\|_{(X')_{W^{-1}}\to X'}.
\]
Since by Theorem~\ref{thm:muckenhouptdef} we have $[X_W]_{A_R}\eqsim_n[X_W]_A$, the result follows.
\end{proof}
While this result seems rather convoluted, it does have very practical use. For example, when $X=L^p(\R^d)$, we have
\[
[X]_A=1,\quad [X_W]_A=[W]_p,
\]
and the operator $M_{L^p_W(\R^d;\F^n)}$ is a modification of the auxiliary maximal operator from \cite{Go03} used to bound the Christ-Goldberg maximal operator $M_W:L^p(\R^d;\F^n)\to L^p(\R^d)$ (and, thus, $M^{\mc{K}}:L^p_W(\R^d;\mc{K})\to L^p_W(\R^d;\mc{K})$ by Proposition~\ref{prop:christgoldbergmax}). This version of the auxiliary operator was used in \cite{KN24} to prove bounds for both the convex-set valued maximal operator and the convex body sparse operator, exemplifying its utility. The above result gives a qualitative way of extending this argument to more general spaces, such as matrix-weighted variable Lebesgue spaces.

\section{Extrapolation}\label{sec:extrapolation}

\begin{theorem}\label{thm:BFSextrapolation}
Let $1\leq p\leq\infty$, let $V$ be a set, and let $S:V\to L^0(\R^d;\F^n)$. Suppose
\[
T:\bigcup_{W\in A_p}S^{-1}(L^p_W(\R^d;\F^n))\to L^0(\R^d;\F^n)
\]
is a map for which there is an increasing function $\phi:[0,\infty)\to[0,\infty)$ such that for all $W\in A_p$ and all $f\in V$ with $Sf\in L^p_W(\R^d;\F^n)$ we have
\[
\|Tf\|_{L^p_W(\R^d;\F^n)}\leq\phi([W]_p)\|Sf\|_{L^p_W(\R^d;\F^n)}.
\]
Let $\mb{X}$ be a K\"othe reflexive $\F^n$-directional Banach function space over $\R^d$  for which both $\mb{X}$ and $\mb{X}'$ are component-wise saturated, and
\[
M^{\mc{K}}:\mb{X}[\mc{K}]\to \mb{X}[\mc{K}],\quad M^{\mc{K}}:\mb{X}'[\mc{K}]\to \mb{X}'[\mc{K}].
\]
Then $Tf$ is well-defined for all $f\in V$ with $Sf\in\mb{X}$, and
\[
\|Tf\|_{\mb{X}}\lesssim_n \phi(C_n\|M^{\mc{K}}\|_{\mb{X}[\mc{K}]\to \mb{X}[\mc{K}]}^{\frac{1}{p'}}\|M^{\mc{K}}\|^{\frac{1}{p}}_{\mb{X}'[\mc{K}]\to \mb{X}'[\mc{K}]})\|Sf\|_{\mb{X}}.
\]
When $p=\infty$ or $p=1$, we can omit the bound $M^{\mc{K}}:\mb{X}'[\mc{K}]\to \mb{X}'[\mc{K}]$ or $M^{\mc{K}}:\mb{X}[\mc{K}]\to \mb{X}[\mc{K}]$ respectively.
\end{theorem}
The proof follows from two abstract theorems. The first is the following generalization of the Rubio de Francia algorithm of \cite[Theorem~7.6]{BC23}. Recall that we have defined the sum $F+G$ through the closed Minkowski sum
\[
(F+G)(x):=\overline{\{u+v:u\in F(x),\, v\in G(x)\}}.
\]
Unlike \cite[Theorem~7.6]{BC23}, we do not require the operator $T$ to be monotone.
\begin{theorem}[Rubio de Francia algorithm]\label{thm:rdfdef}
Let $\mb{X}$ be an $\F^n$-directional Banach function space over $\R^d$ with the Fatou property, and let $T:\mb{X}[\mc{K}]\to\mb{X}[\mc{K}]$ be an operator that is sublinear in the sense that
\[
T(F+G)(x)\subseteq TF(x)+TG(x),\quad T(\lambda F)(x)=\lambda TF(x)
\]
a.e. for all $F,G\in\mb{X}[\mc{K}]$ and $\lambda\in\F$. Then there is a mapping $\mc{R}:\mb{X}[\mc{K}]\to\mb{X}[\mc{K}]$ for which $\|\mc{R}\|_{\mb{X}[\mc{K}]\to\mb{X}[\mc{K}]}\leq 2$, and for all $F\in\mb{X}[\mc{K}]$
\begin{enumerate}[(a)]
    \item\label{it:RdF1} $F(x)\subseteq\mc{R}F(x)$ a.e.;
    \item\label{it:RdF2} $T(\mc{R}F)(x)\subseteq 2\|T\|_{\mb{X}[\mc{K}]\to\mb{X}[\mc{K}]}\mc{R}F(x)$ a.e.
\end{enumerate}
\end{theorem}
For the proof, we require a lemma on nested sets in $\mc{K}$:
\begin{lemma}\label{lem:zerointersection}
Let $(B_k)_{k\geq 1}$ be a sequence in $\mc{K}$ satisfying $B_{k+1}\subseteq B_k$ for all $k\geq 1$. If $\bigcap_{k=1}^\infty B_k=\{0\}$, then there is a $K\geq 1$ such that for all $k\geq K$ the set $B_k$ is compact. Moreover, if $(x_k)_{k\geq 1}$ is a sequence for which $x_k\in B_k$ for all $k\geq 1$, then $x_k\to 0$.
\end{lemma}
\begin{proof}
Arguing by contrapositive, suppose that $B_k$ is unbounded for all $k\geq 1$. Set $S^{n-1}:=\{u\in\F^n:|u|=1\}$ and, for $k\geq 1$,
\[
C_k:=S^{n-1}\cap B_k.
\]
Then $C_k$ is an intersection of closed sets, and hence, closed. Moreover, as each $B_k$ is unbounded, $C_k\neq\emptyset$. As $C_{k+1}\subseteq C_k$ for all $k\geq 1$ and $S^{n-1}$ is compact, we have
\[
S^{n-1}\cap \bigcap_{k=1}^\infty B_k=\bigcap_{k=1}^\infty C_k\neq\emptyset.
\]
Thus, $\bigcap_{k=1}^\infty B_k\neq \{0\}$, as asserted.

For the second assertion, let $(x_{j_k})_{k\geq 1}$ be a subsequence. Since for $j$ large enough, $B_j$ is compact, this subsequence has a further convergent subsequence with a limit in this $B_j$. As this latter sequence eventually lies in $B_j$ for all $j\geq 1$, the limit is contained in $\bigcap_{k=1}^\infty B_k=\{0\}$. Hence, every subsequence of $(x_k)_{k\geq 1}$ has a further subsequence converging to $0$. This proves that $x_k\to 0$, as desired.
\end{proof}

\begin{proof}[Proof of Theorem~\ref{thm:rdfdef}]
Let $F\in\mb{X}[\mc{K}]$ and define
\[
F_k(x):=2^{-k}\|T\|_{\mb{X}[\mc{K}]\to\mb{X}[\mc{K}]}^{-k}T^kF(x),
\]
where $T^k$ is iteratively defined as $T^0F:=F$ and $T^{k+1}F:=T(T^kF)$. Then
\[
\sum_{k=0}^\infty\|F_k\|_{\mb{X}[\mc{K}]}\leq \sum_{k=0}^\infty 2^{-k}\|F\|_{\mb{X}[\mc{K}]}=2\|F\|_{\mb{X}[\mc{K}]},
\]
so, by Theorem~\ref{thm:directionalrieszfischer}, the series
\begin{equation}\label{eq:rdfdef1}
\mc{R}F(x):=\sum_{k=0}^\infty F_k(x)=\overline{\bigcup_{K=0}^\infty\Big\{\sum_{k=0}^K u_k:u_k\in F_k(x)\Big\}}
\end{equation}
satisfies $\mc{R}F\in\mb{X}[\mc{K}]$ with
\[
\|\mc{R}F\|_{\mb{X}[\mc{K}]}\leq 2\|F\|_{\mb{X}[\mc{K}]}.
\]
The assertion \ref{it:RdF1} follows from the fact that the $k=0$ term of the union in \eqref{eq:rdfdef1} is equal to $F(x)$. For \ref{it:RdF2}, sublinearity of $T$ implies that
\[
TF_k(x)=2^{-k}\|T\|_{\mb{X}[\mc{K}]\to\mb{X}[\mc{K}]}^{-k}T^{k+1}F(x)=2\|T\|_{\mb{X}[\mc{K}]\to\mb{X}[\mc{K}]}F_{k+1}(x)
\]
a.e. so that, since for all $J\geq 1$ we have $\mc{R}F=\sum_{k=0}^J F_k+S_J$ where $S_J:=\sum_{k=J+1}^\infty F_k$,
\begin{equation}\label{eq:rdfdef2}
\begin{split}
T(\mc{R}F)(x)&\subseteq 2\|T\|_{\mb{X}[\mc{K}]\to\mb{X}[\mc{K}]}\sum_{k=0}^J F_{k+1}(x)+TS_J(x)\\
&\subseteq 2\|T\|_{\mb{X}[\mc{K}]\to\mb{X}[\mc{K}]}\mc{R}F(x)+TS_J(x).
\end{split}
\end{equation}
Note that
\begin{align*}
\Big\|\bigcap_{J=1}^\infty TS_J\Big\|_{\mb{X}[\mc{K}]}&\leq \|TS_J\|_{\mb{X}[\mc{K}]}\leq\|T\|_{\mb{X}[\mc{K}]\to\mb{X}[\mc{K}]}\Big\|\sum_{k=J+1}^\infty F_k\Big\|_{\mb{X}[\mc{K}]}\\
&\leq\|T\|_{\mb{X}[\mc{K}]\to\mb{X}[\mc{K}]}\sum_{k=J+1}^\infty\|F_k\|_{\mb{X}[\mc{K}]}\to 0
\end{align*}
as $J\to\infty$ and, hence, $\bigcap_{J=1}^\infty TS_J(x)=\{0\}$. 

Since $TS_J(x)\in\mc{K}$, by Lemma~\ref{lem:zerointersection} there is a  $J_0\geq 1$ such that for all $J\geq J_0$, the set $TS_J(x)$ is compact. Let $u\in T(\mc{R}F)(x)$. Since the Minkowski sum of a closed and a compact set is closed, this means that, by \eqref{eq:rdfdef2}, for all $J\geq J_0$ we can write $u=v_J+w_J$, with $v_J\in 2\|T\|_{\mb{X}[\mc{K}]\to\mb{X}[\mc{K}]}\mc{R}F(x)$ and $w_J\in TS_J(x)$. Then, since $w_J\to 0$ as $J\to\infty$, we have $v_J=u-w_J\to u$. Hence, $u\in 2\|T\|_{\mb{X}[\mc{K}]\to\mb{X}[\mc{K}]}\mc{R}F(x)$. This proves \ref{it:RdF2}, as desired.
\end{proof}
The second theorem we need for the extrapolation result is the following application of the Rubio de Francia algorithm:
\begin{theorem}\label{thm:rdfmainestimate}
Let $\mb{X}$ be a K\"othe reflexive $\F^n$-directional Banach function space over $\R^d$ for which both $\mb{X}$ and $\mb{X}'$ are component-wise saturated, and
\[
M^{\mc{K}}:\mb{X}[\mc{K}]\to \mb{X}[\mc{K}],\quad M^{\mc{K}}:\mb{X}'[\mc{K}]\to \mb{X}'[\mc{K}].
\]
Let $1\leq p\leq \infty$. Then for all $f\in\mb{X}$, $g\in\mb{X}'$, and $\varepsilon>0$, there exists a matrix weight $W\in A_p$ with the following properties:
\begin{itemize}
    \item $[W]_p\lesssim_n\|M^{\mc{K}}\|^{\frac{1}{p}}_{\mb{X}[\mc{K}]\to \mb{X}[\mc{K}]}\|M^{\mc{K}}\|^{\frac{1}{p'}}_{\mb{X}'[\mc{K}]\to \mb{X}'[\mc{K}]}$;
    \item $\|f\|_{L^p_W(\R^d;\F^n)}\|g\|_{L^{p'}_{W^{-1}}(\R^d;\F^n)}\lesssim_n(\|f\|_{\mb{X}}+\varepsilon)(\|g\|_{\mb{X}'}+\varepsilon)$.
\end{itemize}
When $p=\infty$ or $p=1$, we can omit the bound $M^{\mc{K}}:\mb{X}'[\mc{K}]\to \mb{X}'[\mc{K}]$ or $M^{\mc{K}}:\mb{X}[\mc{K}]\to \mb{X}[\mc{K}]$ respectively.
\end{theorem}
For the proof, we adopt a norm function approach. We call a mapping $\rho:\R^d\times\F^n\to[0,\infty)$ a norm function if for a.e. $x\in\R^d$ the mapping $u\mapsto\rho(x,u)$ is a norm, and for all $u\in\F^n$ the mapping $x\mapsto\rho(x,u)$ is measurable. We define $L^p_\rho(\R^d)$ as the space of functions $f\in L^0(\R^d;\F^n)$ for which $\rho(x,f(x))\in L^p(\R^d)$, with
\[
\|f\|_{L^p_\rho(\R^d)}:=\Big(\int_{\R^d}\!\rho(x,f(x))^p\,\mathrm{d}x\Big)^{\frac{1}{p}}.
\]
Then, for $1\leq p\leq\infty$, $L^p_\rho(\R^d)$ is an $\F^n$-directional Banach function space over $\R^d$. Moreover, $L^p_\rho(\R^d)'=L^{p'}_{\rho^\ast}(\R^d)$, where $\rho^\ast:\R^d\times\F^n\to[0,\infty)$ is defined as
\[
\rho^\ast(x,v)=\sup_{u\in\F^n\backslash\{0\}}\frac{|u\cdot v|}{\rho(x,u)}.
\]
If $W:\R^d\to\F^{n\times n}$ is a matrix weight, then
\[
\rho_W(x,u):=|W(x)u|
\]
is a norm function, and $L^p_{\rho_W}(\R^d)=L^p_W(\R^d;\F^n)$. Conversely, for any norm function $\rho$ it follows from the John ellipsoid theorem that there is a matrix weight $W$ for which
\begin{equation}\label{eq:normfunctionisweight}
\rho_W(x,u)\leq\rho(x,u)\leq n^{\frac{1}{2}}\rho_W(x,u).
\end{equation}
We will write $\rho\in A_p$ if there is a constant $C\geq 1$ such that for all cubes $Q$ and all $u\in\F^n$ there is a non-zero $v\in\F^n$ such that
\[
\Big(\frac{1}{|Q|}\int_Q\!\rho(x,u)^p\,\mathrm{d}x\Big)^{\frac{1}{p}}\Big(\frac{1}{|Q|}\int_Q\!\rho^\ast(x,v)^{p'}\,\mathrm{d}x\Big)^{\frac{1}{p'}}\leq C|u\cdot v|,
\]
where for an infinite exponent the integral is interpreted as an essential supremum. The smallest possible $C$ is denoted by $[\rho]_p$. As this condition is precisely the condition $L^p_{\rho}(\R^d)\in A$, it follows from Theorem~\ref{thm:muckenhouptdef} that $\rho\in A_p$ if and only if $\rho^\ast\in A_{p'}$, with $[\rho]_p=[\rho^\ast]_{p'}$. 

If $W$ is a matrix weight, then $[\rho_W]_p=[W]_p$. In general, we have $\rho\in A_p$ if and only if $W\in A_p$ for $W$ satisfying \eqref{eq:normfunctionisweight}, with $[\rho]_p\eqsim_n[W]_p$.

Before we get to the proof of Theorem~\ref{thm:rdfmainestimate}, we will need an interpolation result. The following lemma is a version of the direct implication in \cite[Theorem~7.3]{BC23}, as well as a norm function version of \cite[Proposition~8.7]{BC23}:
\begin{lemma}\label{lem:interpolationnorm}
Let $F_0,F_1\in L^1_{\text{loc}}(\R^d;\mc{K})$ have a non-empty interior a.e., and suppose there are $C_0,C_1>0$ such that for a.e. $x\in\R^d$ we have
\[
M^{\mc{K}}F_0(x)\subseteq C_0 F_0(x),\quad M^{\mc{K}}F_1(x)\subseteq C_1 F_1(x).
\]
Then the norm functions
\[
\rho_0(x,u):=
\sup_{v\in F_0(x)}|u\cdot v|,\quad \rho_1(x,u):=
\sup_{v\in F_1(x)}|u\cdot v|
\]
satisfy $\rho_0,\rho_1\in A_1$ with $[\rho_0]_1\lesssim_n C_0$, $[\rho_1]\lesssim_n C_1$.

Moreover, for all $1\leq p\leq\infty$, there exists a norm function $\rho$ satisfying
\begin{enumerate}[(a)]
    \item\label{it:normfunctiona11} $\rho\in A_p$ with $[\rho]_p\lesssim_n C_0^{\frac{1}{p'}}C_1^{\frac{1}{p}}$;
    \item\label{it:normfunctiona12} $\rho(x,u)\leq\rho_0^\ast(x,u)^{\frac{1}{p'}}\rho_1(x,u)^{\frac{1}{p}}$ for all $u\in\F^n$ for a.e. $x\in\R^d$;
    \item\label{it:normfunctiona13} $\rho^\ast(x,v)\leq\rho_0(x,v)^{\frac{1}{p'}}\rho^\ast_1(x,v)^{\frac{1}{p}}$ for all $v\in\F^n$ for a.e. $x\in\R^d$.
\end{enumerate}
\end{lemma}
\begin{proof}
For the first assertion, by Proposition~\ref{prop:locallyintegrablyboundedinA} we have $F_0(x),F_1(x)\in\mc{K}_b$ for a.e. $x\in\R^d$, proving that $\rho_0$ and $\rho_1$ are finite a.e. Thus, we may apply \cite[Theorem~7.3]{BC23} to conclude that $\rho_0,\rho_1\in A_1$ with
\[
[\rho_0]_1\lesssim_n C_0,\quad [\rho_1]_1\lesssim_n C_1.
\]
We note that their result is stated only for $\F=\R$, but it holds mutatis mutandis in the case $\F=\C$ as explained in \cite[Subsection~9.4]{DKPS24}.

For the second assertion, set
\[
K(x):=\mc{K}(\{u\in\F^n:\rho_0^\ast(x,u)^{\frac{1}{p'}}\rho_1(x,u)^{\frac{1}{p}}\leq 1\}).
\]
For a.e. $x\in\R^d$, this set is an absorbing set in $\mc{K}_b$ (namely, for those $x\in\R^d$ for which both $F_0(x)$ and $F_1(x)$ have a non-empty interior), so we may define $\rho:\R^d\times\F^n\to[0,\infty)$ as the Minkowski functional
\[
\rho(x,u):=\inf\{t>0:t^{-1}u\in K(x)\}.
\]
Then, since $t^{-1}u\in\{u\in\F^n:\rho_0^\ast(x,u)^{\frac{1}{p'}}\rho_1(x,u)^{\frac{1}{p}}\leq 1\}$ for $t=\rho^\ast_0(x,u)^{\frac{1}{p'}}\rho_1(x,u)^{\frac{1}{p}}$, we have
\[
\rho(x,u)\leq \rho^\ast_0(x,u)^{\frac{1}{p'}}\rho_1(x,u)^{\frac{1}{p}},
\]
proving \ref{it:normfunctiona12}. For \ref{it:normfunctiona13}, note that
\[
|u\cdot v|=|u\cdot v|^{\frac{1}{p'}}|u\cdot v|^{\frac{1}{p}}\leq(\rho^\ast_0(x,u)\rho_0(x,v))^{\frac{1}{p'}}(\rho_1(x,u)\rho_1^\ast(x,v))^{\frac{1}{p}}.
\]
It follows that if $t>0$ satisfies $t^{-1}u\in K(x)$, i.e., $t^{-1}u=\sum_{k=1}^K\theta_ku_k$ for $\sum_{k=1}^K\theta_k=1$, $\rho^\ast_0(x,u_k)^{\frac{1}{p'}}\rho_1(x,u_k)^{\frac{1}{p}}\leq 1$, we have
\[
t^{-1}|u\cdot v|\leq\sum_{k=1}^K\theta_k|u_k\cdot v|\leq\sum_{k=1}^K\theta_k \rho^\ast_0(x,v)^{\frac{1}{p'}}\rho_1(x,v)^{\frac{1}{p}}=\rho^\ast_0(x,v)^{\frac{1}{p'}}\rho_1(x,v)^{\frac{1}{p}}.
\]
This implies that $|u\cdot v|\leq\rho^\ast_0(x,v)^{\frac{1}{p'}}\rho_1(x,v)^{\frac{1}{p}}\rho(x,u)$ and, hence,
\[
\rho^\ast(x,v)=\sup_{u\neq 0}\frac{|u\cdot v|}{\rho(x,u)}\leq\rho_0(x,v)^{\frac{1}{p'}}\rho_1^\ast(x,v)^{\frac{1}{p}},
\]
as desired. Finally, for \ref{it:normfunctiona11}, it follows from \cite[Remark~2.16,\, Remark~8.10]{BC23}  that $\rho\in A_p$ with the desired bound. The assertion follows.
\end{proof}

The reason we assume the component-wise saturation property in Theorem~\ref{thm:rdfmainestimate}, is because the above lemma only applies when we have mappings $F_0,F_1$ that have non-empty interior. Proposition~\ref{prop:componentsaturation} gives us a way to perturb the mappings $\mc{K}(f)$ and $\mc{K}(g)$ to mappings that have this property. As a matter of fact, we will use the following lemma:
\begin{lemma}\label{lem:componentsaturation}
Let $\mb{X}$ be an $\F^n$-directional quasi-Banach function space over $\Omega$. Then $\mb{X}$ satisfies the component-wise saturation property if and only if there is a measurable Hermitian and positive definite matrix-valued mapping $U:\Omega\to\F^{n\times n}$ with $U\overline{B}\in\mb{X}[\mc{K}]$.
\end{lemma}
\begin{proof}
A mapping $g\in L^0(\Omega;\F^n)$ satisfies $g\in S^0(\Omega;U\overline{B})$ precisely when $|U(x)^{-1}g(x)|\leq 1$ a.e., so the result follows from Proposition~\ref{prop:componentsaturation}.
\end{proof}

\begin{proof}[Proof of Theorem~\ref{thm:rdfmainestimate}]
It suffices to construct a $\rho\in A_p$ instead of $W\in A_p$ for which the desired bounds hold. By applying Theorem~\ref{thm:rdfdef} to $M^{\mc{K}}$ in $\mb{X}$ and $\mb{X}'$, we obtain the respective operators $\mc{R}:\mb{X}[\mc{K}]\to\mb{X}[\mc{K}]$ and $\mc{R}':\mb{X}'[\mc{K}]\to\mb{X}'[\mc{K}]$ satisfying \ref{it:RdF1} and \ref{it:RdF2}. By the component-wise saturation property and Lemma~\ref{lem:componentsaturation}, we can pick invertible matrix-valued mappings $U,V$ satisfying
\[
\|U\overline{B}\|_{\mb{X}[\mc{K}]}=1,\quad \|V\overline{B}\|_{\mb{X}'[\mc{K}]}=1.
\]
Applying Lemma~\ref{lem:interpolationnorm} with $F_0=\mc{R}(\mc{K}(f)+\varepsilon U\overline{B})$, $F_1=\mc{R}'(\mc{K}(g)+\varepsilon V\overline{B})$, we obtain a norm function $\rho\in A_p$ for which
\[
[\rho]_p\lesssim_n\|M^{\mc{K}}\|_{\mb{X}[\mc{K}]\to\mb{X}[\mc{K}]}^{\frac{1}{p'}}\|M^{\mc{K}}\|_{\mb{X}'[\mc{K}]\to\mb{X}'[\mc{K}]}^{\frac{1}{p}},
\]
and \ref{it:normfunctiona12} and \ref{it:normfunctiona13} hold for
\[
\rho_0(x,u):=\sup_{v\in\mc{R}(\mc{K}(f)+\varepsilon U\overline{B})(x)}|u\cdot v|,\quad \rho_1(x,u)=\sup_{v\in\mc{R}'(\mc{K}(g)+\varepsilon V\overline{B})(x)}|u\cdot v|.
\]
By applying Theorem~\ref{thm:filippov} with
\[
\phi(x,v)=|f(x)\cdot v|,\quad h(x)=\sup_{v\in\mc{R}'(\mc{K}(g)+\varepsilon V\overline{B})(x)}|f(x)\cdot v|,
\]
we can find a $k\in S^0(\R^d;\mc{R}'(\mc{K}(g)+\varepsilon V\overline{B}))$ for which
\[
|f(x)\cdot k(x)|=\rho_1(x,f(x)).
\]
Moreover, since $f(x)\in\mc{R}(\mc{K}(f)+\varepsilon U\overline{B})(x)$ a.e., we have $\rho_0(x,u)\geq|u \cdot f(x)|$ a.e., and, hence,
\[
\rho_0^\ast(x,f(x))\leq 1.
\]
Thus, we have
\[
\rho(x,f(x))^p\leq \rho^\ast_0(x,f(x))^{p-1}\rho_1(x,f(x))\leq |f(x)\cdot k(x)|,
\]
or $\rho(x,f(x))=\rho^\ast_0(x,f(x))\leq 1$ if $p=\infty$. We conclude that $f\in L^p_\rho(\R^d)$ with
\begin{align*}
\int_{\R^d}\!\rho(x,f(x))^p\,\mathrm{d}x&\leq\int_{\R^d}\!|f(x)\cdot k(x)|\,\mathrm{d}x\leq\|f\|_{\mb{X}}\|k\|_{\mb{X}'}\\
&\leq\|f\|_{\mb{X}}\|\mc{R}'(\mc{K}(g)+\varepsilon V\overline{B})\|_{\mb{X}'[\mc{K}]}\leq2\|f\|_{\mb{X}}(\|g\|_{\mb{X}'}+\varepsilon),
\end{align*}
or $\|f\|_{L^\infty_\rho(\R^d)}\leq 1$ when $p=\infty$. For the term involving $g$, an analogous argument as the one for $f$ with the roles of $\rho_0$, $\rho_1$, and of $p$, $p'$ reversed, shows that $g\in L^{p'}_{\rho^\ast}(\R^d)$ with
\[
\|g\|_{L^{p'}_{\rho^\ast}(\R^d)}\leq 2^{\frac{1}{p'}}(\|f\|^{\frac{1}{p'}}_{\mb{X}}+\varepsilon)\|g\|^{\frac{1}{p'}}_{\mb{X}'}.
\]
In conclusion, $f\in L^p_\rho(\R^d)$, $g\in L^{p'}_{\rho^\ast}(\R^d)$, and
\[
\|f\|_{L^p_\rho(\R^d)}\|g\|_{L^{p'}_{\rho^\ast}(\R^d)}\leq 2(\|f\|_{\mb{X}}+\varepsilon)(\|g\|_{\mb{X}'}+\varepsilon),
\]
as asserted.
\end{proof}

\begin{proof}[Proof of Theorem~\ref{thm:BFSextrapolation}]
Let $f\in V$ for which $Sf\in\mb{X}$, and let $g\in\mb{X}'$. By Theorem~\ref{thm:rdfmainestimate} we can pick a $W\in A_p$ for which
\[
[W]_p\leq C_n\|M^{\mc{K}}\|_{\mb{X}[\mc{K}]\to\mb{X}[\mc{K}]}^{\frac{1}{p'}}\|M^{\mc{K}}\|_{\mb{X}'[\mc{K}]\to\mb{X}'[\mc{K}]}^{\frac{1}{p}}
\]
and
\[
\|Sf\|_{L^p_W(\R^d;\F^n)}\|g\|_{L^{p'}_{W^{-1}}(\R^d;\F^n)}\lesssim_n(\|Sf\|_{\mb{X}}+\varepsilon)(\|g\|_{\mb{X}'}+\varepsilon).
\]
Thus,
\begin{align*}
\int_{\R^d}\!|Tf\cdot g|\,\mathrm{d}x&\leq \|Tf\|_{L^p_W(\R^d;\F^n)}\|g\|_{L^{p'}_{W^{-1}}(\R^d;\F^n)}
\leq\phi([W]_p)\|Sf\|_{L^p_W(\R^d;\F^n)}\|g\|_{L^{p'}_{W^{-1}}(\R^d;\F^n)}\\
&\lesssim_n\phi(C_n\|M^{\mc{K}}\|_{\mb{X}[\mc{K}]\to\mb{X}[\mc{K}]}^{\frac{1}{p'}}\|M^{\mc{K}}\|_{\mb{X}'[\mc{K}]\to\mb{X}'[\mc{K}]}^{\frac{1}{p}})(\|Sf\|_{\mb{X}}+\varepsilon)(\|g\|_{\mb{X}'}+\varepsilon).
\end{align*}
As $\mb{X}$ is K\"othe reflexive, taking a supremum over all $g\in\mb{X}'$ with $\|g\|_{\mb{X}'}=1$ and letting $\varepsilon\downarrow 0$ proves the result.
\end{proof}

\section{Proof of Theorem~\ref{thm:B}}\label{sec:thmB}

First, we prove the implication \ref{it:thmB2}$\Rightarrow$\ref{it:thmB1}. This follows from the following result:
\begin{theorem}\label{thm:sparsewithmax}
Let $\mb{X}$ be a K\"othe reflexive $\F^n$-directional Banach function space over $\R^d$. Suppose
\[
M^{\mc{K}}:\mb{X}[\mc{K}]\to\mb{X}[\mc{K}],\quad M^{\mc{K}}:\mb{X}'[\mc{K}]\to\mb{X}'[\mc{K}].
\]
Then for any sparse collection $\mc{S}$, we have $T_{\mc{S}}:\mb{X}[\mc{K}]\to\mb{X}[\mc{K}]$ with
\[
\|T_{\mc{S}}\|_{\mb{X}[\mc{K}]\to\mb{X}[\mc{K}]}\lesssim_n \|M^{\mc{K}}\|_{\mb{X}[\mc{K}]\to\mb{X}[\mc{K}]}\|M^{\mc{K}}\|_{\mb{X}'[\mc{K}]\to\mb{X}'[\mc{K}]}.
\]
\end{theorem}
\begin{proof}
By the monotone convergence property of $\mb{X}$, we may assume that $\mc{S}$ is finite. Let $g\in S^0(\R^d;T_{\mc{S}}F)$. First, assume that $g=\sum_{Q\in\mc{S}}g_Q$ with $g_Q\in S^0(\R^d;\langle F\rangle_Q\ind_Q)$. As in the proof of Theorem~\ref{thm:pairwisedisjointavconvex}, by the John ellipsoid theorem, each $g_Q$ is of the form
\[
g_Q(x)=\sum_{k=1}^n h^Q_k(x)\langle f^Q_k\rangle_Q\ind_Q(x),
\]
where $h^Q=(h^Q_1,\ldots,h^Q_n)$ satisfies $|h^Q(x)|\leq n^{\frac{1}{2}}$, and $f^Q_k\in S^0(\R^d;F\ind_Q)$. Let $\psi\in\mb{X}'$, and set 
\[
\widetilde{h}_k^Q=\frac{\langle f_k^Q\rangle_Q\cdot\langle h_k^Q\psi\rangle_Q}{|\langle f_k^Q\rangle_Q\cdot\langle h_k^Q\psi\rangle_Q|}h_k^Q,
\]
or $\widetilde{h}_k^Q=0$ if the denominator in the above expression is equal to $0$, so that
\[
\langle f_k^Q\rangle_Q\cdot \langle \widetilde{h}_k^Q\psi\rangle_Q=|\langle f_k^Q\rangle_Q\cdot \langle h_k^Q\psi\rangle_Q|.
\]
We note that $\langle \widetilde{h}_k^Q\psi\rangle_Q\in C_n\llangle \psi\rrangle_Q$ for some $C_n>0$ depending on $n$, where $\llangle\psi\rrangle_Q$ is well-defined since $M^{\mc{K}}$ is bounded on $\mb{X}'[\mc{K}]$ and, hence, $\psi\in\mb{X}'$ is necessarily locally integrable. Then, defining the sets $E_Q$ for $Q\in\mc{S}$ in the same way as is done in the proof of Theorem~\ref{thm:auxiliarygoldbergreduction},
\begin{align*}
\Big|\int_{\R^d}\!g\cdot \psi\,\mathrm{d}x\Big|
&\leq\sum_{k=1}^n\Big|\sum_{Q\in\mc{S}}\int_Q\!h_k^Q\langle f_k^Q\rangle_Q\cdot \psi\,\mathrm{d}x\Big|\\
&=\sum_{k=1}^n\Big|\sum_{Q\in\mc{S}}\langle f_k^Q\rangle_Q\cdot \langle h_k^Q\psi\rangle_Q|Q|\Big|\\
&\leq 2\sum_{k=1}^n\int_{\R^d}\sum_{Q\in\mc{S}}\langle f_k^Q\rangle_Q\ind_{E_Q}\cdot \sum_{Q\in\mc{S}}\langle \widetilde{h}_k^Q\psi\rangle_Q\ind_{E_Q}\,\mathrm{d}x\\
&\leq 2\sum_{k=1}^n\Big\|\sum_{Q\in\mc{S}}\langle f_k^Q\rangle_Q\ind_{E_Q}\Big\|_{\mb{X}}\Big\|\sum_{Q\in\mc{S}}\langle \widetilde{h}_k^Q\psi\rangle_Q\ind_{E_Q}\Big\|_{\mb{X}'}\\
&\lesssim_n\|M^{\mc{K}}F\|_{\mb{X}[\mc{K}]}\|M^{\mc{K}}(\mc{K}(\psi))\|_{\mb{X}[\mc{K}]}\\
&\leq \|M^{\mc{K}}\|_{\mb{X}[\mc{K}]\to\mb{X}[\mc{K}]}\|M^{\mc{K}}\|_{\mb{X}'[\mc{K}]\to\mb{X}'[\mc{K}]}\|F\|_{\mb{X}[\mc{K}]}\|\psi\|_{\mb{X}'}.
\end{align*}
For a general $g\in S^0(\R^d;T_{\mc{S}}F)$, by Proposition~\ref{prop:sumofconvexmeasurable}, for each $Q\in\mc{S}$ there is a sequence $(g_{k,Q})_{k\geq 1}$ with $g_{k,Q}\in S^0(\R^d:\langle F\rangle_Q\ind_Q)$ for which 
\[
\sum_{Q\in\mc{S}}g_{k,Q}(x)\to g(x)
\]
for a.e. $x\in\R^d$. As $\mb{X}$ is K\"othe reflexive, it satisfies the Fatou property. Thus,
\[
\|g\|_{\mb{X}}\leq\liminf_{k\to\infty}\Big\|\sum_{Q\in\mc{S}}g_{k,Q}\Big\|_{\mb{X}}\lesssim_n\|M^{\mc{K}}\|_{\mb{X}[\mc{K}]\to\mb{X}[\mc{K}]}\|M^{\mc{K}}\|_{\mb{X}'[\mc{K}]\to\mb{X}'[\mc{K}]}\|F\|_{\mb{X}[\mc{K}]}.
\]
We conclude that $T_{\mc{S}}F\in\mb{X}[\mc{K}]$, with
\[
\|T_{\mc{S}}F\|_{\mb{X}[\mc{K}]}\lesssim_n\|M^{\mc{K}}\|_{\mb{X}[\mc{K}]\to\mb{X}[\mc{K}]}\|M^{\mc{K}}\|_{\mb{X}'[\mc{K}]\to\mb{X}'[\mc{K}]}\|F\|_{\mb{X}[\mc{K}]}.
\]
The assertion follows.
\end{proof}

\begin{proof}[Proof of Theorem~\ref{thm:B}]
For the implication \ref{it:thmB2}$\Rightarrow$\ref{it:thmB1}, the bound $T_{\mc{S}}:\mb{X}[\mc{K}]\to\mb{X}[\mc{K}]$ follows from Theorem~\ref{thm:sparsewithmax}. Likewise, as $\mb{X}$ is K\"othe reflexive, the dual bound $T_{\mc{S}}:\mb{X}'[\mc{K}]\to\mb{X}'[\mc{K}]$ follows from Theorem~\ref{thm:sparsewithmax} applied with $\mb{X}'$ instead of $\mb{X}$.

It remains to prove \ref{it:thmB1}$\Rightarrow$\ref{it:thmB2}. By symmetry, we need only prove that $M^{\mc{K}}:\mb{X}[\mc{K}]\to\mb{X}[\mc{K}]$. By the Fatou property and the $3^d$-lattice theorem (see \cite{LN15}), it suffices to prove that $M^{\mc{K}}_{\mc{F}}:\mb{X}[\mc{K}]\to\mb{X}[\mc{K}]$ for all finite collections $\mc{F}$ contained in a dyadic grid $\mc{D}$. Let $F\in\mb{X}[\mc{K}]$. By Theorem~\ref{thm:convexbodydomofmax}, there is a sparse collection $\mc{S}\subseteq\mc{F}$ such that
\[
M^{\mc{K}}_{\mc{F}}F(x)\subseteq C_n M^{\mc{K}}_{\mc{S}}F(x)\subseteq C_nT_{\mc{S}}(x),
\]
where the last inclusion follows from the fact that $\bigcup_{Q\in\mc{S}}\langle F\rangle_Q\ind_Q(x)\subseteq \sum_{Q\in\mc{S}}\langle F\rangle_Q\ind_Q(x)$. Thus, by the directional ideal property of $\mb{X}$, we have
\[
\|M^{\mc{K}}_{\mc{F}}F\|_{\mb{X}[\mc{K}]}\lesssim_n\|T_{\mc{S}}F\|_{\mb{X}[\mc{K}]}\leq\|T_{\mc{S}}\|_{\mb{X}[\mc{K}]\to\mb{X}[\mc{K}]}\|F\|_{\mb{X}[\mc{K}]},
\]
proving the desired bound.
\end{proof}

\section{Applications}\label{sec:applications}
\subsection{Matrix-weighted variable Lebesgue spaces}\label{subsec:varlebesgue}
For $1\leq p\leq\infty$ we define
\[
\phi_p(t):=\begin{cases}
\tfrac{1}{p}t^p & \text{if $p<\infty$}\\
\infty\ind_{(1,\infty)}(t) & \text{if $p=\infty$}
\end{cases}
\]
and, for a measurable function $p:\R^d\to[1,\infty]$ and $f\in L^0(\R^d)$,
\[
\rho_{p(\cdot)}(f):=\int_{\R^d}\!\phi_{p(x)}(|f(x)|)\,\mathrm{d}x.
\]
The variable Lebesgue space $L^{p(\cdot)}(\R^d)$ is defined as those $f\in L^0(\R^d)$ for which there is a $\lambda>0$ such that $\rho_{p(\cdot)}(f/\lambda)<\infty$, with norm
\[
\|f\|_{L^{p(\cdot)}(\R^d)}:=\inf\{\lambda>0:\rho_{p(\cdot)}(f/\lambda)\leq 1\}.
\]
This is a Banach function space over $\R^d$ with K\"othe dual space $L^{p'(\cdot)}(\R^d)$, where
\[
\frac{1}{p'(x)}:=1-\frac{1}{p(x)}.
\]
We refer the reader to \cite{DHHR11} for a full overview of the properties of these spaces.

Given a matrix weight $W:\R^d\to\F^{n\times n}$, we define the matrix weighted variable Lebesgue space $L^{p(\cdot)}_W(\R^d;\F^n)$ as the space $X_W$ for $X=L^{p(\cdot)}(\R^d)$, i.e., the space of those $f\in L^0(\R^d;\F^n)$ for which $|Wf|\in L^{p(\cdot)}(\R^d)$. This definition coincides with the one given in \cite{CP24}.

By combining Theorem~\ref{thm:BFSextrapolation} with Proposition~\ref{prop:christgoldbergmax}, we obtain the following extrapolation theorem:
\begin{theorem}\label{thm:varlebesgueextrapolation}
Let $1\leq r\leq \infty$, let $V$ be a set, and let $S:V\to L^0(\R^d;\F^n)$. Suppose
\[
T:\bigcup_{W\in A_r}S^{-1}(L^r_W(\R^d;\F^n))\to L^0(\R^d;\F^n)
\]
is a map for which there is an increasing function $\phi:[0,\infty)\to[0,\infty)$ such that for all $W\in A_r$ and all $f\in V$ with $Sf\in L^r_W(\R^d;\F^n)$ we have
\[
\|Tf\|_{L^r_W(\R^d;\F^n)}\leq\phi([W]_r)\|Sf\|_{L^r_W(\R^d;\F^n)}.
\]
Let $p:\R^d\to[1,\infty]$ be a measurable function, and let $W:\R^d\to\F^{n\times n}$ be a matrix weight for which
\[
M_W:L^{p(\cdot)}(\R^d;\F^n)\to L^{p(\cdot)}(\R^d),\quad M_{W^{-1}}:L^{p'(\cdot)}(\R^d;\F^n)\to L^{p'(\cdot)}(\R^d).
\]
Then $Tf$ is well-defined for all $f\in V$ with $Sf\in L^{p(\cdot)}_W(\R^d;\F^n)$, and
\begin{align*}
&\|Tf\|_{L^{p(\cdot)}_W(\R^d;\F^n)}\\
&\lesssim_n \phi(C_n\|M_W\|_{L^{p(\cdot)}(\R^d;\F^n)\to L^{p(\cdot)}(\R^d)}^{\frac{1}{r'}}\|M_{W^{-1}}\|_{L^{p'(\cdot)}(\R^d;\F^n)\to L^{p'(\cdot)}(\R^d)}^{\frac{1}{r}})\|Sf\|_{L^{p(\cdot)}_W(\R^d;\F^n)}.
\end{align*}
If $r=\infty$ or $r=1$, we can omit the bound $M_{W^{-1}}:L^{p'(\cdot)}(\R^d;\F^n)\to L^{p'(\cdot)}(\R^d)$ or $M_W:L^{p(\cdot)}(\R^d;\F^n)\to L^{p(\cdot)}(\R^d)$ respectively.
\end{theorem}
When $n=1$, our result recovers the variable Lebesgue space extrapolation theorem of \cite[Theorem~2.7]{CW17}. For a large class of examples of weights $W$ and exponent functions $p(\cdot)$ for which the Christ-Goldberg maximal operators $M_W$ and $M_{W^{-1}}$ satisfy the bounds in the above result we refer the reader to \cite{NP25}.

We say that an exponent function $p:\R^d\to[1,\infty]$ is globally $\log$-H\"older continuous if there are constants $C_1,C_2>0$ and $p_\infty\in[1,\infty]$ for which
\[
\Big|\frac{1}{p(x)}-\frac{1}{p(y)}\Big|\leq\frac{C_1}{\log(e+\frac{1}{|x-y|})}\quad\text{and}\quad \Big|\frac{1}{p(x)}-\frac{1}{p_\infty}\Big|\leq\frac{C_2}{\log(e+|x|)}
\]
for all $x,y\in\R^d$. Moreover, we write $W\in A_{p(\cdot)}$ to mean that $L^{p(\cdot)}_W(\R^d;\F^n)\in A$, and set
\[
[W]_{p(\cdot)}:=[L^{p(\cdot)}_W(\R^d;\F^n)]_A.
\]
Since $L^{p(\cdot)}(\R^d)\in\mc{G}$, see \cite[Theorem~7.3.22]{DHHR11}, Theorem~\ref{thm:E} directly implies the following result:
\begin{theorem}\label{thm:varlebesguemuckenhoupt}
Let $p:\R^d\to[1,\infty]$ be a globally $\log$-H\"older continuous exponent function and let $W:\R^d\to\F^{n\times n}$ be a matrix weight. Then the following are equivalent:
\begin{enumerate}[(i)]
    \item $M:L^{p(\cdot)}_W(\R^d;\mc{K})\to L^{p(\cdot)}_W(\R^d;\mc{K})_{\text{weak}}$;
    \item $W\in A_{p(\cdot)}$;
    \item $L^{p(\cdot)}_W(\R^d;\F^n)\in A_{\text{strong}}$.
\end{enumerate}
Moreover, in this case all the respective constants are equivalent up constants depending on $n$, $d$, and the global $\log$-H\"older regularity constants.
\end{theorem}

\subsection{Matrix-weighted Morrey spaces}
Let $1\leq p\leq q\leq\infty$. Then the Morrey space $M^{p,q}(\R^d)$ is defined as those $f\in L^0(\R^d)$ for which
\[
\|f\|_{M^{p,q}(\R^d)}:=\sup_Q\Big(\frac{1}{|Q|^{1-\frac{p}{q}}}\int_Q\!|f|^p\,\mathrm{d}x\Big)^{\frac{1}{p}}<\infty.
\]
This is a Banach function space with the Fatou property. To define the K\"othe dual of this space, we say that a function $b\in L^0(\R^d)$ is a $(q,p)$-block if it is supported in a cube $Q$ in $\R^d$, and
\[
\Big(|Q|^{\frac{q}{p}-1}\int_Q\!|b|^q\,\mathrm{d}x\Big)^{\frac{1}{q}}\leq 1.
\]
The space $B^{q,p}(\R^d)$ is defined as those $f\in L^0(\R^d)$ for which there is a sequence $(\lambda_k)_{k\geq 1}\in\ell^1$ and a sequence of $(q,p)$-blocks $(b_k)_{k\geq 1}$ such that
\[
f=\sum_{k=1}^\infty\lambda_k b_k.
\]
The norm $\|f\|_{B^{q,p}(\R^d)}$ is defined as the smallest possible value of $\|(\lambda_k)_{k\geq 1}\|_{\ell^1}$ for which $f$ can be represented in this way. Then we have
\[
M^{p,q}(\R^d)'=B^{p',q'}(\R^d).
\]
We refer the reader to \cite[Subsection~3.3]{Ni23} and references therein for an overview of these spaces.

Given a matrix weight $W:\R^d\to\F^{n\times n}$ we define $M^{p,q}_W(\R^d;\F^n)$ and $B^{q,p}_W(\R^d;\F^n)$ as the space $X_W$, respectively for $X=M^{p,q}(\R^d)$ and $X=B^{q,p}(\R^d)$. This definition is a matrix weighted extension of the Samko-type weighted Morrey space as considered in \cite{Sa09}. Matrix weighted Morrey spaces have not been considered for this type before, but a version of matrix weighted Morrey spaces of Komori and Shirai \cite{KS09} with the weight used as a measure has appeared in \cite{BS24}.

Theorem~\ref{thm:BFSextrapolation} combined with Proposition~\ref{prop:christgoldbergmax} yields the following extrapolation theorem:
\begin{theorem}\label{thm:morreyextrapolation}
Let $1\leq r\leq \infty$, let $V$ be a set, and let $S:V\to L^0(\R^d;\F^n)$. Suppose
\[
T:\bigcup_{W\in A_r}S^{-1}(L^r_W(\R^d;\F^n))\to L^0(\R^d;\F^n)
\]
is a map for which there is an increasing function $\phi:[0,\infty)\to[0,\infty)$ such that for all $W\in A_r$ and all $f\in V$ with $Sf\in L^r_W(\R^d;\F^n)$ we have
\[
\|Tf\|_{L^r_W(\R^d;\F^n)}\leq\phi([W]_r)\|Sf\|_{L^r_W(\R^d;\F^n)}.
\]
Let $1\leq p\leq q\leq\infty$, and let $W:\R^d\to\F^{n\times n}$ be a matrix weight for which
\[
M_W:M^{p,q}(\R^d;\F^n)\to M^{p,q}(\R^d),\quad M_{W^{-1}}:B^{p',q'}(\R^d;\F^n)\to B^{p',q'}(\R^d).
\]
Then $Tf$ is well-defined for all $f\in V$ with $Sf\in M^{p,q}_W(\R^d;\F^n)$, and
\begin{align*}
&\|Tf\|_{M^{p,q}_W(\R^d;\F^n)}\\
&\lesssim_n \phi(C_n\|M_W\|_{M^{p,q}(\R^d;\F^n)\to M^{p,q}(\R^d)}^{\frac{1}{r'}}\|M_{W^{-1}}\|_{B^{p',q'}(\R^d;\F^n)\to B^{p',q'}(\R^d)}^{\frac{1}{r}})\|Sf\|_{M^{p,q}_W(\R^d;\F^n)}.
\end{align*}
If $r=\infty$ or $r=1$, we can omit the bound $M_{W^{-1}}:B^{p',q'}(\R^d;\F^n)\to B^{p',q'}(\R^d)$ or $M_W:M^{p,q}(\R^d;\F^n)\to M^{p,q}(\R^d)$ respectively.
\end{theorem}
Several extrapolation theorems for Morrey spaces in the case $n=1$ exist in the literature, see, e.g., the works \cite{DR20, DR21} and references therein. Even in the case $n=1$, a full characterization of the weights for which $M$ is bounded on $M^{p,q}_w(\R^d)$ and $B^{p',q'}_{w^{-1}}(\R^d)$ has not yet been found. An overview and partial results in this problem can be found in \cite[Subsection~5.3]{Ni24}. For general $n$, this leaves open the problem of finding matrix weights $W$ for which the Goldberg maximal operator is bounded on the respective spaces in the above theorem.
 
\subsection{Component-wise spaces}
Here we show that for component-wise directional Banach function spaces over $\R^d$, the developed theory is equivalent to applying the scalar-valued theory to each component. Our main result here is the following:
\begin{theorem}\label{thm:componentwisemaxbound}
Let $X_1,\ldots,X_n$ be Banach function spaces over $\R^d$ with the Fatou property, and let $(v_1,\ldots,v_n)$ be an orthonormal basis of $\F^n$. Let $\mb{X}$ be defined as those $f\in L^0(\R^d;\F^n)$ for which
\[
\|f\|_{\mb{X}}:=\sum_{k=1}^n\|f\cdot v_k\|_{X_k}<\infty.
\]
Then the following are equivalent:
\begin{enumerate}[(i)]
    \item\label{it:componentwisemaxbound1} $M^{\mc{K}}:\mb{X}[\mc{K}]\to\mb{X}[\mc{K}]$;
    \item\label{it:componentwisemaxbound2} $M:X_k\to X_k$ for all $k\in\{1,\ldots,n\}$.
\end{enumerate}
Moreover, in this case we have
\[
\max_{k\in\{1,\ldots,n\}}\|M\|_{X_k\to X_k}\leq\|M^{\mc{K}}\|_{\mb{X}[\mc{K}]\to\mb{X}[\mc{K}]}\leq \sum_{k=1}^n\|M\|_{X_k\to X_k}.
\]
\end{theorem}
\begin{proof}
For \ref{it:componentwisemaxbound1}$\Rightarrow$\ref{it:componentwisemaxbound2}, fix $k\in\{1,\ldots,n\}$ and let $h_k\in X_k$. Setting $f:=h_kv_k\in\mb{X}$, we claim that $(Mh_k)v_k\in S^0(\R^d;M^{\mc{K}}(\mc{K}(f))$. Assuming the claim, using Proposition~\ref{prop:convexsetfequalsf}, we have $(Mh_k)v_k\in\mb{X}$, with
\begin{align*}
\|Mh_k\|_{X_k}&=\|(Mh_k)v_k\|_{\mb{X}}\leq\|M^{\mc{K}}(\mc{K}(f))\|_{\mb{X}[\mc{K}]}\\
&\leq\|M^{\mc{K}}\|_{\mb{X}[\mc{K}]\to\mb{X}[\mc{K}]}\|f\|_{\mb{X}}=\|M^{\mc{K}}\|_{\mb{X}[\mc{K}]\to\mb{X}[\mc{K}]}\|h_k\|_{X_k}.
\end{align*}
Thus, $M:X_k\to X_k$ with $\|M\|_{X_k\to X_k}\leq\|M^{\mc{K}}\|_{\mb{X}[\mc{K}]\to\mb{X}[\mc{K}]}$, as desired. To prove the claim, note that
\[
\llangle f\rrangle_Q=\{\langle hh_k\rangle_Qv_k:\|h\|_{L^\infty(\R^d)}\leq 1\}=\{\lambda \langle |h_k|\rangle_Qv_k:|\lambda|\leq 1\}=\langle |h_k|\rangle_Q\mc{K}(v_k),
\]
where $\llangle f\rrangle_Q$ is well-defined due to the fact that $M^{\mc{K}}$ is bounded on $\mb{X}[\mc{K}]$, where the second equality follows from setting $h=\lambda\cdot \text{sign}(h_k)$, and the last one from Proposition~\ref{prop:smallestkset}. This implies that
\begin{align*}
Mh_k(x)v_k&\in\overline{\big\{\langle|h_k|\rangle_Qv_k:Q\text{ a cube with $x\in Q$}\big\}}\\
&\subseteq\overline{\bigcup_Q\llangle f\rrangle_Q\ind_Q(x)}\subseteq M^{\mc{K}}(\mc{K}(f))(x)
\end{align*}
for a.e. $x\in\R^d$, proving the result.

For the converse implication \ref{it:componentwisemaxbound2}$\Rightarrow$\ref{it:componentwisemaxbound1}, let $F\in\mb{X}[\mc{K}]$. As each of the $X_k$ has the Fatou property, $\mb{X}$ has the monotone convergence property. Thus, it suffices to prove the bound of $M^{\mc{K}}_{\mc{F}}$ for finite collections of cubes $\mc{F}$ uniformly in $\mc{F}$. If $g\in S^0(\R^d;M^{\mc{K}}_{\mc{F}}F)$, then for a.e. $x\in\R^d$ there are $\theta_1,\ldots,\theta_J\geq 0$ with $\sum_{j=1}^J\theta_j=1$, cubes $Q_1,\ldots,Q_J\in\mc{F}$ containing $x$, and $f_1,\ldots,f_J\in S^0(\R^d;F)$ such that
\[
g(x)=\sum_{j=1}^J \theta_j\langle f_j\rangle_{Q_j}.
\]
Thus, we have
\[
|g(x)\cdot v_k|\leq \sum_{j=1}^J \theta_jM(|f_j\cdot v_k|)(x)\leq\sum_{j=1}^J \theta_j M(|F\cdot v_k|)(x)=M(|F\cdot v_k|)(x).
\]
Thus, by the ideal property of the $X_k$, we have $g\in\mb{X}$ with
\begin{align*}
\|g\|_{\mb{X}}&\leq\sum_{k=1}^n\|M(|F\cdot v_k|)\|_{X_k} 
\leq\sum_{k=1}^n\|M\|_{X_k\to X_k}\||F\cdot v_k|\|_{X_k}\\
&\leq\sum_{k=1}^n\|M\|_{X_k\to X_k}\|F\|_{X[\mc{K}]}.
\end{align*}
For this last inequality, we used the fact that by the Filippov selection theorem there is an $f\in S^0(\R^d;F)$ for which
\[
|F(x)\cdot v_k|=|f(x)\cdot v_k|
\]
for a.e. $x\in\R^d$, and, hence, 
\[
\||F\cdot v_k|\|_{X_k}=\|f\cdot v_k\|_{X_k}\leq\|f\|_{\mb{X}}\leq\|F\|_{X[\mc{K}]}.
\]
Thus, taking a supremum over all $g\in S^0(\R^d;M^{\mc{K}}_{\mc{F}}F)$, we conclude that $M_{\mc{F}}^{\mc{K}}:\mb{X}[\mc{K}]\to\mb{X}[\mc{K}]$ with
\[
\|M_{\mc{F}}^{\mc{K}}\|_{\mb{X}[\mc{K}]\to\mb{X}[\mc{K}]}\leq\sum_{k=1}^n\|M\|_{X_k\to X_k}.
\]
The result follows.
\end{proof}

Combining our result with Theorem~\ref{thm:A}, we obtain the following corollary, which we only state for $p=2$:

\begin{corollary}\label{cor:componentwiseextrapolation}
Suppose
\[
T:\bigcup_{W\in A_2}L^2_W(\R^d;\F^n)\to L^0(\R^d;\F^n)
\]
is a map for which there is an increasing function $\phi:[0,\infty)\to[0,\infty)$ such that for all $W\in A_2$ and all $f\in L^2_W(\R^d;\F^n)$ we have
\[
\|Tf\|_{L^2_W(\R^d;\F^n)}\leq\phi([W]_2)\|f\|_{L^2_W(\R^d;\F^n)}.
\]
Let $X_1,\ldots,X_n$ be Banach function spaces over $\R^d$ with the Fatou property, and let $(v_1,\ldots,v_n)$ be an orthonormal basis of $\F^n$. Let $\mb{X}$ be defined as those $f\in L^0(\R^d;\F^n)$ for which
\[
\|f\|_{\mb{X}}:=\sum_{k=1}^n\|f\cdot v_k\|_{X_k}<\infty.
\]
Suppose that
\[
M:X_k\to X_k,\quad M:X_k'\to X_k'
\]
for all $k\in\{1,\ldots,n\}$.
Then $Tf$ is well-defined for all $f\in\mb{X}$, and
\[
\|Tf\|_{\mb{X}}\lesssim_n \phi(C_n\max_{k\in\{1,\ldots,n\}}\|M\|_{X_k\to X_k}^{\frac{1}{2}}\max_{k\in\{1,\ldots,n\}}\|M\|_{X_k'\to X_k'}^{\frac{1}{2}})\|f\|_{\mb{X}}.
\]
\end{corollary}
\begin{proof}
Since
\[
\|g\|_{\mb{X}'}=\max_{k\in\{1,\ldots,n\}}\|g\cdot v_k\|_{X_k'}\eqsim_n\sum_{k=1}^n\|g\cdot v_k\|_{X_k'}
\]
(see Section~\ref{sec:componentwise}), it follows from Theorem~\ref{thm:componentwisemaxbound} that
\[
M^{\mc{K}}:\mb{X}[\mc{K}]\to \mb{X}[\mc{K}],\quad M^{\mc{K}}:\mb{X}'[\mc{K}]\to \mb{X}'[\mc{K}].
\]
Thus, the result follows from Theorem~\ref{thm:A} if we can show that $\mb{X}$ is K\"othe reflexive. This is indeed the case, since $\mb{X}''$ is the component-wise $\F^n$-directional Banach function space constructed from the spaces $X_1'',\ldots,X_n''$ with respect to the basis $(v_1,\ldots,v_n)$ (see Section~\ref{sec:componentwise}). Since $X_1,\ldots,X_n$ have the Fatou property, the Lorentz-Luxemburg theorem implies that $X_k''=X_k$ for all $k\in\{1,\ldots,n\}$. Thus, $\mb{X}''=\mb{X}$, as desired.
\end{proof}
An overview of Banach function spaces $X$ for which $M:X\to X$ and $M:X'\to X'$, including for weighted variable Lebesgue and weighted Morrey spaces, is given in \cite{Ni23, Ni24}. Thus, using Corollary~\ref{cor:componentwiseextrapolation}, these bounds can be directly applied to this component-wise setting.

\appendix

\section{Basic properties of \texorpdfstring{$\F^n$}{Fn}-directional quasi-Banach function spaces}\label{app:A}

\begin{proposition}\label{prop:saturation}
Let $\mb{X}$ be a complete quasi-normed subspace of $L^0(\Omega;\F^n)$ satisfying the directional ideal property. Then the following are equivalent:
\begin{enumerate}[(i)]
    \item\label{it:sat1} $\mb{X}$ satisfies the non-degeneracy property;
    \item\label{it:sat2} $\|\cdot\|_{\mb{X}'}$ is a norm;
    \item\label{it:sat2.5} $\mb{X}\cap L^2(\Omega;\F^n)$ is dense in $L^2(\Omega;\F^n)$;
    \item\label{it:sat3} For all $g\in L^0(\Omega;\F^n)$ there is a sequence $(f_k)_{k\geq 1}$ in $\mb{X}$ for which $f_k\to g$ a.e.
\end{enumerate}
\end{proposition}
\begin{proof}
The equivalence \ref{it:sat1}$\Leftrightarrow$\ref{it:sat2} follows from Proposition~\ref{prop:normoutsideintegral}, as it implies that $\int_\Omega\!f\cdot g\,\mathrm{d}\mu=0$ for all $f\in\mb{X}$ if and only if $\int_\Omega\!|f\cdot g|\,\mathrm{d}\mu=0$ for all $f\in\mb{X}$.

To see \ref{it:sat1}$\Rightarrow$\ref{it:sat2.5}, we note that the statement is equivalent to showing that 
\[
(\mb{X}\cap L^2(\Omega;\F^n))^\perp=\{0\},
\]
where
\[
(\mb{X}\cap L^2(\Omega;\F^n))^\perp=\Big\{g\in L^2(\Omega;\F^n):\int_\Omega\!f\cdot g\,\mathrm{d}\mu=0\text{ for all $f\in\mb{X}\cap L^2(\Omega;\F^n)$}\Big\}.
\]
Let $f\in\mb{X}$, let $(E_k)_{k\geq 1}$ be a sequence of sets for which $\mu(E_k)<\infty$ for all $k\geq 1$ and $\bigcup_{k=1}^\infty E_k=\Omega$. Setting
\[
f_k:=\ind_{\{x\in E_k:|f(x)|\leq k\}}f,
\]
we have $f_k\in\mb{X}\cap L^2(\Omega;\F^n)$, and, by the monotone convergence theorem, for any $g\in L^2(\Omega;\F^n)$ we have
\[
\Big|\int_\Omega\!f\cdot g\,\mathrm{d}\mu\Big|\leq\int_\Omega\!|f\cdot g|\,\mathrm{d}\mu=\sup_{k\geq 1}\int_\Omega\!|f_k\cdot g|\,\mathrm{d}\mu=0.
\]
Hence, $g=0$ a.e. by the non-degeneracy property of $\mb{X}$, as desired. 

For \ref{it:sat2.5}$\Rightarrow$\ref{it:sat3}, let $Y$ denote the closure of $\mb{X}\cap L^2(\Omega;\F^n)$ in $L^0(\Omega;\F^n)$ with respect to a.e. pointwise convergence. Since $L^2(\Omega;\F^n)$ convergence implies a.e. pointwise convergence of a subsequence, it follows from \ref{it:sat2.5} that
\[
L^2(\Omega;\F^n)\subseteq Y.
\]
For any $g\in L^0(\Omega;\F^n)$, the sequence
\[
f_k:=\ind_{\{x\in E_k:|g(x)|\leq k\}}g
\]
satisfies $f_k\in L^2(\Omega;\F^n)$, and $f_k\to g$ a.e. We conclude that
\[
Y=L^0(\Omega;\F^n),
\]
as desired.

Finally, for \ref{it:sat3}$\Rightarrow$\ref{it:sat2}, suppose that $\|g\|_{\mb{X}'}=0$, and let $(f_k)_{k\geq 1}$ in $\mb{X}$ be such that $f_k\to g$ a.e. By Fatou's lemma, we have
\[
\int_\Omega\!|g|^2\,\mathrm{d}\mu\leq\liminf_{k\to\infty}\int_\Omega\!|f_k\cdot g|\,\mathrm{d}\mu=0.
\]
Hence, $g=0$ a.e., as asserted.
\end{proof}

Next, we prove the $\F^n$-directional analogue of the Lorentz-Luxemburg theorem. Recall that K\"othe reflexivity means that $\mb{X}''=\mb{X}$.
\begin{theorem}\label{thm:lorentzluxemburgapp}
Let $\mb{X}$ be an $\F^n$-directional Banach function space over $\Omega$. Then consider the statements:
\begin{enumerate}[(a)]
    \item\label{it:applorentzluxemburga} $\mb{X}$ satisfies the Fatou property;
    \item\label{it:applorentzluxemburgb} $\mb{X}$ is K\"othe reflexive.
\end{enumerate}
Then \ref{it:applorentzluxemburgb}$\Rightarrow$\ref{it:applorentzluxemburga}. If $\mb{X}$ satisfies the directional saturation property, then also \ref{it:applorentzluxemburga}$\Rightarrow$\ref{it:applorentzluxemburgb}.
\end{theorem}

To prove this, we use the following result:
\begin{lemma}\label{lem:lorentzluxemburg}
Let $\mb{X}$ be an $\F^n$-directional Banach function space over $\Omega$. If $\mb{X}$ satisfies the Fatou property, then 
\[
\|f\|_{\mb{X}}=\sup_{\|g\|_{\mb{X}'}=1}\int_\Omega\!|f\cdot g|\,\mathrm{d}\mu
\]
for all $f\in\mb{X}$.
\end{lemma}
We remark here that this result does not hold for an $\F^n$-directional \emph{quasi}-Banach function space over $\Omega$, as this norm equality implies that the triangle inequality holds.
\begin{proof}[Proof of Lemma~\ref{lem:lorentzluxemburg}]
We define the seminorm
\[
\|f\|_{\mb{X}''}:=\sup_{\|g\|_{\mb{X}'}=1}\int_\Omega\!|f\cdot g|\,\mathrm{d}\mu,
\]
and will show that $\|f\|_{\mb{X}''}=\|f\|_{\mb{X}}$ for all $f\in\mb{X}$ if $\mb{X}$ has the Fatou property. Since $\mb{X}\subseteq\mb{X}''$ with $\|\cdot\|_{\mb{X}''}\leq\|\cdot\|_{\mb{X}}$, it remains to prove the converse inequality.

Let $B_{\mb{X}}:=\{f\in\mb{X}:\|f\|_{\mb{X}}\leq 1\}$ and let
\[
Y:=B_{\mb{X}}\cap L^2(\Omega;\F^n).
\]
Then $Y$ is a convex set. We claim that $Y$ is a closed subset of $L^2(\Omega;\F^n)$. Indeed, if $(f_k)_{k\geq 1}$ is a sequence in $Y$ that converges to a function $f\in L^2(\Omega;\F^n)$ in $L^2(\Omega;\F^n)$, then for some subsequence $f_{k_j}$ we have $f_{k_j}\to f$ a.e so that by the Fatou property of $\mb{X}$ we have $f\in\mb{X}$ with
\[
\|f\|_{\mb{X}}\leq\liminf_{j\to\infty}\|f_{k_j}\|_{\mb{X}}\leq 1.
\]
Hence, $f\in Y$, as desired.

Now, let $f\in \mb{X}\cap L^2(\Omega;\F^n)$ with $\|f\|_{\mb{X}}=1$, let $\varepsilon>0$, and set $f_0:=(1+\varepsilon)f$. Then $f_0\in L^2(\Omega;\F^n)$, but, since $\|f_0\|_{\mb{X}}=1+\varepsilon>1$, $f_0\notin Y$. Since $Y$ is convex, it follows from the Hahn-Banach separation theorem, the Riesz representation theorem, and Proposition~\ref{prop:normoutsideintegral}, that there is a $g\in L^2(\Omega;\F^n)$ such that
\[
\int_\Omega\!|f_0\cdot g|\,\mathrm{d}\mu> 1,\quad \int_\Omega\!|h\cdot g|\,\mathrm{d}\mu<1\quad\text{for all $h\in Y$.}
\]
We claim that $g\in\mb{X}'$ with $\|g\|_{\mb{X}'}\leq 1$. Indeed, let $h\in\mb{X}$ with $\|h\|_{\mb{X}}=1$. Since $\Omega$ is $\sigma$-finite we can pick an increasing sequence of sets $(E_k)_{k\geq 1}$ with $\bigcup_{k=1}^\infty E_k=\Omega$ and $\mu(E_k)<\infty$ and set
\[
h_k:=h\ind_{\{x\in E_k:|h(x)|\leq k\}}\in Y.
\]
Then
\[
\int_\Omega\!|h_k\cdot g|\,\mathrm{d}\mu<1.
\]
Since $|h_k\cdot g|\uparrow |h\cdot g|$ a.e., it follows from the monotone convergence theorem that
\[
\int_\Omega\!|h\cdot g|\,\mathrm{d}\mu=\sup_{k\geq 1}\int_\Omega\!|h_k\cdot g|\,\mathrm{d}\mu\leq 1.
\]
Thus indeed, $g\in\mb{X}'$ with $\|g\|_{\mb{X}'}\leq 1$. Noting that, by Proposition~\ref{prop:normoutsideintegral},
\[
1<(1+\varepsilon)\int_\Omega\!|f\cdot g|\,\mathrm{d}\mu\leq(1+\varepsilon)\|f\|_{\mb{X}''}\|g\|_{\mb{X'}}\leq(1+\varepsilon)\|f\|_{\mb{X}''},
\]
letting $\varepsilon\downarrow 0$, we conclude that 
\[
\|f\|_{\mb{X}''}\geq 1=\|f\|_{\mb{X}},
\]
and, thus, $\|f\|_{\mb{X}''}=\|f\|_{\mb{X}}$ for all $f\in \mb{X}\cap L^2(\Omega;\F^n)$. 

Finally, let $f\in\mb{X}$ and set
\[
f_k:=\ind_{\{x\in E_k:|f(x)|\leq k\}} f
\]
so that $|f_k(x)\cdot u|\uparrow |f(x)\cdot u|$ for all $u\in\F^n$ for a.e. $x\in\Omega$. Then, by the monotone convergence property of $\mb{X}$ and $\mb{X}''$ (which they have because they satisfy the Fatou property, see Section~\ref{sec:directionalqbfs}) and the equivalence of \ref{it:ideal1} and \ref{it:ideal3} in Proposition~\ref{prop:ideal}, we have $f\in\mb{X}$ with
\[
\|f\|_{\mb{X}}=\sup_{k\geq 1}\|f_k\|_{\mb{X}}=\sup_{k\geq 1}\|f_k\|_{\mb{X}''}=\|f\|_{\mb{X}''}.
\]
The assertion follows.
\end{proof}

\begin{proof}[Proof of Theorem~\ref{thm:lorentzluxemburgapp}]
If $\mb{X}=\mb{X}''$, then $\mb{X}$ has the Fatou property by Fatou's lemma. Conversely, by Lemma~\ref{lem:lorentzluxemburg}, the Fatou property of $\mb{X}$ implies that $\|f\|_{\mb{X}}=\|f\|_{\mb{X}''}$ for all $f\in\mb{X}$. Let $f\in\mb{X}''$. By the directional saturation property and Proposition~\ref{prop:directionalsaturation}\ref{it:dirsat3} there is a sequence $(f_k)_{k\geq 1}$ in $\mb{X}$ for which $|f_k(x)\cdot u|\uparrow |f(x)\cdot u|$ for all $u\in\F^n$ for a.e. $x\in\Omega$. By the monotone convergence property of $\mb{X}$ and $\mb{X}''$ (see the above proof) and the equivalence of \ref{it:ideal1} and \ref{it:ideal3} in Proposition~\ref{prop:ideal}, we conclude that $f\in\mb{X}$, and
\[
\|f\|_{\mb{X}}=\sup_{k\geq 1}\|f_k\|_{\mb{X}}=\sup_{k\geq 1}\|f_k\|_{\mb{X}''}=\|f\|_{\mb{X}''}.
\]
The assertion follows.
\end{proof}

\section{An intrinsic proof of the sharp Rubio de Francia extrapolation theorem in matrix-weighted Lebesgue spaces}\label{app:lebesgueextrapolation}

In \cite{BC23}, a version of the following theorem was proven:
\begin{theorem}\label{thm:bc23extrapolation}
Let $1\leq r\leq \infty$, let $V$ be a set, and let $S:V\to L^0(\R^d;\F^n)$. Suppose
\[
T:\bigcup_{W\in A_r}S^{-1}(L^r_W(\R^d;\F^n))\to L^0(\R^d;\F^n)
\]
is a map for which there is an increasing function $\phi:[0,\infty)\to[0,\infty)$ such that for all $W\in A_r$ and all $f\in V$ with $Sf\in L^r_W(\R^d;\F^n)$ we have
\[
\|Tf\|_{L^r_W(\R^d;\F^n)}\leq\phi([W]_r)\|Sf\|_{L^r_W(\R^d;\F^n)}.
\]
Let $1<p<\infty$ and let $W\in A_p$. Then $Tf$ is well-defined for all $f\in V$ with $Sf\in L^p_W(\R^d;\F^n)$, and, for a constant $C_{n,d,r,p}>0$ depending only on $n,d,r,p$, we have
\[
\|Tf\|_{L^p_W(\R^d;\F^n)}
\lesssim_n \phi(C_{n,d,r,p}[W]^{\max\{\frac{p}{r},\frac{p'}{r'}\}}_p)\|Sf\|_{L^p_W(\R^d;\F^n)}.
\]
\end{theorem}
In this appendix we will give an alternative proof of this result through the intrinsic strategy used to prove Theorem~\ref{thm:A}. To this end, we need the following version of Theorem~\ref{thm:rdfmainestimate}:
\begin{theorem}\label{thm:bc3extrapolationrdf}
Let $1<p<\infty$, $1\leq r\leq\infty$, and $W\in A_p$. Then for all $f\in L^p_W(\R^d;\F^n)$, $g\in L^{p'}_{W^{-1}}(\R^d;\F^n)$, and $\varepsilon>0$, there exists a matrix weight $\mc{V}\in A_r$ satisfying
\begin{itemize}
    \item $[\mc{V}]_r\lesssim_{n,d,r,p}[W]^{\max\{\frac{p}{r},\frac{p'}{r'}\}}_p$;
    \item $\|f\|_{L^r_\mc{V}(\R^d;\F^n)}\|g\|_{L^{r'}_{\mc{V}^{-1}}(\R^d;\F^n)}\lesssim_n(\|f\|_{L^p_W(\R^d;\F^n)}+\varepsilon)(\|g\|_{L^{p'}_{W^{-1}}(\R^d;\F^N)}+\varepsilon)$.
\end{itemize}
\end{theorem}

For the proof, we need the following analogue of Lemma~\ref{lem:interpolationnorm}, which is a norm function version of \cite[Corollary~8.9]{BC23}:
\begin{lemma}\label{lem:interpolationnormbc3}
Let $F\in L^1_{\text{loc}}(\R^d;\mc{K})$ have a non-empty interior a.e., and suppose there is a $C>0$ such that for a.e. $x\in\R^d$ we have
\[
M^{\mc{K}}F(x)\subseteq C F(x).
\]
Then the norm function
\[
\sigma(x,u):=
\sup_{v\in F(x)}|u\cdot v|
\]
satisfies $\sigma\in A_1$ with $[\sigma]_1\lesssim_n C$.

Moreover, if $1\leq p\leq r\leq\infty$ and $\tau\in A_p$ is a norm function, then there exists a norm function $\rho\in A_r$ satisfying
\begin{enumerate}[(a)]
    \item\label{it:normfunctiona11bc23}  $[\rho]_r\lesssim_n C^{1-\frac{p}{r}}[\tau]_p^{\frac{p}{r}}$;
    \item\label{it:normfunctiona12bc23} $\rho(x,u)\leq\sigma^\ast(x,u)^{1-\frac{p}{r}}\tau(x,u)^{\frac{p}{r}}$ for all $u\in\F^n$ for a.e. $x\in\R^d$;
    \item\label{it:normfunctiona13bc23} $\rho^\ast(x,v)\leq\sigma(x,v)^{1-\frac{p}{r}}\tau^\ast(x,v)^{\frac{p}{r}}$ for all $v\in\F^n$ for a.e. $x\in\R^d$.
\end{enumerate}
\end{lemma}
\begin{proof}
Showing that $[\sigma]_1\lesssim _n C$ follows exactly as in the proof of Lemma~\ref{lem:interpolationnorm}. For the second assertion, we also follow the proof strategy of Lemma~\ref{lem:interpolationnorm}, this time defining
\[
\rho(x,u):=\inf\{t>0:t^{-1}u\in K(x)\}
\]
for
\[
K(x):=\mc{K}(\{u\in\F^n:\sigma^\ast(x,u)^{1-\frac{p}{r}}\tau(x,u)^{\frac{p}{r}}\leq 1\}).
\]
Showing that \ref{it:normfunctiona12bc23} and \ref{it:normfunctiona13bc23} hold is analogous to the proof of \ref{it:normfunctiona12} and \ref{it:normfunctiona13} in Lemma~\ref{lem:interpolationnorm}, and \ref{it:normfunctiona11bc23} follows from \cite[Remark~2.16,\, Remark~8.10]{BC23}. The assertion follows.
\end{proof}

\begin{proof}[Proof of Theorem~\ref{thm:bc3extrapolationrdf}]
It suffices to consider the case $p\leq r$, since the case $p\geq r$ is proven analogously by exchanging the roles of $p$ and $p'$, $W$ and $W^{-1}$, $r$ and $r'$, $f$ and $g$, and $\mc{V}$ and $\mc{V}^{-1}$ for the to be constructed weight $\mc{V}$. Note that this uses the fact that $[W^{-1}]_{p'}=[W]_p$ and $[\mc{V}^{-1}]_{r'}=[\mc{V}]_r$, which follows, e.g., from Theorem~\ref{thm:muckenhouptdef}. Moreover, as in the proof of Theorem~\ref{thm:rdfmainestimate}, it suffices to construct a norm function $\rho\in A_r$ instead of $\mc{V}\in A_r$ with the desired bounds.

First, since by \cite[Theorem~6.9]{BC23} we have $M^{\mc{K}}:L^p_W(\R^d;\mc{K})\to L^p_W(\R^d;\mc{K})$ with
\begin{equation}\label{eq:matrixbuckley}
\|M^{\mc{K}}\|_{L^p_W(\R^d;\mc{K})\to L^p_W(\R^d;\mc{K})}\lesssim_{n,d,p}[W]_p^{p'},
\end{equation}
we may apply Theorem~\ref{thm:rdfdef} to $T=M^{\mc{K}}$ and $\mb{X}=L^p_W(\R^d;\F^n)$ to obtain a mapping $\mc{R}:L^p_W(\R^d;\mc{K})\to L^p_W(\R^d;\mc{K})$ satisfying $\|\mc{R}\|_{L^p_W(\R^d;\mc{K})\to L^p_W(\R^d;\mc{K})}\leq 2$, and for all $F\in L^p_W(\R^d;\mc{K})$ we have
\begin{itemize}
    \item $F(x)\subseteq\mc{R}F(x)$ a.e.;
    \item $M^{\mc{K}}(\mc{R}F)(x)\subseteq 2\|M^{\mc{K}}\|_{L^p_W(\R^d;\mc{K})\to L^p_W(\R^d;\mc{K})}\mc{R}F(x)$ a.e.
\end{itemize}
By Lemma~\ref{lem:componentsaturation}, we can pick an invertible matrix-valued mappings $U$ satisfying
\[
\|U\overline{B}\|_{L^p_W(\R^d;\mc{K})}=1.
\]
Setting $F:=\mc{K}(f)+\varepsilon U\overline{B}$, it follows from Proposition~\ref{prop:sumofconvexmeasurable} and Proposition~\ref{prop:convexsetfequalsf} that
\begin{equation}\label{eq:bc3extrap1}
\|F\|_{L^p_W(\R^d;\mc{K})}\leq\|f\|_{L^p_W(\R^d;\F^n)}+\varepsilon,
\end{equation}
and, by Lemma~\ref{lem:interpolationnorm}, the norm function 
\[
\sigma(x,u):=\sup_{v\in \mc{R}F(x)}|u\cdot v|
\]
satisfies $\sigma\in A_1$ with $[\sigma]_1\lesssim_n \|M^{\mc{K}}\|_{L^p_W(\R^d;\mc{K})\to L^p_W(\R^d;\mc{K})}$. Further defining the norm function $\tau(x,u):=|W(x)u|$, per assumption and per definition we have $\tau\in A_p$ with $[\tau]_p=[W]_p$. Thus, we may apply Lemma~\ref{lem:interpolationnormbc3} to obtain a norm function $\rho\in A_r$ satisfying \ref{it:normfunctiona12bc23}, \ref{it:normfunctiona13bc23}, and, by \eqref{eq:matrixbuckley},
\begin{align*}
[\rho]_r&\lesssim_n \|M^{\mc{K}}\|_{L^p_W(\R^d;\mc{K})\to L^p_W(\R^d;\mc{K})}^{1-\frac{p}{r}}[\tau]_p^{\frac{p}{r}}\lesssim_{n,d,r,p}[W]_p^{p'(1-\frac{p}{r})}[\tau]_p^{\frac{p}{r}}\\
&=[W]_p^{p'(1-\frac{p}{r})+\frac{p}{r}}=[W]_p^{\frac{p'}{r'}},
\end{align*}
as desired. Now, for the next assertion, since $f(x)\in\mc{R}F(x)$ a.e., we have $\sigma(x,u)\geq|u \cdot f(x)|$ a.e., and, hence, the a.e. estimate
\[
\sigma^\ast(x,f(x))\leq 1.
\]
Thus, we have
\[
\rho(x,f(x))\leq \sigma^\ast(x,f(x))^{1-\frac{p}{r}}\tau(x,f(x))^{\frac{p}{r}}\leq |W(x)f(x)|^{\frac{p}{r}}
\]
a.e., implying that $f\in L^r_{\rho}(\R^d;\F^n)$ with
\begin{equation}\label{eq:bc3extrap2}
\|f\|_{L^r_{\rho}(\R^d;\F^n)}\leq\|f\|_{L^p_W(\R^d;\F^n)}^{\frac{p}{r}}.
\end{equation}
Next, by applying Theorem~\ref{thm:filippov} with
\[
\phi(x,u)=|u\cdot g(x)|,\quad h(x)=\sup_{u\in\mc{R}F(x)}|u\cdot g(x)|,
\]
we can find a $k\in S^0(\R^d;\mc{R}F)$ for which
\[
|k(x)\cdot g(x)|=\sigma(x,g(x))
\]
a.e. Then, since $\tau^\ast(x,v)=|W(x)^{-1}v|$ by \cite[Proposition~4.12]{BC23} and $|k\cdot g|=|Wk\cdot W^{-1}g|\leq|Wk||W^{-1}g|$ by the Cauchy-Schwarz inequality, we have
\begin{align*}
\rho^\ast(x,g(x))&\leq\sigma(x,g(x))^{1-\frac{p}{r}}\tau^\ast(x,g(x))^{\frac{p}{r}}=|k(x)\cdot g(x)|^{1-\frac{p}{r}}|W(x)^{-1}g(x)|^{\frac{p}{r}}\\
&\leq |W(x)k(x)|^{1-\frac{p}{r}}|W(x)^{-1}g(x)|
\end{align*}
a.e. Hence, by H\"older's inequality, the properties of $\mc{R}$, and \eqref{eq:bc3extrap1}, we have
\begin{align*}
\|g\|_{L^{r'}_{\rho^\ast}(\R^d;\F^n)}
&\leq\|k\|_{L^p_W(\R^d;\F^n)}^{1-\frac{p}{r}}\|g\|_{L^{p'}_{W^{-1}}(\R^d;\F^n)}
\leq\|\mc{R}F\|_{L^p_W(\R^d;\mc{K})}^{1-\frac{p}{r}}\|g\|_{L^{p'}_{W^{-1}}(\R^d;\F^n)}\\
&\lesssim \|F\|_{L^p_W(\R^d;\mc{K})}^{1-\frac{p}{r}}\|g\|_{L^{p'}_{W^{-1}}(\R^d;\F^n)}
\leq(\|f\|_{L^p_W(\R^d;\F^n)}+\varepsilon)^{1-\frac{p}{r}}\|g\|_{L^{p'}_{W^{-1}}(\R^d;\F^n)}.
\end{align*}
Combining this estimate with \eqref{eq:bc3extrap2}, we conclude that
\[
\|f\|_{L^r_\rho(\R^d;\F^n)}\|g\|_{L^{r'}_{\rho^\ast}(\R^d;\F^n)}\lesssim (\|f\|_{L^p_W(\R^d;\F^n)}+\varepsilon)\|g\|_{L^{p'}_{W^{-1}}(\R^d;\F^n)},
\]
proving the assertion.
\end{proof}

\begin{proof}[Proof of Theorem~\ref{thm:bc23extrapolation}]
Let $f\in V$ for which $Sf\in L^p_W(\R^d;\F^n)$, and let $g\in L^{p'}_{W^{-1}}(\R^d;\F^n)$. By Theorem~\ref{thm:bc3extrapolationrdf} we can pick a $\mc{V}\in A_r$ for which
\[
[\mc{V}]_r\leq C_{n,d,r,p}[W]^{\max\{\frac{p}{r},\frac{p'}{r'}\}}_p
\]
and
\[
\|Sf\|_{L^r_\mc{V}(\R^d;\F^n)}\|g\|_{L^{r'}_{\mc{V}^{-1}}(\R^d;\F^n)}\lesssim_n(\|Sf\|_{L^p_W(\R^d;\F^n)}+\varepsilon)(\|g\|_{L^{p'}_{W^{-1}}(\R^d;\F^N)}+\varepsilon).
\]
Thus,
\begin{align*}
\int_{\R^d}\!|Tf\cdot g|\,\mathrm{d}x&\leq \|Tf\|_{L^r_\mc{V}(\R^d;\F^n)}\|g\|_{L^{r'}_{\mc{V}^{-1}}(\R^d;\F^n)}
\leq\phi([\mc{V}]_r)\|Sf\|_{L^r_\mc{V}(\R^d;\F^n)}\|g\|_{L^{r'}_{\mc{V}^{-1}}(\R^d;\F^n)}\\
&\lesssim_n\phi(C_{n,d,r,p}[W]^{\max\{\frac{p}{r},\frac{p'}{r'}\}}_p)(\|Sf\|_{L^p_W(\R^d;\F^n)}+\varepsilon)(\|g\|_{L^{p'}_{W^{-1}}(\R^d;\F^N)}+\varepsilon).
\end{align*}
Taking a supremum over all $g\in L^{p'}_{W^{-1}}(\R^d;\F^n)$ with $\|g\|_{L^{p'}_{W^{-1}}(\R^d;\F^n)}=1$ and letting $\varepsilon\downarrow 0$ proves the result.
\end{proof}

\section*{Acknowledgements}
I wish to thank Rosemarie Bongers for being better at linear algebra than me, for sharing her knowledge on the Hausdorff distance, and for their insights into the geometry of convex sets. Moreover, I would like to thank her for their feedback on parts of the text, resulting in an overall improvement of the presentation of this work.

I also wish to thank Spyridon Kakaroumpas who provided a counterexample to the non-degeneracy of the weak-type spaces as I have defined them.

Finally, I wish to thank an anonymous referee for their suggestions and corrections, greatly improving the quality and presentation of this work.

\subsection*{Conflict of interest} None.

\subsection*{Availability of data} There is no data associated to this work.

\bibliography{bieb}
\bibliographystyle{alpha}
\end{document}